\documentclass{article}

\usepackage{amscd,amssymb,verbatim,graphicx,stmaryrd,amsthm,color,amsmath} 
\usepackage{hyperref}
\usepackage{diagrams}
\usepackage[margin=1.7in]{geometry}

\input xy    
\xyoption{curve}  
\xyoption{arc} 
\xyoption{frame} 
\xyoption{tips} 
\xyoption{arrow} 
\xyoption{color} 

\usepackage{mathpazo} 

\newtheorem{theorem}{Theorem}[section]  
\newtheorem{lemma}[theorem]{Lemma}  
\newtheorem{proposition}[theorem]{Proposition}

\newtheorem{example}[theorem]{Example}
\newtheorem{remark}[theorem]{Remark}

\newtheorem*{qconjecture}{Quilted Atiyah-Floer Conjecture} 
\newtheorem*{qcconjecture}{Quilted Atiyah-Floer Conjecture (Chain Level Version)} 

\newcommand{\End}{\mathrm{End}}
\newcommand{\sumdd}[2]{\displaystyle \sum_{#1}^{#2}}
\newcommand{\sumd}[1]{\displaystyle \sum_{#1}}
\newcommand{\limd}[1]{\displaystyle \lim_{#1}}
\newcommand{\intd}[1]{\displaystyle \int_{#1}}
\newcommand{\intdd}[2]{\displaystyle \int_{#1}^{#2}}
\newcommand{\fracd}[2]{\displaystyle \frac{#1}{#2}}
\newcommand{\bb}[1]{\mathbb{#1}}

\newcommand{\defeq}{\mathrel{\mathpalette{\vcenter{\hbox{$:$}}}=}}
\newcommand{\SU}{\mathrm{SU}}
\newcommand{\SO}{\mathrm{SO}}
\newcommand{\U}{\mathrm{U}}
\newcommand{\GL}{\mathrm{GL}}

\newcommand{\PU}{\mathrm{PU}}

\newcommand{\fl}{{\small \mathrm{flat}}}
\newcommand{\ba}{\mathrm{basic}}
\newcommand{\A}{{\mathcal{A}}}
\newcommand{\G}{{\mathcal{G}}}
\newcommand{\M}{{{M}}}
\newcommand{\Mtwo}{{\mathcal{M}}}
\newcommand{\La}{{{L}}}

\newcommand{\C}{{\mathcal{C}}}
\newcommand{\K}{{\mathcal{K}}}

\newcommand{\inst}{\mathrm{inst}}
\newcommand{\quilt}{\mathrm{symp}}
\newcommand{\CS}{{\mathcal{CS}}}
\newcommand{\YM}{{\mathcal{YM}}}
\newcommand{\proj}{{\mathrm{proj}}}
\newcommand{\eps}{\epsilon}

\newcommand{\Hom}{{\mathrm{Hom}}}
\newcommand{\D}{{\mathcal{D}}}

\newcommand{\afV}{V}
\newcommand{\afv}{v}
\newcommand{\afR}{R}
\newcommand{\afr}{r}
\newcommand{\afA}{A}
\newcommand{\afB}{B}
\newcommand{\afa}{a}
\newcommand{\afb}{b}
\newcommand{\afp}{p}
\newcommand{\afalpha}{\alpha}

\newcommand{\afphi}{\phi}
\newcommand{\afpsi}{\psi}
\newcommand{\afmu}{\mu}
\newcommand{\afnu}{\nu}
\newcommand{\afU}{U}
\newcommand{\afu}{u}

\title{Higher-rank instanton cohomology and \\ the quilted Atiyah-Floer conjecture}

\author{David L. Duncan}

\date{}

\begin{document}

\maketitle

\begin{abstract}
	Given a closed, connected, oriented 3-manifold with positive first Betti number, one can define an instanton Floer group as well as a quilted Lagrangian Floer group. The quilted Atiyah-Floer conjecture states that these cohomology groups are isomorphic. We initiate a program for proving this conjecture.
\end{abstract}

\tableofcontents

\section{Introduction}

This paper is the first in a series \cite{DunIndex} \cite{DunComp} that prove various aspects of the quilted Atiyah-Floer conjecture. The present paper is dedicated entirely to the underlying Floer theory and differential geometry of the project.

Section \ref{OverviewOfTheQuiltedAtiyahFloerConjecture} begins with a brief history and overview of the conjecture. We also introduce our notation and conventions. The section ends with a description of the quilted Floer group associated to a 3-manifold. This group was first introduced by Wehrheim and Woodward in \cite{WWfloer}, and is defined via a 2+1 field theoretic scheme that relates symplectic and gauge theoretic 3-manifold invariants. We point out that in \cite{WWfloer} the authors define their quilted invariants for 3-manifolds via gauge theory on suitable $\PU(r)$-bundles (one usually considers the case $\PU(2) = \SO(3)$). We adopt their approach by working with the higher rank $\PU(r)$-bundles throughout this paper.

In Section \ref{InstantonFloerCohomology} we discuss Floer's instanton homology \cite{Fl1} \cite{Flinst}. We note that, as in the symplectic theory, the standard references (e.g., Donaldson's book \cite{Donfloer}, as well as Floer's original papers) focus almost entirely on the Lie groups $\SU(2)$ and $\SO(3)$. However, it is possible to define instanton homology for other Lie groups as well. To maintain the analogy with the Woodward-Wehrheim theory, we focus attention on $\PU(r)$ for $r \geq 2$. A full treatment of higher-rank instanton Floer theory does not seem to appear in the literature, so we have spent some time developing the theory in this case, proving several folklore results that are well-known in the case $r = 2$. We note also that Kronheimer has defined higher-rank Donaldson invariants \cite{Kron}. These can be viewed as the 4-dimensional version of the higher-rank instanton Floer theory discussed here.

In Section \ref{TheQuiltedAtiyahFloerConjecture} we make a precise statement of the quilted Atiyah-Floer conjecture, as well as a chain level version of the conjecture. We then describe our overall approach, and prove several chain level statements. This includes a detailed discussion of how to obtain transversality \emph{simultaneously} in both Floer theories.

\medskip

\noindent {\bfseries Acknowledgments:} The author is grateful to his thesis advisor Chris Woodward for his insight and valuable suggestions. In addition, the author benefited greatly from discussions with Katrin Wehrheim, via Chris Woodward, of her unpublished work on the Atiyah-Floer conjecture. This unpublished work analyzes various bubbles that appear in the limit of instantons with Lagrangian boundary conditions with degenerating metrics, and it outlines the remaining problems in such an approach. The present paper takes a different approach that avoids instantons with Lagrangian boundary conditions.

The author would also like to thank Penny Smith, Dietmar Salamon, Fabian Ziltener, Peter Ozsv\'{a}th, and Tom Parker for enlightening discussions. Finally, thanks are certainly due to the author's wife Claudette D'Souza for her help with early versions of the paper, and for her support along the way. This work was partially supported by NSF grants DMS 0904358 and DMS 1207194.

\section{Background on the quilted Atiyah-Floer conjecture}\label{OverviewOfTheQuiltedAtiyahFloerConjecture}

Let $Y$ be a closed, connected, oriented 3-manifold, and fix an $\SO(3)$-bundle $Q \rightarrow Y$ such that $w_2(Q)$ pairs non-trivially with an element of $H_2(Y, \bb{Z})$ (so $Y$ is necessarily \emph{not} a homology 3-sphere). In \cite{Flinst} \cite{Fl1}, Floer showed that this data defines a chain complex $(CF_\inst(Q), \partial_\inst)$ where $CF_\inst(Q)$ is the abelian group freely generated by the gauge equivalence classes of flat connections on $Q$. (Strictly speaking, one needs to suitably perturb the defining equations in order to obtain $\partial_\inst^2 = 0$.) Then the associated homology group $HF_\inst(Q)$, called \emph{instanton Floer cohomology}, is well-defined and depends only on $Q \rightarrow Y$.

In \cite{DS} \cite{DS2} \cite{DS3} \cite{DSerr}, Dostoglou and Salamon considered the following special case. Let $\Sigma$ be a closed, connected, oriented surface, $P \rightarrow \Sigma$ a non-trivial $\SO(3)$-bundle, and $\Phi: P \rightarrow P$ a bundle map covering an orientation-preserving diffeomorphism $\varphi: \Sigma \rightarrow \Sigma$. Then Dostoglou-Salamon applied Floer's construction to the mapping tori $Y = (I \times \Sigma) / \varphi$ and $Q = (I \times P)  / \Phi$. This is particularly interesting due to the following symplectic interpretation: Let $\M(\Sigma)$ denote the space of flat connections on $P$ modulo the action of the identity component of the gauge group. This is naturally a smooth finite-dimensional symplectic manifold, and pullback by $\Phi$ induces a symplectomorphism $\Phi^*: \M(\Sigma) \rightarrow \M(\Sigma)$. Then the graph $\mathrm{Gr}(\Phi^*) \subset \M(\Sigma)^- \times \M(\Sigma)$ is a Lagrangian submanifold, where the superscript in $\M(\Sigma)^- $ means we have replaced the symplectic form by its negative. Similarly, the diagonal $\Delta \subset \M(\Sigma)^- \times \M(\Sigma)$ is also Lagrangian, and one can use this data to define a \emph{Lagrangian Floer group} $HF_\quilt(\mathrm{Gr}(\Phi^*) , \Delta )$. This can be thought of as the homology of a chain complex $\left(CF_\quilt(\mathrm{Gr}(\Phi^*) , \Delta ), \partial_\quilt\right)$ generated by the intersection points $\mathrm{Gr}(\Phi^*) \cap \Delta$. On the other hand, the elements of $\mathrm{Gr}(\Phi^*) \cap \Delta$ can be naturally identified with the gauge equivalence classes of flat connections on $Q$. That is, $HF_\quilt(\mathrm{Gr}(\Phi^*), \Delta)$ and $HF_\inst(Q) $ are both well-defined homology groups that measure the flat connections on $Q$. Following an outline suggested by Atiyah and Floer, Dostoglou and Salamon were able to show that these Floer groups are naturally isomorphic 

$$HF_\inst(Q) \cong HF_\quilt(\mathrm{Gr}(\Phi^*), \Delta).$$ 
They prove this by demonstrating a chain-level isomorphism in the spirit of Atiyah's idea to `stretch the neck' \cite{A}.

\medskip

The problem we consider in the present paper is a generalization of the Dostoglou-Salamon set-up to 3-manifolds $Y$ with $b_1(Y) > 0$. More explicitly, we assume $Y$ is a closed, connected, oriented 3-manifold equipped with a Morse function $f: Y \rightarrow S^1$ that is not homotopically trivial. We assume further that $f$ has connected fibers. Such functions were considered by Lekili \cite{Lek1}; following him we call these {\bfseries broken circle fibrations}. It follows from work of Gay and Kirby \cite{GK} that $Y$ admits a broken circle fibration if and only if $Y$ has positive first Betti number. Moreover, each critical point of $f$ has index 1 or 2 (the index is defined with respect to the canonical orientation on the circle). The invariants in which we are interested will only depend on the homotopy class of $f$. In particular, by homotoping $f$ we can assume the critical points of $f$ have distinct critical values, and we let $N$ denote the number of critical points. If $N = 0$, then by performing a suitable homotopy we can replace $f$ by a Morse function with two critical points; this is the Morse-theoretic version of destabilization. We may therefore assume $N > 0$. 

Following \cite{WWfloer} and \cite{GWW}, we will use $f$ to decompose $Y$ into a union of elementary cobordisms (a so-called \emph{Cerf decomposition}): Identify $S^1 \cong \mathbb{R} / C \mathbb{Z}$ for some $C > 0$. Find regular values $r_i \in S^1$ such that, for each $0 \leq i \leq N-1$, there is exactly one critical value $c_{i(i+1)}$ with $r_i + \delta < c_{i(i+1)} < r_{i+1} - \delta$, for some fixed $\delta > 0$. Here and below we work with $i$ modulo $N$. We may assume the circumference $C$ is large enough to take $\delta = 1/2$. Define

$$\Sigma_i \defeq f^{-1}(r_i- 1/2), \indent Y_{i(i+1)}\defeq f^{-1}\left(\left[r_i+1/2, r_{i+1}-1/2\right] \right),$$
which are compact, connected, oriented surfaces and cobordisms, respectively. 

	\begin{figure}[h]
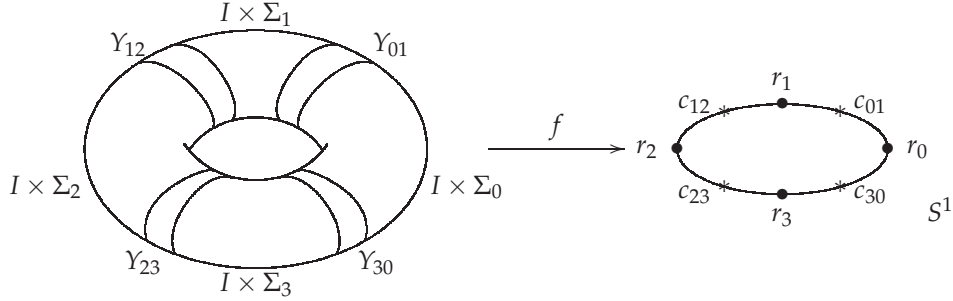
\label{figure6}
	\centerline{\xy 
		(-40,10)*={\xy
		/r.15pc/:,
	 (0,0)*\ellipse(36,25){-};
	 (-15,1)*{}="t1";
(15,1)*{}="t2";
"t1";"t2" **\crv{(-8, -9)& (8, -9)}; 
 (-14,-.2)*{}="t1";
(14,-.2)*{}="t2";
"t1";"t2" **\crv{(-8, 9)& (8, 9)}; 
(9,4.5)*{}="At1";
(24.4,18.2)*{}="Bt1";
"At1";"Bt1" **\crv{(7.5,9)& (18,20)};  
(5,6)*{}="At2";
(17,22)*{}="Bt2";
"At2";"Bt2" **\crv{(3,9)& (9,20)& (14,22)};  
(-9,4.5)*{}="At3";
(-24.4,18.2)*{}="Bt3";
"At3";"Bt3" **\crv{(-7.5,9)& (-18,20)};  
(-5,6)*{}="Atl2";
(-17,22)*{}="Btl2";
"Atl2";"Btl2" **\crv{(-3,9)& (-9,20)& (-14,22)};  
(9,-4.5)*{}="Ab1";
(23,-19)*{}="Bb1";
"Ab1";"Bb1" **\crv{(10,-3.5)& (25.5,-9)};  
(5,-6)*{}="Ab2";
(17,-22)*{}="Bb2";
"Ab2";"Bb2" **\crv{(11,-4)& (20,-16)};  
(-9,-4.5)*{}="Ab14";
(-23,-19)*{}="Bb14";
"Ab14";"Bb14" **\crv{(-10,-3.5)& (-25.5,-9)};  
(-5,-6)*{}="Ab3";
(-17,-22)*{}="Bb3";
"Ab3";"Bb3" **\crv{(-11,-4)& (-20,-16)};  
	 \endxy};
(15,5)*\ellipse(14,6){-};
(-12,5)*{I \times \Sigma_0};
(-40,28)*{I \times \Sigma_1};
(-68,5)*{I \times \Sigma_2};
(-40,-8)*{I \times \Sigma_3};
(-22,24)*{Y_{01}};
(-57,24)*{Y_{12}};
(-55,-5)*{Y_{23}};
(-24,-5)*{Y_{30}};
(44,10)*{\bullet};
(48,10)*{r_0};
(30,16)*{\bullet};
(30,19)*{r_1};
(16,10)*{\bullet};
(12,10)*{r_2};
(30,4)*{\bullet};
(30,1)*{r_3};
(37.7,15)*{*};
(41.7,16)*{c_{01}};
(22.3,15)*{*};
(18.3,16)*{c_{12}};
(22.3,5)*{*};
(18.3,4)*{c_{23}};
(37.7,5)*{*};
(41.7,4)*{c_{30}};
(51,2)*{S^1};
(-9,10)*{}="c";
(9,10)*{}="e";
{\ar@{->} "c";"e"};
(0,13)*{f};
	 \endxy}
	 \caption{An illustration of a broken circle fibration.}
	 \end{figure}

Fix a metric $g$ on $Y$. We refer to $g$, or its restriction to any submanifold of $Y$, as the {\bfseries fixed metric}. Note that there are no critical values between $r_i-1/2$ and $r_{i}+1/2$, so the normalized gradient vector field $V \defeq \nabla f / \left| \nabla f\right|$ is well-defined on $f^{-1}(\left[r_i-1/2,r_{i}+1/2\right])$. The time-1 gradient flow of $V$ provides an identification 

\begin{equation}
f^{-1}(\left[r_{i}-1/2, r_{i}+1/2\right]) \cong I \times \Sigma_i,
\label{id1}
\end{equation}
where we have set $I \defeq \left[0, 1\right]$. This also provides an identification of $f^{-1}(t)$ with $\Sigma_i$ for $t \in \left[r_i-1/2,r_{i}+1/2\right]$. So the function $f$ together with the metric $g$ allow us to view $Y$ as a composition of cobordisms:

\begin{equation}\label{decopmofYbullet}
Y_{01} \cup_{\Sigma_1} (I \times \Sigma_1) \cup_{\Sigma_1} Y_{12} \cup_{\Sigma_2} \ldots \cup_{\Sigma_{N-1}}  Y_{(N-1)0} \cup_{\Sigma_0} (I \times \Sigma_0) \cup_{\Sigma_0}
\end{equation}
Note that this is \emph{cyclic} in the sense that the cobordism $I \times \Sigma_0$ on the right is glued to the cobordism $Y_{01}$ on the left, reflecting the fact that $f$ maps to the circle. By construction, each $Y_{i(i+1)}$ is an {\bfseries{elementary cobordism}}; that is, it admits an $I$-valued Morse function that (i) preserves the cobordism structure between $Y_{i(i+1)}$ and $I$, and (ii) has at most one critical point (necessarily in the interior of $Y_{i(i+1)}$). Since each critical point has index either 1 or 2, it follows from standard Morse theory considerations \cite[Theorem 3.14]{Mil} that the genus of $\Sigma_i$ differs from that of $\Sigma_{i+1}$ by one. In particular, the number of critical points $N$ is \emph{even}.

It is possible to choose an $\SO(3)$-bundle $Q \rightarrow Y$ so its restriction to each $\Sigma_i$ is non-trivial. Then $HF_\inst(Q)$ is well-defined. On the symplectic side, define $\M(\Sigma_i)$ for each $\Sigma_i$ as above. This choice of $Q$ ensures that each $\M(\Sigma_i)$ is smooth. Restricting to each of the two boundary components of the $Y_{i(i+1)}$ defines two smooth Lagrangian submanifolds $L^{(0)}, L^{(1)} \subset \M(\Sigma_0) \times \ldots \times \M(\Sigma_{N-1})$ for which $HF_\quilt(L^{(0)}, L^{(1)})$ is well-defined; see Section \ref{QuiltedFloerCohomologyForBrokenCircleFibrations}. Then the {\bfseries quilted Atiyah-Floer conjecture} states that there is a natural isomorphism $HF_\inst(Q) \cong HF_\quilt (L^{(0)}, L^{(1)})$.

It follows from the definitions that there is a natural isomorphism of abelian groups 
 
 \begin{equation}\label{Overviewisoofgenerators}
 CF_\inst(Q) \cong CF_\quilt(L^{(0)}, L^{(1)}),
 \end{equation}
 where these are (suitably defined) chain groups inducing the Floer homology groups $HF_\inst(Q)$, $HF_\quilt (L^{(0)}, L^{(1)})$; see Theorem \ref{generators} below. Moreover, these chain groups each admit natural relative gradings. In \cite{DunIndex} it is shown that the isomorphism (\ref{Overviewisoofgenerators}) preserves these gradings. The quilted Atiyah-Floer conjecture would follow if one could show that this isomorphism intertwines the boundary operators. One way to show this is to prove that there is a one-to-one correspondence between instantons and holomorphic curves, at least for suitably chosen data (metric, almost complex structure and perturbations). The main result of \cite{DunComp} implies that, when this data is chosen well, every instanton trajectory counted by the boundary operator $\partial_\inst$ is close to some holomorphic curve trajectory counted by the boundary operator $\partial_\quilt$. (Proving the converse statement, and hence the conjecture, is still a work in progress.) In the present paper we describe how the aforementioned data can be chosen so that both chain level groups are comparable.

\begin{remark}
(a) Fukaya \cite{Fuk} \cite{Fuk2} describes a program similar to the one presented here. Though quite similar, there is one striking difference between his approach and ours: Fukaya deals with a fixed smooth (but not definite) metric, whereas we deal with a sequence of metrics that degenerate to a singular metric, and these metrics \emph{are not smooth}. 

\medskip

(b) Recently M. Lipyanskiy has developed compactness results for quilts with patches consisting both of instantons and holomorphic curves. The motivation is that these quilts may define a chain map that interpolates between the trajectories defining each of the two Floer theories. 

\medskip

(c) The quilted Atiyah-Floer conjecture is a variant of the older \emph{Atiyah-Floer conjecture}; see \cite{A}. In this latter statement, one considers a homology 3-sphere $Y$, and works with $\SU(2)$, rather than $\SO(3)$. Then the Atiyah-Floer conjecture posits an isomorphism between an instanton Floer homology and a symplectic Floer homology associated to this data. Unfortunately, at the time of writing, the relevant symplectic Floer homology is not well-defined due to issues arising from reducible connections; however, see \cite{SW}. The homology theories relevant for the quilted version are well-posed essentially because the bundles are chosen so that one can avoid the use of reducible connections.

\medskip

(d) In \cite{A}, Atiyah outlined a metric degeneration approach to the Atiyah-Floer conjecture. The degenerating metric $g_\eps$ we define in the next section is, up to conformal equivalence, a direct analogue of Atiyah's metric for the quilted Atiyah-Floer conjecture.
\end{remark}

\subsection{$\eps$-dependent smooth structures}\label{EpsDependentSmoothStructures}

As a matter of notational convenience, we set

$$\Sigma_\bullet \defeq \bigsqcup_i \Sigma_i, \indent Y_\bullet \defeq Y \backslash \left(I \times \Sigma_\bullet\right).$$
We refer to the connected components of the boundary $\partial Y_\bullet = \left\{0, 1\right\} \times \Sigma_\bullet$ as the {\bfseries seams}, and we use the letter $t$ to denote the coordinate variable on the interval $I$. In particular, $dt \in \Omega^1(I \times \Sigma)$ is identified with 

$$df / \left|df \right| \in \Omega^1\left(\cup_i f^{-1}\left(\left[r_i -1/2, r_i+1/2\right]\right)\right)$$ 
under the identification (\ref{id1}). 

Over $I \times \Sigma_\bullet$ the metric $g$ has the form

$$dt^2 + g_{\Sigma}$$
where $g_{\Sigma}$ is a path of metrics on $\Sigma_\bullet$. To simplify the exposition, we assume that $g$ has been chosen so that $g_\Sigma$ is a \emph{constant} path, which can always be achieved. For $\epsilon > 0$ define a new metric $g_\eps$ by

$$g_\epsilon \defeq \left\{\begin{array}{ll} 
											dt^2 + \epsilon^2 g_{\Sigma} & \textrm{on $I \times \Sigma_\bullet$}\\
											\epsilon^2 g & \textrm{on $Y_\bullet$}\\
										\end{array}\right.$$
We will be interested in taking the limit as $\epsilon$ approaches $0$. See Figure \ref{figure5}.

\begin{figure}[h]
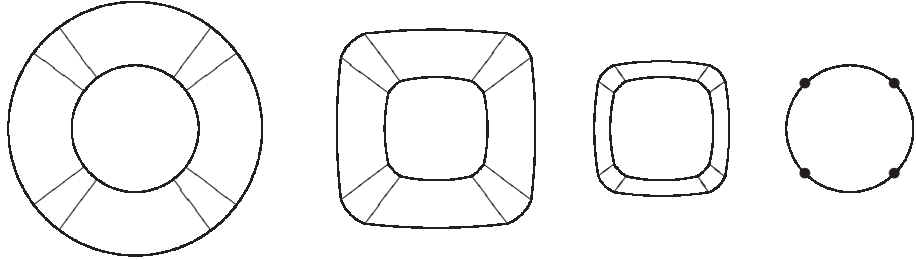

\centerline{ \xy
	(0,-15)*{};
	 (12,-20)*={\xy
		/r.2pc/:,	
		(0,0)*\ellipse(10,10){-};
		(0,0)*\ellipse(20,20){-};
		(16,12)*{}="A1";
		(-16,12)*{}="B1";
		(-16,-12)*{}="C1";
		(16,-12)*{}="D1";
		(12,16)*{}="A2";
		(-12,16)*{}="B2";
		(-12,-16)*{}="C2";
		(12,-16)*{}="D2";
		(8,6)*{}="a1";
		(-8,6)*{}="b1";
		(-8,-6)*{}="c1";
		(8,-6)*{}="d1";
		(6,8)*{}="a2";
		(-6,8)*{}="b2";
		(-6,-8)*{}="c2";
		(6,-8)*{}="d2";
		"a1";"A1" **\dir{-};
		"b1";"B1" **\dir{-};
		"c1";"C1" **\dir{-};
		"d1";"D1" **\dir{-};
		"a2";"A2" **\dir{-};
		"b2";"B2" **\dir{-};
		"c2";"C2" **\dir{-};
		"d2";"D2" **\dir{-};
	 \endxy};
	 (52,-20)*={\xy
		/r.5pc/:,	
			(4.5,6)*{}="A1";
(6,4.5)*{}="A2";
(-4.5,6)*{}="B1";
(-6,4.5)*{}="B2";
(-4.5,-6)*{}="C1";
(-6,-4.5)*{}="C2";
(4.5,-6)*{}="D1";
(6,-4.5)*{}="D2";
(3,2.25)*{}="a2";
(2.25,3)*{}="a1";
(-3,2.25)*{}="b2";
(-2.25,3)*{}="b1";
(-3,-2.25)*{}="c2";
(-2.25,-3)*{}="c1";
(3,-2.25)*{}="d2";
(2.25,-3)*{}="d1";
"A1";"A2" **\crv{(5.55,5.55)};
"B1";"B2" **\crv{(-5.55,5.55)};
"C1";"C2" **\crv{(-5.55,-5.55)};
"D1";"D2" **\crv{(5.55,-5.55)};
"a1";"a2" **\crv{(2.7,2.7)};
"b1";"b2" **\crv{(-2.7,2.7)};
"c1";"c2" **\crv{(-2.7,-2.7)};
"d1";"d2" **\crv{(2.7,-2.7)};
"A1";"a1" **\dir{-};
"A2";"a2" **\dir{-};
"B1";"b1" **\dir{-};
"B2";"b2" **\dir{-};
"C1";"c1" **\dir{-};
"C2";"c2" **\dir{-};
"D1";"d1" **\dir{-};
"D2";"d2" **\dir{-};
"A1";"B1" **\crv{(0,6.5)};
"B2";"C2" **\crv{(-6.5,0)};
"C1";"D1" **\crv{(0,-6.5)};
"D2";"A2" **\crv{(6.5,0)};
"a1";"b1" **\crv{(0,3.5)};
"b2";"c2" **\crv{(-3.5,0)};
"c1";"d1" **\crv{(0,-3.5)};
"d2";"a2" **\crv{(3.5,0)};
	 \endxy};
 (82,-20)*={\xy
		/r.5pc/:,
		(3,4)*{}="A1";
(4,3)*{}="A2";
(-3,4)*{}="B1";
(-4,3)*{}="B2";
(-3,-4)*{}="C1";
(-4,-3)*{}="C2";
(3,-4)*{}="D1";
(4,-3)*{}="D2";
(3,2.25)*{}="a2";
(2.25,3)*{}="a1";
(-3,2.25)*{}="b2";
(-2.25,3)*{}="b1";
(-3,-2.25)*{}="c2";
(-2.25,-3)*{}="c1";
(3,-2.25)*{}="d2";
(2.25,-3)*{}="d1";
"A1";"A2" **\crv{(3.7,3.7)};
"B1";"B2" **\crv{(-3.7,3.7)};
"C1";"C2" **\crv{(-3.7,-3.7)};
"D1";"D2" **\crv{(3.7,-3.7)};
"a1";"a2" **\crv{(2.7,2.7)};
"b1";"b2" **\crv{(-2.7,2.7)};
"c1";"c2" **\crv{(-2.7,-2.7)};
"d1";"d2" **\crv{(2.7,-2.7)};
"A1";"a1" **\dir{-};
"A2";"a2" **\dir{-};
"B1";"b1" **\dir{-};
"B2";"b2" **\dir{-};
"C1";"c1" **\dir{-};
"C2";"c2" **\dir{-};
"D1";"d1" **\dir{-};
"D2";"d2" **\dir{-};
"A1";"B1" **\crv{(0,4.5)};
"B2";"C2" **\crv{(-4.5,0)};
"C1";"D1" **\crv{(0,-4.5)};
"D2";"A2" **\crv{(4.5,0)};
"a1";"b1" **\crv{(0,3.5)};
"b2";"c2" **\crv{(-3.5,0)};
"c1";"d1" **\crv{(0,-3.5)};
"d2";"a2" **\crv{(3.5,0)};
	 \endxy};
	  (107,-20)*={\xy
		/r.2pc/:,	
		(0,0)*\ellipse(10,10){-};
		(7.071,7.071)*{\bullet}="A";
		(-7.071,7.071)*{\bullet}="B";
		(-7.071,-7.071)*{\bullet}="C";
		(7.071,-7.071)*{\bullet}="D";
		\endxy};
\endxy}~\\
~\\
~\\
~\\
~\\
\caption{Above are four illustrations of the manifold $Y$; in each, $Y$ is viewed from the top. Moving from left to right, the metric on $Y$ is being deformed in such a way that the volume of the $\Sigma_i$ and the $Y_{i(i+1)}$ are going to zero. However, the volume in the $I$-direction (the `neck') is remaining fixed. In the picture on the far right, the $Y_{i(i+1)}$ have collapsed entirely to the critical points of $f$, which are represented by dots.}
\label{figure5}
\end{figure}

Let ${\mathcal{S}}_1$ denote the smooth structure on $Y$; that is, the smooth structure in which $g$ and $f$ are smooth. We call this the {\bfseries standard smooth structure}. It is important to note that when $\epsilon \neq 1$, the metric $g_\epsilon$ is \emph{not} smooth with the standard smooth structure. For example, take $V = \nabla f / \left| \nabla f\right|$, where the norm and gradient are taken with respect to $g = g_1$. Then away from the critical points of $f$, $V$ is smooth on $(Y, {\mathcal{S}}_1)$, but

\begin{equation}
g_\epsilon(V, V) = \left\{\begin{array}{ll} 
											1 & \textrm{on $I \times \Sigma_\bullet$}\\
											\epsilon^2  & \textrm{on $Y_\bullet$}\backslash \left\{\textrm{critical points} \right\}\\
										\end{array}\right.
\label{shear}
\end{equation}
is not even continuous, so $g_\epsilon$ cannot be continuous on $(Y, {\mathcal{S}_1})$.

\begin{figure}[h]
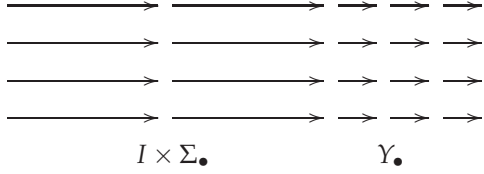

\centerline{
\xy
(-44,0)*{}="E1";
(-24,0)*{}="C1";
{\ar@{->} "E1";"C1"};
(-44,5)*{}="E2";
(-24,5)*{}="C2";
{\ar@{->} "E2";"C2"};
(-44,10)*{}="E3";
(-24,10)*{}="C3";
{\ar@{->} "E3";"C3"};
(-44,15)*{}="E4";
(-24,15)*{}="C4";
{\ar@{->} "E4";"C4"};
(-22,0)*{}="E12";
(-2,0)*{}="C12";
{\ar@{->} "E12";"C12"};
(-22,5)*{}="E22";
(-2,5)*{}="C22";
{\ar@{->} "E22";"C22"};
(-22,10)*{}="E32";
(-2,10)*{}="C32";
{\ar@{->} "E32";"C32"};
(-22,15)*{}="E42";
(-2,15)*{}="C42";
{\ar@{->} "E42";"C42"};
(0,0)*{}="e1";
(5,0)*{}="c1";
{\ar@{->} "e1";"c1"};
(0,5)*{}="e2";
(5,5)*{}="c2";
{\ar@{->} "e2";"c2"};
(0,10)*{}="e3";
(5,10)*{}="c3";
{\ar@{->} "e3";"c3"};
(0,15)*{}="e4";
(5,15)*{}="c4";
{\ar@{->} "e4";"c4"};
(7,0)*{}="e12";
(12,0)*{}="c12";
{\ar@{->} "e12";"c12"};
(7,5)*{}="e22";
(12,5)*{}="c22";
{\ar@{->} "e22";"c22"};
(7,10)*{}="e32";
(12,10)*{}="c32";
{\ar@{->} "e32";"c32"};
(7,15)*{}="e42";
(12,15)*{}="c42";
{\ar@{->} "e42";"c42"};
(14,0)*{}="e13";
(19,0)*{}="c13";
{\ar@{->} "e13";"c13"};
(14,5)*{}="e23";
(19,5)*{}="c23";
{\ar@{->} "e23";"c23"};
(14,10)*{}="e33";
(19,10)*{}="c33";
{\ar@{->} "e33";"c33"};
(14,15)*{}="e43";
(19,15)*{}="c43";
{\ar@{->} "e43";"c43"};
(-22,-5)*{ I \times \Sigma_\bullet};
(7,-5)*{ Y_\bullet};
\endxy}
\caption{The vector field $V$ is not continuous with respect to the topology on $TY$ defined using the standard smooth structure. However, it is continuous with respect to the topology on $TY$ defined using the $\epsilon$-\emph{dependent} smooth structure.}
\label{figure7}
\end{figure}

However, there is a \emph{different} smooth structure ${\mathcal{S}}_\epsilon$ in which $g_\epsilon$ is smooth. Moreover, $(Y, {\mathcal{S}}_\epsilon)$ is \emph{diffeomorphic} to $(Y, {\mathcal{S}}_1)$. This can be seen as follows: View $Y$ as the \emph{topological} manifold in (\ref{decopmofYbullet}). Following Milnor \cite[Theorem 1.4]{Mil}, any choice of collar neighborhoods of the seams determines a smooth structure on $Y$, and any two choices are given by isotopic data. These isotopies determine a diffeomorphism between the smooth structures. In this language, the smooth structure ${\mathcal{S}}_1$ arises by choosing collar neighborhoods of the seams determined by the time-$\delta$ gradient flow of $f$, and then using the identity to glue these neighborhoods on the overlap. Here $\delta > 0$ is small, but fixed. On the other hand, the smooth structure ${\mathcal{S}}_\epsilon$ arises by taking the time-$\delta$ gradient flow on the $Y_{i(i+1)}$ side of the seam $\left\{1\right\} \times \Sigma_i$, but the time-$\delta \epsilon$ gradient flow on the $\left[0, 1\right] \times \Sigma_i$ side of the seam, and then gluing using the gluing map $(t, \sigma) \mapsto (\epsilon t, \sigma).$ 

By the discussion of the previous paragraph, there is a diffeomorphism 

\begin{equation}
F_\epsilon: (Y, {\mathcal{S}}_1) \longrightarrow (Y, {\mathcal{S}}_\epsilon).
\label{milnorsdiffeo}
\end{equation}
In fact, there is a canonical choice of $F_\eps$ given by taking the obvious straight line homotopy between the gluing maps described above. When the specific smooth structure on $Y$ is relevant, we will write $Y^\epsilon$ for $(Y, {\mathcal{S}}_\epsilon)$. When a function, section of a bundle, connection, etc. is smooth on $Y^\epsilon$, we will say that it is {\bfseries $\epsilon$-smooth}. We will write $f^\eps$ for the pullback of $f: Y \rightarrow S^1$ under $F_\eps^{-1}$.

Observe that $g_\eps$ only fails to be smooth on $Y^1$ at the seams $\left\{0, 1\right\} \times \Sigma_i$. Furthermore, even at the seams, the metric $g_\epsilon$ is smooth in directions \emph{parallel} to the seams. So the discontinuity illustrated in Figure \ref{figure7} is the only thing that goes wrong.

\begin{remark}
Fix $1 \leq p \leq \infty$ and let $W^{1,p}(Y^1)$ denote the Sobolev space associated to the standard smooth structure on $Y = Y^1$. By passing to local coordinates, it is straightforward to show that every $\epsilon$-smooth \emph{function} $h$ on $Y^\epsilon$ is of Sobolev class $W^{1,p}(Y^1)$. That is, the identity map on $Y$ determines a pullback map of the form $\mathrm{Id}^*: C^\infty(Y^\eps) \longrightarrow W^{1,p}(Y^1)$. This can be seen as follows: The underlying topologies on $Y^1$ and $Y^\eps$ are identical, so the $\eps$-smooth function $h$ is continuous on $Y^1$. Moreover, on the complement of the seams, $h$ is 1-smooth with bounded derivative. This implies $\mathrm{Id}^*h$ is of Sobolev class $W^{1,p}$. 

However, in general, an $\epsilon$-smooth \emph{tensor} will only be in $L^p(Y^1)$ with respect to the standard smooth structure. Indeed, any tensor that is non-zero in directions transverse to the seam will necessarily have a jump discontinuity as in Figure \ref{figure7}. So taking a derivative transverse to the seam will introduce a delta function. This applies to connections as well. 
\end{remark}

\subsection{Gauge theory}\label{GaugeTheory}

		Let $X$ be an oriented Riemannian $n$-manifold, possibly with boundary. Given a fiber bundle $E \rightarrow X$, we denote the space of smooth sections by $\Gamma(E)$. Now suppose $E$ is a vector bundle equipped with a metric $\langle \cdot, \cdot \rangle$. Then we write

$$\Omega^\bullet(X, E) \defeq \bigoplus_k \Omega^k(X, E), \indent \Omega^k(X, E) \defeq \Gamma(\Lambda^k T^*X \otimes E),$$
for the spaces of smooth $E$-valued forms on $X$. Given $\mu \in \Omega^j(X, E) , \nu \in \Omega^{k}(X, E)$ we can combine the inner product with the wedge to define a new section $\langle \mu \wedge \nu \rangle \in \Omega^{j+k}(X)$. When $X$ is compact, this can be integrated to obtain a non-degenerate bilinear pairing on forms of dual degree:
			
			\begin{equation}
			\Omega^k(X, E) \otimes \Omega^{n-k}(X, E) \longrightarrow \bb{R}, \indent \indent \indent \mu\otimes \nu  \longmapsto  \intd{X} \langle \mu \wedge \nu \rangle.
			\label{pairing}
			\end{equation}

			The metric on $X$ induces a Hodge star $*: \Omega^k(M, E) \rightarrow \Omega^{n-k}(M, E)$, and this satisfies $** = (-1)^{k(n-k)}$ on $k$-forms. Sticking $*$ in the second slot of (\ref{pairing}) defines an $L^2$-inner product on the vector space $\Omega^k(X, E)$. More generally, for $p \geq 1$ we can define the $L^p$-norm on $\Omega^k(M, E) $ by
			
			$$\Vert \mu \Vert_{L^p}^p \defeq  \intd{X} \vert \mu \vert^p \: d \mathrm{vol},$$
			where $\vert \mu \vert \defeq \left(*\langle \mu \wedge * \mu \rangle \right)^{1/2}$. Similarly, fixing a metric connection on $E$, we can define the Sobolev norms $\Vert \cdot \Vert_{W^{k, p}}$ on $\Omega^\bullet(X, E)$ by the usual formula. We denote by $W^{k, p}(E)$ the closure, with respect to $\Vert \cdot \Vert_{W^{k, p}}$, of the space of compactly supported smooth $E$-valued forms on $X$. When the bundle $E$ is clear from context, we will write $W^{k, p}$ for $W^{k,p}(E)$. In situations where it is particularly important to emphasize the underlying base manifold, we will write $W^{k,p}(X)$ instead of $W^{k,p}(E)$.

			\medskip			
			
			Let $G$ be a compact Lie group with Lie algebra $\frak{g}$, and let $\pi: P \rightarrow X$ be a principal $G$-bundle. Given any matrix representation $\rho: G \rightarrow \GL(V)$, we can form the {\bfseries associated bundle}
			 
			$$P(V) \defeq P \times_G V = (P \times V) / G.$$
			This is naturally equipped with the structure of vector bundle $P(V) \rightarrow X$. Of particular interest is the case when $V = \frak{g}$ is the Lie algebra of $G$, and $\rho$ is the adjoint representation. We will assume that $\frak{g}$ is equipped with a choice of $\mathrm{Ad}$-invariant inner product $\langle \cdot, \cdot \rangle$. This is always possible since $G$ is compact, and this is unique (up to multiplication by a scalar) when $\frak{g}$ is simple. The $\mathrm{Ad}$-invariance implies that the inner product descends to a bundle metric $\langle \cdot, \cdot \rangle$ on the vector bundle $P(\frak{g})$.

			The Lie bracket $\left[ \cdot, \cdot \right]: \frak{g} \otimes \frak{g} \rightarrow \bb{R}$ is Ad-invariant, and so combines with the wedge to endow $\Omega^{\bullet}(X, P(\frak{g}))$ and $\Omega^\bullet(P, \frak{g})$ with the structure of a graded algebra; we denote this multiplication by $\left[ \mu \wedge \nu \right]$ in both cases. Then pullback by $\pi:P \rightarrow X$ induces a graded algebra homomorphism
			
			$$\pi^*: \Omega^\bullet(X, P(\frak{g})) \longrightarrow \Omega^\bullet(P, \frak{g}).$$
			Moreover, $\pi^*$ is injective with image given by the {\bfseries basic} forms

$$\Omega^\bullet(P, \frak{g})_{\ba} \defeq \left\{\mu \in \Omega^\bullet(P, \frak{g}) \left| \begin{array}{ll}
																								(\afu_P)^*\mu= \mathrm{Ad}(\afu^{-1}) \mu,& \forall \afu \in G\\
																								\iota_{\xi_P} \mu = 0, & \forall \xi \in \frak{g}\\
																								\end{array} \right.\right\}.$$
			Here $\afu_P$ (resp. $\xi_P$) is the image of $\afu \in G$ (resp. $\xi \in \frak{g}$) under the map $G \rightarrow \mathrm{Diff}(P)$ (resp. $\frak{g} \rightarrow \mathrm{Vect}(P)$) afforded by the group action. We will use $\pi^*$ to identify the spaces $\Omega^\bullet(X, P(\frak{g})) $ and $\Omega^\bullet(P, \frak{g})_\ba$.

\medskip

		Denote by 
		
		$${\mathcal{A}}(P) = \left\{\afA \in \Omega^1(P, \frak{g}) \left|\begin{array}{ll}
																								(\afu_P)^*\afA = \mathrm{Ad}(\afu^{-1}) \afA,& \forall \afu \in G\\
																								\iota_{\xi_P} \afA = \xi, & \forall \xi \in \frak{g}\\
																								\end{array}\right. \right\}$$
		the space of {\bfseries connections} on $P$. Then $\A(P)$ is an affine space modeled on $\Omega^1(X, P(\frak{g}))$, and the affine action is given by $\afA +  \pi^* \mu$ for $\afA \in \A(P)$ and $\mu \in \Omega^1(X, P(\frak{g}))$. In particular, $\A(P)$ is a smooth (infinite dimensional) manifold with tangent space $\Omega^1(X, P(\frak{g}))$. Each connection $\afA \in {\A}(P)$ determines a covariant derivative
		
		$$d_{\afA}: \Omega^\bullet(X, P(\frak{g}))  \longrightarrow  \Omega^{\bullet + 1} (X, P(\frak{g})), \indent \indent \indent \mu  \longmapsto (\pi^*)^{-1}\left( d \left( \pi^*\mu\right)+ \left[ {\afA}\wedge \pi^*\mu\right]\right),$$
			where $d$ is the trivial connection on the trivial bundle $P \times \frak{g}$. Define the {\bfseries curvature} of $\afA$ by
			
			$$F_{\afA} = (\pi^*)^{-1}\left(d {\afA} + \frac{1}{2} \left[{\afA} \wedge {\afA} \right]\right) \in \Omega^2(X, P(\frak{g})).$$
			It follows that $d_{\afA}(d_{\afA}(\mu)) = \left[ F_{\afA} \wedge \mu\right]$ for all $\mu \in \Omega^\bullet(X, P(\frak{g}))$. It is convenient to view $\Omega^\bullet(X, P(\frak{g}))$ as acting on itself by left multiplication, in which case we can write $d_{\afA} \circ d_{\afA} = F_{\afA}$. In identifying $\Omega^\bullet(X, P(\frak{g}))$ with $\Omega^\bullet(P, \frak{g})_\ba$ we will typically drop $\pi^*$ from the notation. For example, the covariant derivative and curvature satisfy the following
			
			$$
			d_{{\afA} + \mu} = d_{\afA} + \mu, \indent \indent \indent  F_{{\afA} + \mu}  =  F_{\afA} + d_{\afA} \mu + \frac{1}{2} \left[\mu \wedge \mu \right], \indent \indent \indent d_{\afA} F_{\afA} = 0.$$
			The last is the {\bfseries Bianchi identity}.
			
				Given a connection ${\afA} \in \A(P)$ with covariant derivative $d_{\afA}$, we define the {\bfseries formal adjoint} on $k$-forms by $d_{\afA}^* \defeq -(-1)^{(n-k)(k-1)} *d_{\afA}* \mu$. Stokes' theorem shows that this satisfies $(d_{\afA} \mu, \nu) = (\mu, d_{\afA}^* \nu)$, where $\left( \cdot, \cdot\right)$ is the $L^2$-pairing, and $\mu, \nu$ have compact support in the interior of $X$. It follows from the scaling properties of forms that $d_{\afA}^{*_c} = c^{-2}d_{\afA}^*,$ where $d_{\afA}^{*_c}$ is the adjoint defined with respect to the metric $c^2g$. 		
				
			We will be interested in the {\bfseries flat} connections ${\afA} \in \A(P)$. By definition, these satisfy $F_{\afA} = 0$, and we will denote the set of flat connections on $P$ by $\A_{\fl}(P)$. If ${\afA}$ is flat then $\mathrm{im}\: d_{\afA} \subseteq \ker d_{\afA}$ and we can form the {\bfseries harmonic spaces}
			
			$$H^k_{\afA} \defeq H^k_{\afA}(X, P(\frak{g})) \defeq \frac{\ker\left( d_{\afA}\vert \Omega^k(X, P(\frak{g}))\right)}{\mathrm{im}\left( d_{\afA}\vert \Omega^{k-1}(X, P(\frak{g}))\right)}, \indent \indent \indent H^\bullet_{\afA}  \defeq \bigoplus_k H^k_{\afA} .$$
			In general, we say that a connection ${\afA}$ is {\bfseries irreducible} if $d_{\afA}$ is injective on $\Omega^0(X, P(\frak{g}))$ (when $\afA$ is flat this is equivalent to $H^0_{\afA} = 0$). It follows from (\ref{infgaugeaction}) below that if ${\afA}$ is irreducible, then the gauge action at ${\afA}$ is locally free.

			Suppose now that $X$ is closed. Then the Hodge isomorphism \cite[Theorem 6.8]{War} says
			\begin{equation}
			H^\bullet_{\afA} \cong \ker(d_{\afA} \oplus d_{\afA}^*), \indent \Omega^\bullet(X, P(\frak{g})) \cong H_{\afA}^\bullet \oplus \mathrm{im}(d_{\afA}) \oplus \mathrm{im}(d_{\afA}^*),
			\label{hodegedecomp}
			\end{equation}
			for any flat connection ${\afA}$ on $X$. Here the decomposition is orthogonal with respect to the $L^2$-inner product defined above. We will treat these isomorphisms as identifications. From the first isomorphism in (\ref{hodegedecomp}) we see that $H_{\afA}^\bullet$ is finite dimensional since $d_{\afA} \oplus d_{\afA}^*$ is elliptic (locally its leading order term is $d \oplus d^*$). Furthermore, it is clear that $*:H^\bullet_{\afA} \rightarrow H^{n-\bullet}_{\afA}$ restricts to an isomorphism on the harmonic spaces, and so the pairing (\ref{pairing}) continues to be non-degenerate when restricted to the harmonic spaces.

			\begin{example}
			Suppose $X = \Sigma$ is a closed, oriented surface equipped with a metric. Then on 1-forms $*$ squares to -1, and so defines a complex structure on $\Omega^1(\Sigma, P(\frak{g}))$. Furthermore, the pairing in (\ref{pairing}) determines a symplectic structure $\omega$ on $\Omega^1(\Sigma, P(\frak{g}))$ and the triple $(\Omega^1(\Sigma, P(\frak{g})), *, \omega)$ is K\"{a}hler. Similarly, $(H^1_{\afA}, *, \omega)$ is K\"{a}hler and finite-dimensional.
			\end{example}

			The {\bfseries Yang-Mills functional} is defined to be the map
	
	$$\YM: \A(P)  \longrightarrow  \bb{R}, \indent \indent {\afA}  \longmapsto  \fracd{1}{2} \Vert F_{\afA} \Vert_{L^2}^2.$$
		We will refer to the value at $\afA$ as the {\bfseries energy} of the connection $\afA$. Suppose $X$ is a 4-manifold. On 2-forms the Hodge star squares to the identity, and we denote by $\Omega^{\pm} (X, P(\frak{g}))$ the $\pm 1$ eigenspace of $*$. Then we have an $L^2$-orthogonal decomposition $\Omega^2(X, P(\frak{g})) = \Omega^{+} (X, P(\frak{g}))  \oplus \Omega^{-} (X, P(\frak{g})),$ with the orthogonal projection to $\Omega^{\pm} (X, P(\frak{g}))$ given by $\mu \mapsto \frac{1}{2}(1\pm *)\mu$. The elements of $\Omega^{-}(X, P(\frak{g}))$ are called {\bfseries anti-self dual} 2-forms. A connection ${\afA}$ is said to be {\bfseries anti-self dual (ASD)} or an {\bfseries instanton} if its curvature $F_{\afA} \in \Omega^{-} (X, P(\frak{g}))$ is an anti-self dual 2-form. The instantons are the minimizers of the Yang-Mills functional.

		\medskip

			A {\bfseries gauge transformation} is an equivariant bundle map $\afu:P \rightarrow P$ covering the identity. The set of gauge transformations on $P$ forms an infinite-dimensional Lie group, called the {\bf gauge group}, and is denoted ${\G}(P)$. This can be equivalently described as (i) the group of $G$-equivariant maps $P \rightarrow G$, or (ii) the group of sections of the fiber bundle $P \times_G G \rightarrow X$; in both cases $G$ is acting on itself by conjugation. From this latter description it is clear that the space $\Omega^0(X , P(\frak{g})) = \Gamma\left(P \times_G \frak{g}  \right)$ can be identified with the Lie algebra of $\G(P)$. Moreover, given $\xi \in \Omega^0(X, P(\frak{g}))$, we define $\exp(\xi) \in \G(P)$ to be the section $x \longmapsto \exp(-\xi(x))$, where in this latter formula $\exp$ is the exponential for $G$ (the negative sign is introduced to absorb a negative sign later). We define $\G_0(P) \subseteq \G(P)$ to be the connected component of the identity gauge transformation. 
			
			\begin{remark}\label{representationremark}
			Fix a faithful matrix representation $G \rightarrow \GL(\bb{C}^n)\subset \End(\bb{C}^n)$. This induces a Lie algebra representation $\frak{g} \rightarrow \End(\bb{C}^n)$. Then $\Omega^0(X, P(\frak{g}))$ and $\G(P)$ both embed into $\Gamma(P \times_G \mathrm{End}(\bb{C}^n))$. With respect to this embedding, it makes sense to multiply elements of $\G(P)$ with elements of its Lie algebra.
			\end{remark}

			The gauge group acts on the left on the space $\Omega^\bullet(P, \frak{g})$ by $({\afu}, {\afA}) \longmapsto  ({\afu}^{-1})^* {\afA}$, for ${\afu} \in \G(P), \:{\afA} \in \Omega^\bullet(P, \frak{g})$. Here the star indicates the pullback map induced by the bundle map ${\afu}^{-1}: P \rightarrow P$. The action of $\G(P)$ on $\Omega^\bullet(P, \frak{g})$ restricts to actions on $\Omega^\bullet(P, \frak{g})_{\ba}$ and ${\A}(P)$. Viewing a gauge transformation as a map ${\afu}: P \rightarrow G$ we can write this latter action as

			\begin{equation}
			({\afu}^{-1})^* {\afA}= {\afu}{\afA}{\afu}^{-1} + {\afu} d({\afu}^{-1}),
			\label{slop}
			\end{equation}
			where we have fixed a faithful matrix representation as in Remark \ref{representationremark}, and so the concatenation appearing on the right is just matrix multiplication. The curvature of ${\afA} \in \A(P)$ transforms under ${\afu}\in \G(P)$ by $F_{({\afu}^{-1})^*{\afA}} = \mathrm{Ad}({\afu})F_{\afA}$. This shows that $\G(P)$ restricts to an action on ${\A}_{\fl}(P)$ and, in 4-dimensions, the instantons. The infinitesimal action of ${\G}(P)$ at ${\afA} \in {\A}(P)$ takes the form
			
			\begin{equation}
				\Omega^0(X, P(\frak{g}))   \longrightarrow  \Omega^1(X, P(\frak{g})), \indent \indent \indent \xi  \longmapsto  d_{\afA} \xi.
				\label{infgaugeaction}
				\end{equation}

\medskip

Suppose $X$ is compact. Fix a reference connection ${\afA}_{\mathrm{ref}} \in {\A}(P)$. This combines with the Levi-Civita connection to define a norm on $W^{k,p}\left(\Lambda^j T^*X \otimes P(\frak{g})\right)$, as above. Then we set
				
				$$\A^{k,p}(P) \defeq {\afA}_{\mathrm{ref}} + W^{k,p}\left(T^*X \otimes P(\frak{g})\right).$$
				The space $\A^{k,p}(P)$ is independent of the choice of reference connection ${\afA}_{\mathrm{ref}}$, and the norm only depends on ${\afA}_{\mathrm{ref}}$ up to norm equivalence. Moreover, $\A^{k,p}(P) $ is an affine space modeled on the vector space $W^{k,p}\left(T^*X \otimes P(\frak{g})\right)$, and so Sobolev's embedding theorems carry over directly to $\A^{k, p}(P)$.

				Given an embedding $G \subseteq \U(r)$, define $\G^{k,p}(P)$ to be the subset of sections in $W^{k,p}( \mathrm{End}(\bb{C}^{r}))$ whose images lie in $G \subset \U(r) \subset \mathrm{End}(\bb{C}^{r})$. Under such an embedding, $G$ necessarily has measure zero in $\mathrm{End}(\bb{C}^r)$; nonetheless, this is a meaningful definition whenever we are in the continuous range for Sobolev embedding (e.g. $kp \geq n$ and $k \geq 2$, or $kp > n$ and $k = 1$). The space $\G^{k,p}(P)$ forms a group when we are in the continuous range, and the group operations are smooth, making $\G^{k, p}(P)$ a Banach Lie group. Moreover, in the continuous Sobolev range the group $\G^{k,p}(P)$ acts smoothly on $\A^{k-1,p}(P)$. See \cite[Appendix B]{Wuc} for more details. When $X$ is non-compact, it is convenient to consider the space $\G_{loc}^{k,p}(P)$ of sections that are {\bfseries locally} $W^{k, p}$, defined in the obvious way.

				Next we address the case where $X = \bb{R} \times Y$, and $Y$ is compact. We assume that $P = \bb{R} \times Q$ is the pullback of a principal $G$-bundle $Q$ over $Y$. Fix connections $\afa^\pm \in \A(Q)$ on $Y$. Let ${\afA}_{\mathrm{ref}}$ be any connection on $\bb{R} \times Q$ that is constantly equal to $\afa^-$ on $\left(-\infty, -S_0 \right] \times Y$, and $\afa^+$ on $\left[ S_0 , \infty \right) \times Y$ for some $S_0 > 0$. Then we set
				
				\begin{equation}\label{a0}
				\A^{k, p}(\afa^-, \afa^+) \defeq {\afA}_{\mathrm{ref}} + W^{k,p}\left(T^*X \otimes P(\frak{g})\right).
				\end{equation}
				This is a Banach affine space, and all elements limit to $\afa^\pm$ at $\pm \infty$. Moreover, as a topological space, it is independent of the choice of reference connection that agrees with the $\afa^\pm$ off of a compact set. However, the particular choice of Banach norm depends on the choice of $\afA_{\mathrm{ref}}$. The appropriate gauge group is 
				
				\begin{equation}\label{g0}
				\G_c^{k, p}(\bb{R} \times Q),
				\end{equation}
				which we define to be the $W^{k,p}$-closure of the elements of $\G^{k, p}_{loc}(\bb{R} \times Q)$ that restrict to the identity element off of compact sets. It follows that $\G_c^{k, p}(\bb{R} \times Q)$ is a Banach Lie group with Lie algebra the $W^{k,p}$-completion of the compactly supported 0-forms with values in $Q(\frak{g})$. Moreover, $\G_c^{k, p}(\bb{R} \times Q)$ acts smoothly on the space $\A^{k-1, p}(\afa^-, \afa^+)$, at least when $k, p$ are in the appropriate Sobolev range.

		Let $s$ denote the coordinate variable on $\bb{R}$, and $\frac{\partial}{\partial s} \in \mathrm{Vect}(\bb{R} \times P)$ the obvious vector field. Then a connection ${\afA} \in \Omega^1(\bb{R} \times P, \frak{g})$ is in {\bfseries temporal gauge} if its contraction with $\partial / \partial s$ vanishes $\iota_{\partial / \partial s } {\afA}  = 0.$ It can be shown that for each connection ${\afA}$, there is an identity-component gauge transformation ${\afu} \in \G_0(\bb{R} \times P)$ with ${\afu}^* {\afA}$ in temporal gauge. Note that any connection ${\afA} \in \A(\bb{R} \times P)$ can be written in the form ${\afA} = \afa(s) + \afp(s) \: ds$ for unique $\afa: \bb{R} \rightarrow \A(P)$ and $\afp: \bb{R} \rightarrow \Omega^0(X, P(\frak{g}))$, where $ds \in \Omega^1(\bb{R} \times Y)$ is the obvious 1-form. Then ${\afA}$ is in temporal gauge if and only if $\afp = 0$.

		\medskip

	We will need a \emph{perturbed} version of Uhlenbeck's compactness theorem for flat connections and instantons. Specifically, we fix a $\G(P)$-equivariant map $V : \A(P) \rightarrow \Omega^2(X, P(\frak{g}))$ which is smooth with respect to the $W^{k, p}$-topology on each space, where $k, p \geq 1$. We assume the following:
	
	\begin{itemize}
	\item[(a)] $V$ is uniformly $L^\infty$-bounded, $\sup_{A \in \A(P)} \Vert V(A) \Vert_{L^\infty(X)} < \infty$.
	\item[(b)] If $\left\{A_n \right\}$ is a sequence of connections such that $\Vert A_n - A_0 \Vert_{L^p(X)}$ is bounded, then $V(A_n)$ has an $L^p$-convergent subsequence.
	\end{itemize}
	In particular, Kronheimer proves that these conditions are satisfied whenever $V$ is the differential of a holonomy perturbation \cite[Lemma 10]{Kron}. 
	
	\begin{theorem}\cite[Proposition 11]{Kron} \label{kronsmagic}
	 Suppose $V$ satisfies conditions (a) and (b) above, and $\left\{A_n \right\}$ is a sequence of connections on a fixed bundle over $X$ such that $\Vert F_{A_n} - V(A_n) \Vert_{L^p(X)}$ is bounded. Then there is a finite set of points $B \subset X$, a subsequence (still denoted $\left\{A_n\right\}$), and a sequence of gauge transformations $\left\{\afu_n\right\}$ such that $\afu_n^* \afA_n$ converges weakly in $W^{1, p}$ on compact subsets of $X \backslash B$ to a connection $\afA_\infty$ defined on $X \backslash B$. If $ F_{A_n} - V(A_n) = 0$ for all $n$, then this limiting connection satisfies $F_{A_\infty} - V(A_\infty)=0$, the sequence $\afu^* \afA_n$ converges strongly in $W^{1, p}$ on compact subsets of $X \backslash B$, and $A_\infty$ is defined on a (possibly different) bundle over $X$.
	
	If $\dim\: X = 4$, then the same result holds with $F_{A} - V(A)$ replaced by $(F_{A} - V(A))^+$, where the superscript $+$ denotes the anti-self dual component. 
	\end{theorem}
	
	Consider the case where $(F_{A_n} - V(A_n))^+ = 0$ for all $n$. When $V$ is zero, the (unperturbed version of) Uhlenbeck's compactness theorem implies that a subsequence of $\afu^*_n \afA_n$ converges in ${\mathcal{C}}^\infty$ to $\afA_\infty$. Kronheimer points out that when $V$ is a (non-zero) holonomy perturbation, then one can no longer expect to get anything stronger than the $W^{1, p}$ convergence of Theorem \ref{kronsmagic}. This is due to the non-locality inherent in holonomy perturbations.

\medskip

As we have already begun to see, when the base manifold $X$ is a product, it is useful to have a notation that distinguishes between a connection or gauge transformation and its components. We therefore adopt the following notation: 

\begin{itemize}
\item On 4-manifolds we will use $\afA$ and $\afU$ for connections and gauge transformations.
\item We tend to denote 3-manifolds by $Y$, and use $\afa$ and $\afu$ for connections and gauge transformations on $Y$.
\item We tend to denote 2-manifolds by $\Sigma$, and use $\afalpha$ and $\afmu$ for connections and gauge transformations on $\Sigma$.
\end{itemize}
On the other hand, when the base manifold has no relevant product structure, then we continue to use $\afA$ and $\afu$ for connections and gauge transformations.

\subsection{Topological aspects of principal $\PU(r)$-bundles}\label{TopolicaAspectsOfPrincipalPU(r)-Bundles}

			Throughout the remainder of this paper we work primarily with the Lie group $\PU(r) \defeq \SU(r)/ \bb{Z}_r = \U(r) / \U(1),$ for $r \geq 2$.	 It follows that $\pi_1(\PU(r)) = \bb{Z}_r,$ since $\SU(r)$ is simply-connected. We fix an $\mathrm{Ad}$-invariant inner product $\langle \cdot, \cdot \rangle$ on the Lie algebra ${\frak{pu}}(r)$ to $\PU(r)$. This Lie algebra is simple, and so $\langle \mu , \nu \rangle = -\kappa_{r} \mathrm{tr}(\mu  \cdot \nu),$ for some $\kappa_{r} > 0$, where the trace is the one induced from the identification ${\frak{pu}}(r) \cong \frak{su}(r) \subset \mathrm{End}(\bb{C}^r)$. We leave $\kappa_r$ arbitrary, but fixed.

		\medskip
		
		In \cite{Woody2}, L.M. Woodward exploited the adjoint representation to classify the principal $\PU(r)$-bundles over spaces of dimension $\leq 4$. This classification scheme assigns cohomology classes

		$$t_2(P) \in H^2(X, \bb{Z}_r), \indent q_4(P) \in H^4(X, \bb{Z})$$
		to each principal $\PU(r)$-bundle $P \rightarrow X$. For example, $q_4$ is defined to be the second Chern class of the complexified adjoint bundle $P(\frak{g})_{\bb{C}} \defeq P(\frak{g}) \otimes \bb{C}$, where $\frak{g} = {\frak{pu}}(r)$. The class $t_2$ is defined as the mod $r$ reduction of a suitable first Chern class. We will be mostly interested in the case where $X$ is a smooth manifold, but these classes are defined for CW complexes as well. 
				
\begin{example}
When $r = 2$ we have $\PU(2) = \SO(3)$, and the classes $t_2$ and $q_4$ are the second Stiefel-Whitney class and the (negative of the) first Pontryagin class, respectively. 
\end{example}

				\noindent We summarize the properties of these classes that we will need.

		\begin{itemize}
			\item If $\dim(X) \leq 4$ and $X$ is a manifold, then two bundles $P$ and $P'$ over $X$ are isomorphic if and only if $t_2(P) = t_2(P')$ and $q_4(P)= q_4(P')$;

			\item The class $t_2(P)$ is zero if and only if the structure group of $P$ can be lifted to $\SU(r)$.

		\item The class $t_2(P)$ is the mod $r$ reduction of an integral class if and only if the structure group of $P$ can be lifted to $\U(r)$.
		\end{itemize}

		These characteristic classes can be used to study the components of the gauge group $\G(P)$ for a principal $\PU(r)$-bundle $P \rightarrow X$. We take as a starting point the following observation by Donaldson. We refer the reader to \cite{DunPU} for more details of the assertions in the remainder of this section. 
		
		\begin{proposition}{\cite[Section 2.5.2]{Donfloer}}\label{donfloer}
		Suppose $G$ is a compact Lie group, $X$ is a smooth manifold and $P \rightarrow X$ is a principal $G$-bundle. Then there is a bijection between $\pi_0(\G(P))$ and the set of isomorphism classes of principal $G$-bundles over $S^1 \times X$ that restrict to $P$ on a fiber. This bijection is induced from the map that sends a gauge transformation $\afu$ to the bundle $P_{\afu} \defeq \left[0, 1 \right] \times P / (0, {\afu}(p)) \sim (1, p),$ given by the mapping torus of $\afu$. 
		\end{proposition}

		Now we combine this with L.M. Woodward's classification. Fix a principal $\PU(r)$-bundle $P \rightarrow X$. To focus the discussion, we assume $X$ is a manifold of dimension at most $3$. Let ${\afu} \in \G(P)$, define $P_{\afu} \rightarrow S^1 \times X$ as in the proposition above, and consider the classes $t_2(P_{\afu}) \in H^2(S^1 \times X, \bb{Z}_r), q_4(P_{\afu}) \in H^4(S^1 \times X, \bb{Z})$. By the K\"{u}nneth formula, we have an isomorphism
		
		$$H^k(S^1 \times X, R) \cong H^k(X, R) \oplus H^{k-1}(X, R),$$
		where $R = \bb{Z}$ or $\bb{Z}_r$. The image of $t_2(P_{\afu}) \in H^2(X, \bb{Z}_r) \oplus H^1(X, \bb{Z}_r)$ in the first factor is exactly $t_2(P)$, so the dependence of $t_2(P_{\afu})$ on the gauge transformation ${\afu}$ is contained entirely in the projection of $t_2(P_{\afu})$ to the second factor $H^1(X, \bb{Z}_r)$. We denote this projection by 
		
		$$\eta({\afu}) \defeq \eta_X({\afu}) \in H^1(X, \bb{Z}_r),$$
		and call this the {\bfseries parity} of ${\afu}$. 
		
		In a similar fashion, we will define the \emph{degree} of a gauge transformation using $q_4$. Before defining this explicitly, we note that a straight-forward characteristic class argument shows that the image of $q_4(P_{\afu})$ in $H^3(X, \bb{Z})$ is always even. Then we define the {\bfseries degree} of ${\afu}$ to be the class 
		
		$$\deg({\afu})  \defeq \frac{1}{2} \proj_{H^3} \:q_4(P_{\afu}) \in H^3(X, \bb{Z}).$$ 
		where $\proj_{H^3}: H^4(S^1 \times X, \bb{Z}) \rightarrow H^3(X, \bb{Z})$ is given by contraction with the generator of $H^1(S^1, \bb{Z})$. See \cite{DS2} and \cite{Flinst} for alternative realizations of the parity and degree.

			Proposition \ref{donfloer}, and the properties of $t_2$ and $q_4$ immediately imply that the parity and degree detect the components of the gauge group when $\dim(X) \leq 3$. The next proposition describes some information captured by the parity. It is standard when $r=2$ and the extension to larger $r$ is no different.

		\begin{proposition}\label{trivgaugetrans}
				Fix a gauge transformation ${\afu}: P \rightarrow \PU(r)$. Then the following are equivalent:

				\begin{itemize}
					\item $\eta({\afu}) = 0$;
					\item ${\afu}: P \rightarrow \PU(r)$ lifts to a $\PU(r)$-equivariant map $\widetilde{{\afu}}: P \rightarrow \SU(r)$;
					\item When restricted to the 1-skeleton of $X$, ${\afu}$ is homotopic to the identity map.
				\end{itemize}
				If $X$ is a compact, connected, oriented 3-manifold and $\eta({\afu}) = 0$, then $\deg({\afu})$ is divisible by $r$.
		\end{proposition}

		Next we give a formula for computing the degree. Let $Q \rightarrow Y$ be a $\PU(r)$-bundle over a closed, connected, oriented 3-manifold $Y$, and let $\afu \in \G(Q)$. Fix a connection ${\afa}_0 \in \A(Q)$, and let ${\afa}: I \rightarrow \A(Q)$ be any path from $\afa_0$ to $\afu^* \afa_0$. This defines a connection ${\afA}$ on $I \times Q \rightarrow I \times Y$ by declaring ${\afA} \vert_{\left\{s \right\} \times Y } = \afa(s)$ for $s \in I$. Moreover, ${\afA}$ descends to a connection on the mapping torus $Q_{\afu}$. The usual Chern-Weil formula for the second Chern class of $\SU(r)$-bundles gives $q_4(Q_u) =  -r(4\pi^2 \kappa_r)^{-1}\int_{I \times Y} \langle F_{\afA} \wedge F_{\afA} \rangle.$ The curvature decomposes into components as $F_{\afA} = F_{\afa(s)}  + ds \wedge \partial_s \afa(s)$, and so
		
		\begin{equation}
		\deg(\afu) = \fracd{r}{4\pi^2 \kappa_r}\intd{I} \left( \intd{Y} \langle F_{\afa(s)} \wedge \partial_s \afa(s) \rangle \right). 
		\label{constantthings}
		\end{equation}
		(Our orientation convention is that $ds \wedge d\mathrm{vol}_Y$ is a positive volume form on $I \times Y$.) This formula also shows that the degree operator is a homomorphism of groups.
		
		\medskip

		We end the discussion of the degree and parity by mentioning two applications. These will be useful in our treatment of instanton Floer theory later. 
		
		\medskip

\noindent \emph{Application 1: Free actions on flat connections.} We suppose $X$ is a connected and oriented manifold of dimension 2 or 3, equipped with a principal $\PU(r)$-bundle $P \rightarrow X$. In case $X$ is 2-dimensional, we assume that $t_2(P)\left[ X \right] \in \bb{Z}_r$ is a generator. If $X$ is 3-dimensional then we assume there is an embedding $\Sigma \hookrightarrow X$ of a connected, oriented surface such that $t_2(P) \left[ \Sigma \right] \in \bb{Z}_r$ is a generator.

		\begin{lemma}\label{gaugestab0} \cite[Lemma 2.5]{DS2}
			Suppose $P \rightarrow X$ is as above. Then all of the flat connections on $P$ are irreducible. In fact, any flat connection $\afa$ on $P$ has trivial stabilizer in $\ker \eta$:
			$$\left\{\left. {\afu} \in \G(P) \: \right|\: \eta(\afu) = 0,\indent {\afu}^*{\afa} = {\afa} \:  \right\} = \left\{e \right\}.$$
			\end{lemma}

	\medskip 

	\noindent \emph{Application 2: Existence of degree $d$ gauge transformations.} Fix a closed, connected, oriented 3-manifold $X$ and a principal $\PU(r)$-bundle $P \rightarrow X$. We assume that $t_2(P)$ is the reduction of an integral class. This implies $P = \overline{P} \times_{\U(r)} \PU(r)$ is induced from a principal $\U(r)$-bundle $\overline{P} \rightarrow X$, and $t_2(P) = c_1(\overline{P}) $ mod $r$. Let $ d\in \bb{Z}$, and assume that $c_1(\overline{P})\left[ \Sigma \right] = d$ for some embedded connected, oriented surface $\Sigma \hookrightarrow X$. (In particular, $P$ satisfies the conditions of Application 1 when $d,r$ are relatively prime.) 
	
	\begin{proposition}\label{degprop}
			Let $d \in \bb{Z}$, $P \rightarrow X$ and $\Sigma \subset X$ be as above. 

			\medskip			
			
			(a) There exists a gauge transformation $\afu \in \G(P)$ of degree $d$. Moreover, the parity $\eta(\afu) :H_1(X) \rightarrow \bb{Z}_r$ is given by the mod $r$ reduction of the Poincar\'{e} dual $\textrm{PD}_X\left[\Sigma \right] : H_1(X) \rightarrow \bb{Z}$. 
			
			\medskip
			
			(b) Suppose $d$ and $r$ are relatively prime. Then there exists a gauge transformation $\afu \in \G(P)$ of degree 1. Moreover, writing $md + nr =1$, the parity $\eta(\afu) :H_1(X) \rightarrow \bb{Z}_r$ is given by the mod $r$ reduction of $m$ times the Poincar\'{e} dual $\textrm{PD}_X\left[\Sigma \right] : H_1(X) \rightarrow \bb{Z}$. 
	\end{proposition}

In the case $r = 2$, Dostoglou and Salamon \cite[Lemma 2.3, Lemma A.2]{DS2} prove this by explicitly constructing the desired gauge transformation. In \cite{DunPU}, we present a proof for general $r \geq 2$ using the properties of the characteristic classes $t_2, q_4$.

  \subsection{Compatible bundles for broken circle fibrations}\label{CompatibleBundlesForBrokenCircleFibrations}
  
Fix an integer $r \geq 2$ and $d \in \bb{Z}_r$. Let $f: Y \rightarrow S^1$ be a broken circle fibration. Suppose we are given a map $\gamma: S^1 \rightarrow Y$ that is a section of $f$ (these always exist). Then by the classification of principal $\PU(r)$-bundles on 3-manifolds, there is a unique, up to bundle isomorphism, principal $\PU(r)$-bundle $Q= Q(\gamma)  \longrightarrow Y$ such that $t_2(Q) \in H^2(Y, \bb{Z}_r)$ is Poincar\'{e} dual to the class $d \left[ \gamma \right] \in H_1(Y, \bb{Z}_r)$. We will say that a bundle $Q$ is {\bfseries $d$-compatible} with $f: Y\rightarrow S^1$ if it can be obtained in this way. Since $\left[ \gamma \right]$ is the reduction of an integral class, it follows that $t_2(Q)$ is as well. Moreover, $t_2(Q) \left[ \Sigma_i\right] = d$ for all $i$.

			In general, there may be multiple isomorphism classes of bundles that are $d$-compatible with $f$. For example, suppose $\gamma$ is a section of $f$, and $\alpha : S^1 \rightarrow \Sigma_0$ is a loop in the fiber $\Sigma_0$. Then the concatenation $\gamma *\alpha$ is homologous in $H_1(Y)$ to a section $\gamma'$ of $f$. The associated bundles $Q(\gamma)$ and $Q(\gamma')$ will be isomorphic if and only if $\gamma$ and $\gamma'$ induce the same element of $H_1(Y, \bb{Z}_r)$.

Recall the decomposition (\ref{decopmofYbullet}) of a broken circle fibration $Y \rightarrow S^1$ into elementary cobordisms $I \times \Sigma_j$ and $Y_{j(j+1)}$. Given a $d$-compatible bundle $Q \rightarrow Y$, we set

$$Q_\bullet \defeq \bigcup_i Q_{i(i+1)}, \indent Q_{i(i+1)} \defeq Q\vert_{Y_{i(i+1)}}, \indent \indent \indent P_\bullet \defeq \bigcup_i P_i, \indent P_i \defeq Q \vert_{\Sigma_i}.$$
Then $Q_\bullet$ and $P_\bullet$ are naturally bundles over $Y_\bullet$ and $\Sigma_\bullet$. Let $F_\epsilon: Y^1 \rightarrow Y^\epsilon$ be the diffeomorphism (\ref{milnorsdiffeo}), and define $Q^\epsilon \defeq (F_\epsilon^{-1})^*Q.$ This is a smooth $\PU(r)$-bundle over $Y^\epsilon$ enjoying the same topological properties as $Q \rightarrow Y$. Since $F_\epsilon$ is the identity on $Y_\bullet$ and $\left\{0 \right\} \times \Sigma_\bullet$ it follows that the restrictions $Q_\bullet = \left.Q^\epsilon \right|_{Y_\bullet}$ and $P_\bullet = \left. Q^\epsilon \right|_{\left\{0\right\} \times \Sigma_\bullet}$ do not depend on $\epsilon$.

		We will use a subscript $\eta_Y$ (resp. $\eta_{\Sigma_i}$) to denote the parity operator for $Q \rightarrow Y$ (resp. $P_i \rightarrow \Sigma_i$). Let $\afu \in \G(Q)$, and notice that $\eta_{\Sigma_i}\left(\afu\vert_{\Sigma_i} \right) = 0 $ for some $i$ if and only if $\eta_{\Sigma_i}\left(\afu\vert_{\Sigma_i} \right) $ for all $i$. We will be particularly interested in the following subgroup
		$$\G_{\Sigma_\bullet} \defeq \left\{ \afu \in \G(Q) \: \vert \: \eta_{\Sigma_0}\left(\afu\vert_{\Sigma_0} \right) = 0 \right\}.$$
		There is a sequence of inclusions
		$$\G_0(Q) \subset \ker \eta_Y \subset \G_{\Sigma_\bullet} \subset \G(Q).$$
		Each of these inclusions is strict. For example, the first is strict since any gauge transformation in $\ker \eta_Y$ lifts to an $\SU(r)$-gauge transformation, and there are certainly $\SU(r)$-gauge transformations with non-trivial degree. That the second inclusion is strict follows from Proposition \ref{degprop} (a) and the fact that the restriction of $\eta_Y(u)$ to $H^1(\Sigma_0) \subset H^1(Y)$ is exactly $\eta_{\Sigma_0}(u \vert_{\Sigma_0})$. 
		
		The following is the key result regarding $\G_{\Sigma_\bullet}$.

		\begin{proposition}\label{GSigma} 
			Suppose $d \in \bb{Z}_r$ is a generator and $Q \rightarrow Y$ is $d$-compatible with a broken circle fibration $f: Y \rightarrow S^1$. Then the degree descends to an isomorphism $\G_{\Sigma_\bullet} / \G_0(Q) \cong \bb{Z}$. 
		\end{proposition}  
 			 
  See \cite[Proposition 4.8]{DunPU} for a proof. It follows that the positive generator of $\G_{\Sigma_\bullet} / \G_0(Q)  \cong \bb{Z}$ is the homotopy class of the degree 1 gauge transformation from Proposition \ref{degprop} (b). Moreover, by Proposition \ref{trivgaugetrans}, the image of $\ker \eta_Y$ in $\G_{\Sigma_\bullet} / \G_0(Q) $ is identified with the multiples of $r$:
  $$\ker \eta_Y / \G_0(Q) \cong r \bb{Z} \subset \bb{Z}.$$

\subsection{Quilted Floer theory}\label{QuiltedFloerCohomology}

In this section we describe how to assign a quilted Floer cohomology group to a $d$-compatible bundle $Q$ over a broken circle fibration $f: Y \rightarrow S^1$. We begin by introducing the moduli spaces of flat connections on surfaces and cobordisms. These bring symplectic geometry into the picture.

			\subsubsection{Moduli spaces of flat connections}\label{ModuliSpacesOfFlatConnections}

 Fix a principal $\PU(r)$-bundle $P \rightarrow X$, and we assume $\dim(X) \leq 3$. For $2 \leq q < \infty$ and $k \geq 1$, define
			
			$$\M(P) \defeq \A_{\fl}^{k,q} (P) / \G_0^{k+1,q}(P),$$			
			where $\G_0^{k+1,q}(P)	\subseteq \G^{k+1,q}(P)$ is the identity component of the gauge group, and let $\Pi: \A^{1,q}_\fl (P) \longrightarrow \M(P)$ be the quotient map. Then $\Pi$ and $\M(P)$ are independent of the choice of $2 \leq q < \infty$ and $k \geq 1$ in the sense that the natural inclusion $\A^{k,q}_\fl (P) \hookrightarrow \A^{1,2}_\fl (P)$ induces a homeomorphism 
			
			$$\A^{k,q}_\fl (P)/ \G^{k+1,q}_0(P) \cong \A^{1,2}_\fl (P) / \G^{2,2}_0(P).$$ 
			Moreover, this is a diffeomorphism whenever the gauge group acts freely. 
  
 \medskip

Now suppose $X = \Sigma$ is a closed, connected, oriented surface, and $P \rightarrow \Sigma$ is a principal $\PU(r)$-bundle, with $t_2(P)\left[ \Sigma \right] \in \bb{Z}_r$ a generator. It follows from Lemma \ref{gaugestab0} that $\G_0^{2, q}(P)$ acts freely on the space of flat connections, and so it is immediate that the moduli space $M(P)$ is a smooth manifold. The next theorem goes back to Atiyah-Bott \cite{AB}, and captures the relevant properties of $\M(P)$.

\begin{theorem}\label{surfaces}
			Let $\Sigma$ be a closed, connected, oriented surface, and $P\rightarrow \Sigma$ a principal $\mathrm{PU}(r)$-bundle with $t_2(P)\left[\Sigma\right] \in \bb{Z}_r$ a generator. 
			
			\medskip
			1. \cite[Prop. 3.2.4, Thm. 3.3.2]{WWfloer} If $\Sigma$ has genus $g(\Sigma) \geq 1$, then the moduli space $\M(P)$ is a nonempty compact, symplectic manifold of dimension $(2g(\Sigma) - 2)(r^2-1)$, with even minimal Chern number, and with monotonicity constant $2\pi^2 \kappa_r / r$. The tangent space at $\left[\afalpha \right] \in \M(P)$ is isomorphic to $H^1_{\afalpha}$ and the symplectic form $\omega_{\M(P)}$ is given by restricting the pairing (\ref{pairing}). If $g(\Sigma) = 0$, then $\M(P) = \emptyset$. Moreover, $\M(P)$ is always connected and simply-connected.
			
			\medskip
			2. \cite[Lemma 3.3.5]{WWfloer} If $P' \rightarrow \Sigma$ is a second principal $\PU(r)$-bundle then any $\PU(r)$-equivariant bundle isomorphism ${\psi}: P \rightarrow P'$ covering the identity induces a symplectomorphism ${\psi}^*: \M(P') \rightarrow \M(P)$ by pullback. Furthermore, if ${\phi}: P \rightarrow P'$ is a second bundle map then ${\psi}^* = {\phi}^*$, so the moduli spaces $\M(P)$ and $M(P')$ are canonically symplectomorphic. 
			\end{theorem}
			
			Here, $\kappa_r$ is the constant coming from our choice of metric on $\frak{g}$ as in Section \ref{TopolicaAspectsOfPrincipalPU(r)-Bundles}. In \cite{WWfloer} the authors work with $G = \U(r)$ and the space of \emph{central curvature connections with fixed determinant}, rather than the space of flat $\PU(r)$ connections, as we consider here. Their theorems carry over verbatim to our situation: On a surface, the space of central curvature $\U(r)$-connections with fixed determinant is naturally diffeomorphic to the space of flat $\PU(r)$-connections, and this diffeomorphism intertwines the relevant gauge group actions (see \cite[Lemma 3.2.5]{WWfloer}).

			\begin{remark}\label{surfacesremark}
			(a) Since $\Sigma$ is assumed to be oriented, any choice of metric $g$ on $\Sigma$ induces a Hodge star, $*$. Then $*$ descends to a compatible complex structure on the tangent bundle $T\M(P)$, which we denote by $*$ or $J_g$, depending on the context. 
			
			\medskip
			
			(b) The symplectic form on $\M(P)$ is given by integration over the surface $\Sigma$, so reversing the orientation of $\Sigma$ changes $(\M(P), \omega_{\M(P)})$ to $\M(P)^- \defeq (\M(P), -\omega_{\M(P)})$. 
			
			\end{remark}

			 The manifold $M(P)$ has two alternative descriptions that are often quite useful. We refer the reader to \cite{AB}, \cite{DS2} and \cite{WWfloer} for more details. For the first description, fix a basepoint $x_0 \in \Sigma$. Then the holonomy determines a map $\A_\fl(P) / \G(P) \rightarrow \Hom(\pi_1(\Sigma, x_0), \PU(r)) / \PU(r)$ that is a homeomorphism onto a component of the right-hand side, and $\PU(r)$ is acting on the $\Hom$-set by conjugation. We will use the notation $\Hom^d_{\PU(r)}$ to refer to this component. On the other hand, quotienting $\A_\fl(P)$ only by the identity component $\G_0(P)$ of the gauge group recovers $M(P)$. In particular, we obtain a map $\M(P) \rightarrow \Hom^d_{\PU(r)}$ that is a $\bb{Z}_r = \G(P) / \G_0(P)$-covering space away from the singularities of $\Hom^d_{\PU(r)}$.

			For the second description, fix a basepoint $p_0 \in P$ over $x_0 \in \Sigma$. The fundamental group of $P$ is generated by tuples $\left(A_1, B_1, \ldots, A_g, B_g, z\right)$ subject to the relations $\prod_{j = 1}^g \left[ A_j, B_j \right] = z$ and $\left[A_j ,z \right] = \left[ B_j , z \right] = z^r = \mathrm{Id}.$ Consider the space
			
			$$\begin{array}{rcl}
			\Hom^d_{\SU(r)}  & \defeq & \left\{ \rho \in \Hom(\pi_1(P, p_0), \SU(r)) \left| \: \rho(z) = e^{2\pi i d/r} \mathrm{Id} \right. \right\}\\
			\end{array}$$
			The group $\SU(r)$ acts freely on $\Hom^d_{\SU(r)} $ by conjugation and quotienting by this action recovers $M(P) \cong \Hom^d_{\SU(r)} / \SU(r)$. This diffeomorphism arises by sending a flat connection to its $\SU(r)$-valued holonomy in $\Hom(\pi_1(P), \SU(r))$.

			\medskip

			Next, suppose $Y$ is a compact, connected, oriented cobordism between non-empty, connected, oriented surfaces $\Sigma^\pm$. Our orientation convention is that $\partial Y = \overline{\Sigma^-} \sqcup \Sigma^+$, where the bar denotes the manifold with the opposite orientation. Let $Q \rightarrow Y$ be a principal $\PU(r)$-bundle and assume that $t_2(Q) \left[\Sigma^+\right] \in \bb{Z}_r$ is a generator. It follows that $t_2(Q) \left[\Sigma^+\right] = t_2(Q) \left[\Sigma^-\right]$. In particular, Lemma \ref{gaugestab0} implies that the moduli spaces $\M(Q)$ and $\M\left(Q \vert_{\Sigma^\pm} \right)$ of flat connections are all \emph{smooth} manifolds. 
			
			Restriction to each boundary component of $Y$ induces a $\G(Q)$-equivariant map $\rho: \A(Q) \longrightarrow \A(\left. Q \right|_{{\Sigma^-}}) \times \A(\left. Q \right|_{\Sigma^+})$ that preserves the flat connections. It follows that $\rho$ descends to a map, still denoted by $\rho$, at the level of moduli spaces. Then we set 
			
			$$\La(Q) \defeq \rho(\M(Q)) \subset \M(\left. Q \right|_{\Sigma^-})^- \times \M(\left. Q \right|_{\Sigma^+}).$$

  			\begin{theorem}\label{cobordisms}
			Let $Q \rightarrow Y$ be as above. Assume, in addition, that $Y$ is an elementary cobordism.
			
			\medskip			
			1. \cite[Theorem 3.4.1]{WWfloer} The map $\rho: \M(Q) \rightarrow \La(Q) \subset \M(P^-)^- \times \M(P^+)$ is a Lagrangian embedding. Furthermore, $\La(Q)$ is compact, oriented, and simply-connected (hence monotone).
			
			\medskip
			2. \cite[Lemma 3.4.4]{WWfloer} The Lagrangian $\La(Q)$ is independent of the choice of $Q$ under the canonical symplectomorphisms of Theorem \ref{surfaces}. If $Y = I \times \Sigma^+$ is a product cobordism from $\Sigma^- = \Sigma^+$ to itself, then $\La(Q) \subset \M(P^-)^- \times \M(P^-)$ is the diagonal.
			\end{theorem}

  		\subsubsection{Quilted Floer cohomology for broken circle fibrations}\label{QuiltedFloerCohomologyForBrokenCircleFibrations}

Fix a broken circle fibration $f: Y \rightarrow S^1$ and a $d$-compatible principal $\PU(r)$-bundle $Q \rightarrow Y$, where $d \in \bb{Z}_r$ is a generator. In this section we review the definition from \cite{WWfloer} of the \emph{quilted Floer group} associated to this data. Recall the decomposition (\ref{decopmofYbullet}) of $Y$ into elementary cobordisms $Y_{j(j+1)}$ and $I \times \Sigma_j$, and the associated decomposition of the bundle $Q$ from Section \ref{CompatibleBundlesForBrokenCircleFibrations}. By construction, each $P_j \rightarrow \Sigma_j$ and $Q_{j(j+1)} \rightarrow Y_{j(j+1)}$ satisfies the conditions of Theorems \ref{surfaces} and \ref{cobordisms}, respectively. In the terminology of \cite{WWquilt}, this implies that the tuple 

$$\underline{L}(Q) \defeq \left(\La(Q_{01}), \ldots, \La(Q_{(N-1)0}\right)$$ 
determines a cyclic Lagrangian correspondence in the product symplectic manifold 

$$\underline{M} \defeq \left(M(P_j) \right)_{j = 0}^{N-1}.$$ 
It also implies that the quilted Floer cohomology group of $\underline{L}(Q)$ is well-defined. We denote this cohomology group by $HF^\bullet_\quilt(Y, f, Q)_{r, d}$ and call this the {\bfseries quilted Floer group} of $(Y, f, Q, r, d)$. Since each $M(P_j)$ has even minimal Chern number, this quilted Floer group admits a relative $\bb{Z}_4$-grading. Moreover, $HF^\bullet_\quilt(Y, f, Q)_{r, d}$ only depends on the bundle isomorphism type of $Q$. The following theorem, due to Gay, Wehrheim and Woodward, addresses the dependence of this group on $f$. Note that if $f$ and $f'$ are two homotopic broken circle fibrations on $Y$, then any bundle $Q$ that is $d$-compatible with $f$ is also $d$-compatible with $f'$. 
		
		\begin{theorem}\cite{GWW} \cite{WWfloer} \label{GWWtheorem}
		If $f$ and $f'$ are broken circle fibrations that are homotopic as maps $Y \rightarrow S^1$, then there is a canonical isomorphism $HF^\bullet_\quilt(Y, f, Q)_{r, d} \cong HF^\bullet_\quilt(Y, f', Q)_{r, d},$ and this isomorphism preserves the relative gradings.
		\end{theorem}

		The remainder of this section is spent describing $HF^\bullet_\quilt(Y, f, Q)_{r, d}$ in more detail. We first remark that the group $HF^\bullet_\quilt(Y, f, Q)_{r, d}$ is the cohomology of a chain complex $(CF^\bullet_\quilt, \partial_\quilt)$. We will give a precise definition of $CF^\bullet_\quilt$ and $\partial_\quilt$, beginning with the former. 
		
		For each $j$, fix a time-dependent Hamiltonian $H_j = H_{t, j}: \M(P_j) \rightarrow \bb{R}$, which we assume vanishes for $t \notin (0, 1)$. Let $\underline{H}$ denote the tuple $(H_j)_j$, and call this a {\bfseries split-type} Hamiltonian for $\underline{M}$. The time-dependent function $H_j$ lifts to a $\G_0(P_j)$-invariant time-dependent function on $\A_\fl(P_j)$. Similarly, the Hamiltonian vector field of $H_j$ on $M(P_j)$ lifts to a Hamiltonian vector field $X_{t,j}: \A_\fl(P_j) \rightarrow \Omega^1(\Sigma_j, P_j(\frak{g}))$ satisfying
			
			$$X_{t,j}({\afu}^*{\afalpha}_j) = \mathrm{Ad}({\afu}^{-1})X_{t,j}({\afalpha}_j), \indent \textrm{and} \indent d_{\afalpha_j} X_{t,j}({\afalpha}_j) = 0,$$
			for all ${\afu} \in \G_0(P_j)$ and $\afalpha_j \in \A_{\fl}(P_j)$. Note that $X_{t, j}$ vanishes when $t \notin (0, 1)$. Then the chain group $CF^\bullet_\quilt$ is freely generated over $\bb{Z}_2$ by the set of {\bfseries $\underline{H}$-perturbed generalized Lagrangian intersection points} $\underline{e}$. These are tuples
		
		$$\underline{e} = \left(\left[{\afa}_{01}\right], \left[{\afa}_{12}\right], \ldots, \left[{\afa}_{(N-1)0}\right] \right),$$
  			where ${\afa}_{j(j+1)} \in \A_{\fl}(Q_{j(j+1)})$, $\left[ {\afa}_{j(j+1)} \right]$ denotes the $\G_0(Q_{j(j+1)})$-equivalence class, and each tuple is required to satisfy the following: For each $j$, there is a map $\afalpha_j: I \rightarrow \A_\fl(P_j)$ such that  
  			
  			\begin{equation}
  			\partial_t \afalpha_j(t)  = X_{t, j}(\afalpha_j(t)), \indent \left[ \afa_{(j-1)j} \vert_{\Sigma_j} \right] = \left[ \afalpha_{j}(0) \right], \indent \textrm{and} \indent \left[ \afalpha_{j}(1) \right] = \left[ \afa_{j(j+1)} \vert_{\Sigma_j} \right],
  			\label{quiltboundarycond}
  			\end{equation}
  			where, here, the brackets denote $\G_0(P_j)$-equivalence class. We use ${\mathcal{I}}_{\underline{H}}(\underline{L}(Q))$ to denote the set of $\underline{H}$-perturbed generalized Lagrangian intersection points. This arises naturally as the critical point set of a suitably defined \emph{${\underline{H}}$-perturbed symplectic action functional}, see \cite[Section 5.2]{WWquilt}. We say ${\underline{e}} \in {\mathcal{I}}_{\underline{H}}(\underline{L}(Q))$ is {\bfseries non-degenerate} if it is a non-degenerate critical point in the Morse-theoretic sense. This is equivalent to requiring that the linearization at ${\underline{e}}$ of the defining equations (\ref{quiltboundarycond}) becomes an injective operator, modulo the linearized gauge action on the $P_j$. Unless otherwise specified, we will assume $\underline{H}$ has been chosen so that all elements of ${\mathcal{I}}_{\underline{H}}(\underline{L}(Q))$ are non-degenerate. This can always be done \cite[Proposition 5.2.1]{WWquilt} and, when this is the case, it follows that ${\mathcal{I}}_{\underline{H}}(\underline{L}(Q))$ is a finite set.

  			  Now we move on to discuss the boundary operator $\partial_\quilt$. By linearity, it suffices to define $\partial_\quilt$ in terms of its matrix coefficient $\langle \partial_\quilt \underline{e}^-,  \underline{e}^+ \rangle$, for $\underline{e}^\pm \in {\mathcal{I}}_{\underline{H}}(\underline{L}(Q))$. This can be defined concisely by saying that it is the mod-2 count of the isolated $(\underline{J}, \underline{H})$-holomorphic quilted cylinders that limit to $\underline{e}^\pm$ at $\pm \infty$. Now we unravel this. Here, $\underline{J} = (J_j)_j$ is a {\bfseries split-type} almost complex structure on $\underline{M}$, meaning that each $J_j$ is a compatible almost complex structure on $\M(P_j)$. In this paper we will always assume that each ${J_j}$ arises as follows: Fix a metric on $Y$, and let $J_j = *$ be the compatible almost complex structure on $\M(P_j)$ induced by the Hodge star on $\Sigma_j$ as in Remark \ref{surfacesremark} (a). 
  			  
  			  Next, let $\underline{H} = (H_j)_j$ be a split-type Hamiltonian as above. Then we define a $(\underline{J}, \underline{H})$-{\bfseries holomorphic quilted cylinder} to be a tuple $\underline{v} = \left(\left[{\afalpha}_0\right], \ldots, \left[{\afalpha}_{N-1}\right] \right)$ where $\left[ \cdot \right]$ denotes the $\G_0(P_j)$-equivalence class for the relevant $j$, and ${\afalpha}_j$ is a map $\mathbb{R} \times I \rightarrow \A_{\fl}(P_i)$ satisfying the following conditions:

  			$$\begin{array}{lrcl}
  			\bullet~~ \textrm{($(\underline{J}, \underline{H})$-holomorphic)} &\mathrm{proj}_{H_{{\afalpha}_i}^1} \left( \partial_s {\afalpha}_i + * \left(\partial_t {\afalpha}_i - X_{t,i}({\afalpha}_i) \right) \right) &=& 0\\
  			\bullet~~ \textrm{(Lagrangian seam)}& \left(\left[{\afalpha}_i\left(s, 1\right)\right], \left[{\afalpha}_{i+1}\left(s, 0\right)\right]\right)  &\in &\La(Q_{i(i+1)})
			\end{array}$$
			where $\mathrm{proj}_{H^1_{{\afalpha}}}$ denotes the $L^2$-orthogonal projection to the harmonic space $H^1_{{\afalpha}}$. We assume that $\underline{v}$ {\bfseries limits to $\underline{e}^\pm$ at $\pm \infty$}, in the sense that $\lim_{s \rightarrow \pm \infty} \afalpha_j(s, t) = \afalpha^\pm (t)$ for all $j$, where the $\afalpha^\pm$ are coming from $\underline{e}^\pm$ as in (\ref{quiltboundarycond}). Here we require that the convergence is uniform in $t$, and when this is the case it implies that the convergence is actually ${\mathcal{C}}^\infty$ in $t$. We say that $\underline{v}$ is {\bfseries regular} if the linearization at $\underline{v}$ of the defining equations is surjective. The desirable cases are when all $(\underline{J}, \underline{H})$-holomorphic quilted cylinders are regular. When this is the case, we say that $(\underline{J}, \underline{H})$ is {\bfseries regular (for quilted Floer theory)}. We show in Section \ref{CompatiblePerturbations} that, given any $\underline{J} = (*)_j$, there always exists a perturbation $\underline{H}$ so that $(\underline{J} , \underline{H})$ is regular. (In Lagrangian intersection Floer theory it is standard that there exists a perturbation $H$ making $(J, H)$ regular \cite{FHS}. However the existence becomes more subtle in quilted Floer theory because $\underline{H} = (H_j)$ is required to be of \emph{split-type}; see \cite{WWquilt} \cite{WWquilterr} for a similar problem.) Unless otherwise specified, we will always assume $\underline{H}$ has been chosen so that $(\underline{J},\underline{H})$ is regular. When this is the case, the set of all $(\underline{J} , \underline{H})$-holomorphic quilted cylinders limiting to $\underline{e}^\pm$ forms a smooth manifold. Finally, we say $\underline{v}$ is {\bfseries isolated} if it belongs to the zero-dimensional component of this manifold.

	Next we describe how a $(\underline{J},\underline{H})$-holomorphic quilted cylinder $\underline{v} = (\afalpha_0, \ldots, \afalpha_{N-1})$ can be represented as a connection on the bundle $\bb{R} \times Q \rightarrow \bb{R} \times Y$. In light of the Hodge isomorphism (\ref{hodegedecomp}), the \emph{$(\underline{J},\underline{H})$-holomorphic condition} is equivalent to the statement that, for all $j$,
			
			$$\partial_s {\afalpha}_j + * \left(\partial_t {\afalpha}_j - X_{t,j}({\afalpha}_j) \right)= d_{{\afalpha}_j} {\afphi}_j + *d_{{\afalpha}_j} {\afpsi}_j$$ 
			for some ${\afphi}_j, {\afpsi}_j: \mathbb{R} \times I \rightarrow \Omega^0(\Sigma_j, P_j(\frak{g}))$. In fact, ${\afphi}_j, {\afpsi}_j$ are uniquely determined by this equation since all flat connections on $P_j$ are irreducible. Furthermore, ${\afphi}_j$ and ${\afpsi}_j$ are as smooth in $s, t$ as the connection ${\afalpha}_j$. On the other hand, the \emph{Lagrangian seam} condition implies that, for each $j$, there is some path ${\afa}_{j(j+1)}: \mathbb{R} \rightarrow \A_{\fl}(Q_{j(j+1)})$  with 
 
 $$\left({\afalpha}_j\left(s, 1\right), {\afalpha}_{j+1}\left(s, 0\right)\right) = \left(\left. {\afa}_{j(j+1)}(s)\right|_{\Sigma_j} , \left. {\afa}_{j(j+1)}(s) \right|_{\Sigma_{j+1}}\right),$$ 
and $\afa_{j(j+1)}$ is the unique connection having this property, up to the action of gauge transformations on $Y_{j(j+1)}$ that restrict to the identity on the boundary. Then the irreducibility of flat connections on $Q_{j(j+1)}$ implies that there is a unique ${\afp}_{j(j+1)} : \mathbb{R} \rightarrow \Omega^0(Y_{j(j+1)}, Q_j(\frak{g}))$ such that $\partial_s {\afa}_{j(j+1)}   - d_{{\afa}_{j(j+1)} } {\afp}_{j(j+1)} \in H^1_{{\afa}_{j(j+1)} }$ and $\afp_{j(j+1)}$ agrees with $\afphi_j$ and $\afphi_{j+1}$ on the seams. We will write ${\afalpha}$ (resp. ${\afa}$) for the connection on $\Sigma_\bullet$ (resp. $Y_\bullet$) that restricts to ${\afalpha}_j$ on $\Sigma_j$ (${\afa}_{j(j+1)}$ on $Y_{j(j+1)}$). In a similar manner, we define ${\afphi}, {\afpsi}$, which are forms on $\Sigma_\bullet$, and ${\afp}$, which is a form on $Y_\bullet$. Then $\afa_{j(j+1)}$ can be chosen so that the data of ${\afalpha}, {\afphi}, {\afpsi}, {\afa}, {\afp}$ patch together to define a smooth connection $\afA$ on $\bb{R} \times Q$ with
 		
 		$$\afA\vert_{\left\{s \right\} \times Y_\bullet} = \afa(s) + p(s) \: ds, \indent \afA\vert_{\left\{(s,t) \right\} \times \Sigma_\bullet} = \afalpha(s, t) + \afphi(s, t) \: ds + \afpsi(s, t)\: dt.$$
 		We will refer to such a connection $\afA$ as a {\bfseries holomorphic curve representative}.
 		
 		 To summarize, the boundary operator $\partial_{\quilt}$ counts isolated $\G_0(P_\bullet)$-equivalence classes of $\afalpha$ satisfying

			\begin{equation}
			\begin{array}{lrcl}
  			\bullet~~ \textrm{($(\underline{J}, \underline{H})$-holomorphic)} &\partial_s {\afalpha} - d_{{\afalpha}} {\afphi} + * (\partial_t {\afalpha} - X_t(\afalpha) - d_{{\afalpha}} {\afpsi}) \vert_{\Sigma_\bullet} & =& 0\\
  			& F_{{\afalpha}} \vert_{\Sigma_\bullet} &=& 0\\
  			  			&&&\\
  			\bullet~~ \textrm{(Lagrangian seam)}& F_{{\afa}} \vert_{Y_\bullet}&=& 0
  			
  			\end{array}
  			\label{jholoconditions}
  			\end{equation}	
			with appropriate limits at $\pm \infty$, and where the $\afphi, \afpsi$ and $\afa$ are uniquely determined by these conditions.

\section{Higher rank instanton Floer theory}\label{InstantonFloerCohomology}

Throughout this section, $Y$ will denote a closed, connected, oriented 3-manifold equipped with a principal $\PU(r)$-bundle $Q \rightarrow Y$. We assume this bundle satisfies the following hypothesis.

\medskip

\begin{itemize}
\item[{\bfseries (H1)}] There is an embedding $\iota: \Sigma \rightarrow Y$ of a closed, oriented surface $\Sigma$ with the property that $d \defeq  t_2(Q)\left[ \Sigma \right] \in \bb{Z}_r$ is a generator.
\end{itemize}
\medskip
\noindent Such bundles exist if, for example, $Y$ has positive first Betti number. Note also that (H1) implies that $Y$ is necessarily \emph{not} a homology $S^3$. We also fix a subgroup $\K \subseteq \G(Q)$, which we assume consists of connected components of $\G(Q)$. In particular, this means that the identity component $\G_0(Q) \subseteq \G(Q)$ acts freely on $\K$ by left multiplication. Hence $\K / \G_0(Q)$ is well-defined and is the group $\pi_0(\K)$ of connected components of $\K$. 

We begin by formulating instanton Floer cohomology, working modulo the gauge subgroup $\K$. To ensure we have a well-defined theory, we will make various hypotheses along the way. In Section \ref{Gradings} we restrict attention to specific subgroups $\K$ and show that these hypotheses are satisfied in each case. In that section we also discuss how the various choices of $\K$ determine various gradings. Section \ref{InstantonFleorCohomologyForBrokenCircleFibrations} specializes the discussion to the case where $Y$ is a broken circle fibration.

\subsection{Instanton Floer cohomology for $G = \PU(r)$}\label{InstantonFloerCohomologyForG=PU(r)}

Let $\K$ be as above. We will use the same symbol to denote the Sobolev completion of $\K$ in $\G^{k+1, p}(Q)$. Throughout we assume $k, p$ are chosen so that $\G^{k+1, p}(Q)$ is a Lie group that acts smoothly on $\A^{k, p}(Q)$. Fix a smooth $\K$-invariant function $H: \A^{k,p}(Q) \rightarrow \bb{R}$. Then there is a smooth map ${X} = {X}_H$ from $\A^{k, p}(Q)$ into the $W^{k, p}$-completion of $\Omega^2(Y, Q(\frak{g}))$ that represents the differential of $H$ in the sense that $(dH)_{\afa} (v) = \int_Y \langle X(a) \wedge v \rangle$ for all $v \in T_{\afa} \A^{k, p}(Q)$. By the invariance of $H$, we have the identity ${X}({\afu}^*{\afa}) = \mathrm{Ad}({\afu}^{-1}) {X}({\afa})$ for ${\afu} \in \K$. We assume $V \defeq X_H$ satisfies conditions (a) and (b) from Theorem \ref{kronsmagic}, and when this is the case we call $H$ an {\bfseries instanton perturbation}. 

Next, fix a reference connection ${\afa}_0 \in \A_\fl^{k,p}(Q)$ and define the {\bfseries perturbed Chern-Simons functional} $	\CS_{H, {\afa}_0}: \A^{k,p}(Q) \longrightarrow \bb{R}$ by setting
			
				$$\begin{array}{rcl}
				\CS_{H, {\afa}_0}({\afa})&  \defeq  & -H({\afa}) + \fracd{1}{2} \intd{Y}\:  \left\langle d_{{\afa}_0} v \wedge v \right\rangle   + \fracd{1}{6} \intd{Y} \: \left\langle \left[ v \wedge v\right] \wedge v \right\rangle.
				\end{array}$$
			where $v \defeq {\afa} - {\afa}_0$. This map is smooth provided $k, p$ are in an appropriate range for Sobolev multiplication, which we also assume throughout. For example, this is the case if $k = 1, p = 2$. From now on we will drop the Sobolev exponents from the notation, unless they are relevant.

			The perturbed Chern-Simons functional only depends on ${\afa}_0$ up to an overall constant, so we will usually write $\CS_{H} \defeq \CS_{H, {\afa}_0}.$ The differential of $\CS_H$ at a connection ${\afa}$ satisfies 
			
			$$(d\CS_H)_{\afa} {\afv} = \intd{Y} \: \langle (F_{\afa} - {X}_H({\afa})) \wedge {\afv} \rangle$$ 
			for all ${\afv} \in \Omega^1(Y, Q(\frak{g}))$. It follows from (\ref{constantthings}) that there is an identity of the form 
			
			\begin{equation}\label{thignynew}
			\CS_{H}({\afu}^*{\afa}) - \CS_{H}({\afa} )  = {4 \pi^2 r^{-1}\kappa_r} \deg({\afu})
			\end{equation}		
			for all ${\afa} \in \A(Q)$ and ${\afu} \in \K$. (Our definition of degree is the negative of the one appearing in \cite{DS2}.) If $H$ is invariant under a subgroup of $\G(Q)$ larger than $\K$, then (\ref{thignynew}) holds for all $\afu$ in this larger subgroup. 
			
					\medskip

			Next, we define an instanton Floer cohomology group $HF^\bullet_{\inst}(Q)^\K$. This can be viewed as the Morse cohomology, modulo gauge transformations in $\K$, of the Chern-Simons functional. That is, when it is defined, $HF_{\inst}^\bullet(Q)^\K$ is the cohomology associated to a chain complex $(CF^\bullet_{\inst}(Q), \partial_{\inst})$. We describe this now.
			
		The chain complex will be generated by $\K$-equivalence classes of the critical points of $\CS_{H}$. These critical points are precisely the {\bfseries $H$-flat connections} ${\afa} \in \A(Q)$, which are defined by the condition $F_{\afa} = {X}({\afa}).$ We denote the set of $H$-flat connections by $\A_\fl(Q, H)$. Note that $\K$ acts on this space, and we set 
		
		$${\mathcal{I}}_{H}(Q) \defeq \A_\fl^{k,p}(Q, H) / \K.$$ 
		This is independent of the choice of $k,p$. Given an $H$-flat connection $\afa$, we denote its image in ${\mathcal{I}}_{H}(Q) $ by $\left[ \afa \right]$. 
		 		
		 		\begin{example}
		 		Suppose $H = 0$. If $\K = \G_0(Q)$, then ${\mathcal{I}}_{0}(Q) $ is the moduli space $M(Q) = \A_\fl(Q) / \G_0(Q)$ of flat connections. For general $\K$, ${\mathcal{I}}_{0}(Q)$ is a quotient of $M(Q) $ by the discrete group $\K/ \G_0(Q)$. 
				\end{example}		 		
		 		
		 		Now we set
			
			$$CF_{\inst}^\bullet (Q) \defeq \bigoplus_{\left[{\afa} \right] \in {\mathcal{I}}_{{H}}(Q) } \mathbb{Z}_2 \langle \left[ {\afa} \right] \rangle.$$
			The desirable cases are when all of the critical points ${\afa} \in \A_\fl(Q, H)$ are {\bfseries non-degenerate} in the sense that they are non-degenerate critical points of $\CS$ in the Morse-theoretic sense, taken modulo gauge. Fixing a metric $g$ on $Y$, this is equivalent to requiring that the extended Hessian
			
			\begin{equation}\label{extendedhess}
			\D_{\afa} \defeq \left(\begin{array}{cc}
			*d_{\afa} - *d{X}_{\afa} & -d_{\afa}\\
			-d_{\afa}^* & 0 
			\end{array}\right)
			\end{equation}			
			is non-degenerate as an operator $W^{k,p} \rightarrow W^{k-1,p}$. Here, the short-hand $W^{j,p}$ denotes the obvious Sobolev completion of the vector space $\Omega^1(Y, Q(\frak{g})) \oplus \Omega^0(Y, Q(\frak{g})).$ There always exists a perturbation $H$ such that all $H$-flat connections are non-degenerate \cite[Chapter 5]{Donfloer}. We discuss this further in Section \ref{CompatiblePerturbations}. However, for the remainder of this section, we take this as a hypothesis.
			
\medskip			
			
				\begin{itemize}
					\item[{\bfseries (H2)}] The instanton perturbation $H$ has been chosen so that all $H$-flat connections are non-degenerate.  
				\end{itemize}			
				\medskip
			It follows immediately from (H2) that all $H$-flat connections $a$ are irreducible.

			One consequence of Hypothesis (H2) is that the chain complex $CF_{\inst}^\bullet (Q) $ admits a relative grading. To define this grading, set $n_{\K} \defeq \inf\left\{ \deg(\afu) > 0 \: \vert \: \afu \in \K \right\}.$ By convention, if all elements of $\K$ have degree zero, then $n_{\K} = \infty$.

		 \begin{proposition}\label{gradinglemma}
		 	Assume (H1-2). Then there is a relative $\bb{Z}$-grading 
		 	
		 	$$\mu_\inst: \A_\fl(Q, H) \times \A_\fl(Q, H) \longrightarrow \bb{Z},$$ 
		 	and this satisfies
		 	
		 	\begin{equation}\label{mod4grading1000}
		 	\mu_\inst({\afa}_0, {\afu}^*{\afa} ) - \mu_\inst({\afa}_0, {\afa}) = 4 \deg({\afu})
		 	\end{equation}
		 	for all ${\afa}_0, {\afa} \in \A_\fl(Q, H) $ and ${\afu} \in \K$. In particular, this induces a relative $\bb{Z}_{4n_\K}$-grading on ${\mathcal{I}}_{H}(Q)$, and hence on $ CF_{\inst}^\bullet (Q) $, where $\bb{Z}_{4n_{\K}} \defeq \bb{Z}$ if $n_{\K} = \infty$.
		 	\end{proposition}
				
			We defer the proof of this proposition to the end of this section. Using this relative $\bb{Z}_{4n_{\K}}$-grading, we have a decomposition 
			
			$$CF^\bullet_{\inst}(Q) = \bigoplus_{k \in \bb{Z}_{4n_\K}} CF^k_\inst(Q) , \indent \indent CF^k_\inst (Q) \defeq \bigoplus\bb{Z}_2 \langle \left[a \right] \rangle,$$			
			where the sum on the right is over all $\left[{\afa}\right] \in {\mathcal{I}}_{{H}}(Q)$ with $\mu_{\inst}(\left[{\afa}_0\right], \left[{\afa}\right]) = k$, and $\left[{\afa}_0 \right] \in  {\mathcal{I}}_{{H}}(Q)$ is some fixed reference connection. It follows from Theorem \ref{kronsmagic} that each $CF^k_\inst (Q)$ is finitely generated.

			It is immediate from (\ref{mod4grading1000}) that each gauge transformation with non-zero degree acts freely on $\A_{\fl}(Q, H)$. However, it will be convenient to know that the subgroup $\K$ acts freely on the space of $H$-flat connections. We take this as a hypothesis as well. 
			
			\medskip
			
			\begin{itemize}
					\item[{\bfseries (H3)}] The subgroup $\K$ acts freely on $\A_\fl(Q, H)$.  
				\end{itemize}

\medskip
			
			The next step is to introduce the boundary operator 
			
			$$\partial_{\inst}: CF^k_\inst (Q) \longrightarrow CF^{k+1}_\inst (Q).$$ 
			It will take some time to develop the machinery necessary to define this explicitly. However, when we are done it will be given by a mod-2 count of isolated negative gradient trajectories of the perturbed Chern-Simons functional. These gradient trajectories are solutions ${\afa}: \mathbb{R} \rightarrow \A(Q)$ to the equation
			
			\begin{equation}
			\partial_s {\afa}  = -* (F_{\afa}- {X}({\afa})).
			\label{asd1}
			\end{equation}
			This equation is plainly invariant under the action of $\K$. Note that we may alternatively view any path $s \mapsto {\afa}(s)$ of connections as defining a single connection ${\afA} = {\afa}(s)$ on the bundle $\mathbb{R} \times Q\rightarrow \mathbb{R} \times Y$. Conversely, every connection on $\mathbb{R} \times Q$ has the form ${\afA} = {\afa}(s) + {\afp}(s) \:ds,$ where ${\afa}: \mathbb{R} \rightarrow \A(Q)$ and ${\afp}: \mathbb{R} \rightarrow \Omega^0(Y, Q(\frak{g}))$. The curvature decomposes into components as $F_{\afA} = F_{\afa} + ds \wedge  (\partial_s {\afa} - d_{\afa} {\afp}).$ We say that $\afA$ is {\bfseries $(g,H)$-ASD} or a {\bfseries $(g,H)$-instanton} if it satisfies
			
			\begin{equation}
			\partial_s {\afa} - d_{\afa} {\afp} + *(F_{\afa}- {X}({\afa})) = 0.
			\label{asd2}
			\end{equation}
			When $H = 0$, equation (\ref{asd2}) is exactly the ASD equation for the connection $\afA$ (see Section \ref{GaugeTheory}; the Hodge star is the one on $\bb{R} \times Y$ coming from the product metric $ds^2 + g$). Our immediate interest in (\ref{asd2}) is that it reduces to (\ref{asd1}) when ${\afA}$ is in temporal gauge $\afp = 0$. It follows that solutions to (\ref{asd1}) modulo $\mathrm{Maps}(\bb{R}, \G_0(Q))$ are identical to solutions of (\ref{asd2}) modulo $\G_0(\mathbb{R} \times Q)$, since every connection $\afA$ can be put into temporal gauge by an element of $\G_0(\bb{R} \times Q)$.

 The following are equivalent for a smooth $(g,H)$-instanton ${\afA}= {\afa}(s) + {\afp}(s)$ on $\mathbb{R} \times Q$:
			
			\begin{enumerate}
					\item[(i)] The connection ${\afA}$ has finite {\bfseries $H$-energy}:
					
							$$\YM_{H}({\afA}) \defeq \fracd{1}{2} \intd{\bb{R} \times Y} \left| F_{{\afa}(s)} - {X}({\afa}(s)) \right|^2 + \left| \partial_s {\afa}(s) - d_{{\afa}(s)}{\afp}(s) \right|^2 < \infty.$$
					\item[(ii)] The $s$-derivative $\partial_s {\afa}(s) - d_{{\afa}(s)}{\afp}(s)$ decays exponentially to zero as $s$ approaches $\pm \infty$:
					
							$$ \left\| \partial_s {\afa}(s) - d_{{\afa}(s)}{\afp}(s)\right\|_{L^2(Y)}  \leq Ce^{-\kappa \vert s \vert},$$
							for some constants $C, \kappa > 0$.
					\item[(iii)] The connection ${\afA}$ converges to $H$-flat connections ${\afa}^\pm \in \A_\fl(Q, H)$ at $\pm \infty$:
					
							$$\limd{s \rightarrow \pm \infty} {\afa}(s) = {\afa}^\pm, \indent \limd{s \rightarrow \pm \infty} {\afp}(s) = 0;$$
here the convergence is in ${\mathcal{C}}^\infty$ on $Y$. 
			\end{enumerate}
	When the last condition holds we will say that $\afA$ {\bfseries limits to $\afa^\pm$ at $\pm \infty$}. The equivalence of (i) and (iii) above imply that the group of gauge transformation in $\K$ act freely on the space of $(g, H)$-instantons. To state this precisely, we introduce the group
		
		$$\mathrm{Maps}_{\pm \infty}(\bb{R}, \K) \subseteq \mathrm{Maps}(\bb{R} , \G(Q)) = \G(\bb{R} \times Q)$$ 
		of smooth gauge transformation $\afU = \afu(\cdot)$ on $\bb{R} \times Q$ that converge at $\pm \infty$ to elements of $\K$. We assume the convergence at $\pm \infty$ is at least in ${\mathcal{C}}^1$ on $Y$.

		\begin{proposition}\label{freeactiononinstantons}
			Assume (H1-3). Let ${\afA} \in \A(\bb{R} \times Q)$ be a smooth $(g, H)$-instanton on $\bb{R} \times Q$ that has finite $H$-energy. Suppose ${\afU} \in \mathrm{Maps}_{\pm \infty}(\bb{R}, \K)$ is a smooth gauge transformation such that ${\afU}^* {\afA} = {\afA}.$ Then ${\afU} = e$ is the identity.
		\end{proposition}

		\begin{proof}	
		To prove the proposition, it suffices to prove the following:

				\begin{itemize}
					\item[(i)] Assume (H1). Fix ${\afa}_\infty  \in \A(Q)$ and suppose this is irreducible. If ${\afA} \in \A(\bb{R} \times Q)$ is any connection limiting to ${\afa}_\infty$ at $\infty$ (or $- \infty$), then ${\afA}$ is also irreducible. 
					\item[(ii)] Suppose $H$ satisfies (H2), and $\K$ satisfies (H3). Let $\afa_\infty$ and $\afA$ be as in (i). If ${\afa}_\infty$ is $H$-flat, then the stabilizer of ${\afA}$ in $\mathrm{Maps}_{\pm \infty}(\bb{R}, \K)$ is trivial. 
				\end{itemize}

			Note that ${\afA}$ is irreducible if and only if there is some ${\afU} \in \G(\bb{R} \times Q)$ such that ${\afU}^* \afA$ is irreducible. In particular, we may assume $\afA = {\afa}(\cdot)$ is in temporal gauge. First we prove (i). Suppose $d_{\afA} {\afR} = 0$ for some 0-form ${\afR}$ on $\bb{R} \times Y$. We want to show ${\afR} = 0$. Write ${\afR} = {\afr}(\cdot )$ as a path of 0-forms on $Y$. Then writing $d_{\afA} \afR = 0$ in components gives 
			
			$$\partial_s {\afr}(s) = 0, \indent d_{\afa(s)} \afr(s) = 0$$
			for each $s \in \bb{R}$. The first condition says $\afr(s) = \afr_0$ is constant. Taking the limit as $s$ approaches $\infty$ implies that $d_{\afa_\infty} \afr_0 = 0$ and so $\afr_0 = 0$, as desired.
			
			To prove (ii), let $\afU = \afu(\cdot) \in \mathrm{Maps}_{\pm \infty}(\bb{R}, \K)$ and suppose this fixes $\afA = \afa(\cdot)$. This implies $\afu(s)$ fixes $\afa(s)$ for each $s \in \bb{R}$. By the assumptions on $\afU$, the limit $\lim_{s \rightarrow \infty} \afu(s) = \afu_\infty$ exists, and it is immediate that $\afu_\infty$ fixes $\afa_\infty$. Since $\afa_\infty$ is an $H$-flat connection, it follows from (H3) that $\afu_\infty = e$. In particular, $\afu(s) \in \G_0(Q)$ lies in the identity component for all $s$. By Lemma \ref{14days} below, it follows that $\afu(s) = e \in \G(Q)$ is the identity for all sufficiently large $s$. 
			
			To prove that $\afu(s) = e$ for all $s \in \bb{R}$, consider the set of $s \in \bb{R}$ such that $\afu(s) = e$. This set is clearly closed, and we just saw that it is non-empty. It suffices to show this set is also open. It follows from (i) that $\afa(s)$ is irreducible for each $s \in \bb{R}$. In general, a connection $\afa \in \A^{k,p}(Q)$ is irreducible if and only if it satisfies a bound of the form
			
			\begin{equation}\label{ppppd}
			\Vert \afr \Vert_{L^p(Y)} \leq C \Vert d_{\afa} \afr \Vert_{L^p(Y)}
			\end{equation}
			for all 0-forms ${\afr} \in \Omega^0(Y, Q(\frak{g}))$; this bound continues to hold for all connections near an irreducible $\afa$. Suppose there is some $s_0 \in \bb{R}$ such that $\afu(s_0) = e$. We want to show $\afu(s) = e$ for all $s$ near $s_0$. Using the exponential map for the gauge group, we can write $\afu(s) = \exp(\afr(s)),$ at least when $\vert s - s_0 \vert$ is small. Here, $\afr(s)$ is a path of 0-forms on $Y$ with values in $Q(\frak{g})$. Now differentiate the identity $\afu(s)^* \afa(s) = \afa(s)$ to get $d_{\afa(s)}  \afr(s) = 0.$ By (\ref{ppppd}), this implies $\afr(s) = 0$, and hence $\afu(s) = e$, for all $s$ sufficiently close to $s_0$. 
			\end{proof}

			Fix two $H$-flat connections ${\afa}^\pm \in \A_{\fl}(Q, H)$. Our immediate goal now is to define the moduli space ${\Mtwo}(\left[{\afa}^-\right], \left[{\afa}^+\right]) $ of instanton trajectories from $\left[ \afa^-\right]$ to $\left[ \afa^+ \right]$. This is the primary geometric object used to define the Floer boundary operator $\partial_\inst$. Intuitively, ${\Mtwo}(\left[{\afa}^-\right], \left[{\afa}^+\right]) $ can be viewed as the space
						
			\begin{equation}\label{wantmod}
			\left. \left\{{\afA}\in \A(\bb{R} \times Q)\: \left| \: (\ref{asd2}), \: \limd{s \rightarrow \pm \infty}\afA\vert_{\left\{s \right\} \times Y} \in \left[ {\afa}^\pm \right]\right. \right\} \right\slash \mathrm{Maps}_{\pm \infty}(\bb{R}, \K)
			\end{equation}			
			where $\A(\bb{R} \times Q)$ denotes the space of \emph{smooth} connections, and $\mathrm{Maps}_{\pm \infty}(\bb{R}, \K)$ is as defined above Proposition \ref{freeactiononinstantons}. However, to specify a smooth structure on the space ${\Mtwo}(\left[{\afa}^-\right], \left[{\afa}^+\right]) $ we need to work with suitable function theoretic completions of these spaces. Due to the non-compactness of $\bb{R} \times Y$, this is not as straightforward as one might hope. As a consequence, we adopt the following more circuitous approach. 
			
			Fix representatives $\afa^\pm \in \left[ \afa^\pm \right]$. Also fix representatives ${\mathcal{U}} = \left\{ \afu_0, \afu_1, \ldots \right\} \subset \K$, one for each element of $\K / \G_0(Q)$. That is, ${\mathcal{U}}$ is chosen so that the inclusion map ${\mathcal{U}} \hookrightarrow \K$ descends to a bijection 
			
			\begin{equation}\label{kacts}
			{\mathcal{U}} \cong \K / \G_0(Q).
			\end{equation}
			The bijection (\ref{kacts}) imposes a group structure on ${\mathcal{U}}$. (The gauge theoretic product of two elements of ${\mathcal{U}}$ may not be in ${\mathcal{U}}$, but this product is always \emph{homotopic} to a unique element of ${\mathcal{U}}$; this is the group structure determined by (\ref{kacts}).) Next, recall the Banach manifolds $\A^{k, p}(\afa, \afa') \subset \A^{k, p}(\bb{R} \times Q)$ and $\G_c^{k+1, p}(\bb{R} \times Q) \subset \G_{loc}^{k+1, p}(\bb{R} \times Q)$ from (\ref{a0}) and (\ref{g0}), respectively. For $\afa, \afa' \in \A_\fl(Q, H)$, define the following auxiliary space

			$${{\mathcal{N}}}(\afa, \afa') \defeq \left. \left\{\Big.{\afA} \in \A^{k, p}\left(\afa, \afa'\right)\: \Big|\: \left(\ref{asd2}\right) \: \right\} \right\slash \G_c^{k+1, p}(\bb{R} \times Q).$$
			We will denote the elements of the space by $\left[ \afA \right]_{\G_c}$. Then the group ${\mathcal{U}}$ acts on the disjoint union 
			
			$$	\bigsqcup_{\afu_i, \afu_j \in {\mathcal{U}}} \: {{{\mathcal{N}}}}(\afu_i^*\afa^- , \afu_j^* \afa^+)$$ 
			via the isomorphism (\ref{kacts}), where we are viewing elements of ${\mathcal{U}}$ as gauge transformations on $\bb{R} \times Q$ that are constant in the $\bb{R}$-direction. Moreover, this action is free by Proposition \ref{freeactiononinstantons}. Finally, we set
			
			$${\Mtwo}(\left[{\afa}^-\right], \left[{\afa}^+\right]) \defeq \left. \left(\bigsqcup_{\afu_i, \afu_j \in {\mathcal{U}}} \: {{{\mathcal{N}}}}(\afu_i^*\afa^- , \afu_j^* \afa^+) \right) \right\slash {\mathcal{U}} .$$
		At this point, there are three matters that need to be addressed: 
		
		\begin{itemize}
			\item[(i)] The dependence of ${\Mtwo}(\left[{\afa}^-\right], \left[{\afa}^+\right])$ on the choices of representatives $\afa^\pm$ and ${\mathcal{U}}$;  
			\item[(ii)] The sense in which ${\Mtwo}(\left[{\afa}^-\right], \left[{\afa}^+\right]) $ is a smooth manifold.
		\end{itemize}
		We begin with (i). Suppose $\afb^\pm \in \left[ \afa^\pm \right]$ are two other choices of representatives. Likewise, suppose ${\mathcal{V}} = \left\{ \afv_k \right\}_k$ is another set of representatives of $\K/ \G_0(Q)$. We want to show that the spaces
		
		$$\left. \left(\bigsqcup_{i , j} \: {{{\mathcal{N}}}}(\afu_i^* \afa^- , \afu_j^* \afa^+) \right) \right\slash {\mathcal{U}}, \indent \textrm{and} \indent \left. \left(\bigsqcup_{k , l} \: {{{\mathcal{N}}}}(\afv_k^*\afb^- , \afv_l^* \afb^+) \right) \right\slash {\mathcal{V}}$$		
		are canonically identified. By assumption, we can write $\afb^\pm = \afu_\pm^* \afa^\pm$ for some $\afu_\pm \in \K$. Note that the dependence of the space $\sqcup_{i, j} {{{\mathcal{N}}}}(\afv_k^*\afb^- , \afv_l^* \afb^+)$ on the representatives is just through the set $\left\{ \afv_k^* \afb^\pm \right\}_k.$ Since $\afv_k^* \afb^-  = (\afu_- \afv_k)^* \afa^-$, by replacing $\afv_k$ with $\afu_- \afv_k$, we may suppose $\afu_-  = e$ is the identity. Using the actions of ${\mathcal{U}}$ and ${\mathcal{V}}$, it suffices to show that, for each $j$, there is some $l$ for which the spaces $ {{{\mathcal{N}}}}(\afa^- , \afu_j^* \afa^+),$ and $ {{{\mathcal{N}}}}(\afa^- , \afv_l^* \afb^+)$ are canonically identified. To see this, suppose we are given some $\afu_j \in {\mathcal{U}}$. Then there is a unique $\afv_l \in {\mathcal{V}}$ such that ${\afu}_j$ and $\afu_+ {\afv}_l$ are in the same component of $\K$. Let $w: \bb{R} \rightarrow \G_0(Q)$ be a path with $w(s) = e$ for $s < -1$ and $w(s) = \afu_j^{-1} \afu_+\afv_l$ for $s > 1$. This determines a gauge transformation $W$ on $\bb{R} \times Q$ in the obvious way, and hence a bijection $ {{{\mathcal{N}}}}(\afa^- , \afu_j^* \afa^+) \cong  {{{\mathcal{N}}}}(\afa^- , \afv_l^* \afb^+)$ defined by sending $ \left[ \afA \right]_{\G_c} $ to $\left[ W^* \afA \right]_{\G_c}$. This map is independent of the choice of path $w$ since any two choices differ by an element of $\G_c^{k+1, p}(\bb{R} \times Q)$, and this group acts by the identity. This resolves (i). 

Now we discuss (ii). By assumption, ${H}$ has been chosen so all ${H}$-flat connections are non-degenerate, and it follows that the defining equations for each ${{{\mathcal{N}}}}(\afu_i^* \afa^- , \afu_j^* \afa^+)$ are Fredholm. Indeed, their linearization at $\left[ {\afA}\right]_{\G_c} \in {{{\mathcal{N}}}}(\afu_i^* \afa^- , \afu_j^* \afa^+)$ is an operator $\D_{{\afA}, g, {H}}$ obtained by using ${H}$ to perturb the Fredholm operator $d_{\afA}^+ \oplus d_{\afA}^*$. This perturbation determined by $H$ is compact (at least when restricted to compact subsets of $\bb{R} \times Y$), so many of the Fredholm properties of the operators $\D_{{\afA}, g, {H}}$ and $d_{\afA}^+ \oplus d_{\afA}^*$ are the same. For example, they have the same index. (This relies on the fact that we considered the instanton equations on the space $\A^{k,p}(a, a')$, rather than, say, the much larger space $\A^{k, p}_{loc}(\bb{R} \times Q)$.)

			When $\D_{{\afA}, g, {H}}$ is onto, there is a neighborhood of $\left[{\afA}\right]_{\G_c}$ in ${{{\mathcal{N}}}}(\afu_i^* \afa^- , \afu_j^* \afa^+)$ that is a smooth manifold of dimension $\mathrm{Ind}(\D_{{\afA}, g, {H}}) = \mu_{\inst}(\afu_i^* \afa^-, \afu_j^* \afa^+)$; see the proof of Proposition \ref{gradinglemma}. We say that the pair $(g, {H})$ is {\bfseries regular (for instanton Floer theory)}, if all ${H}$-flat connections are non-degenerate, and if for all $i, j$, all $\left[{\afa}^\pm \right] \in {\mathcal{I}}_{{H}}(Q)$, and all $\left[ {\afA} \right]_{\G_c} \in {{{\mathcal{N}}}}(\afu_i^* \afa^- , \afu_j^* \afa^+)$, the operator $\D_{{\afA}, g, {H}}$ is onto. In Section \ref{CompatiblePerturbations} we will show that, given any metric $g$, there always exists some instanton perturbation $H$ so that $(g, H)$ is regular. For the remainder of this section we assume that this is the case. 
			
			\medskip
			
			\begin{itemize}
					\item[{\bfseries (H4)}] The pair $(g, H)$ is regular for instanton Floer theory.
			\end{itemize}
			\medskip
			\noindent 
Assuming (H4), it follows that $\sqcup_{i, j} {{{\mathcal{N}}}}(\afu_i^* \afa^- , \afu_j^* \afa^+)$ is a smooth manifold. By Proposition \ref{freeactiononinstantons}, the group ${\mathcal{U}} = \K/ \G_0(Q)$ acts freely on this space and so the quotient ${\Mtwo}(\left[{\afa}^-\right], \left[{\afa}^+\right])$ is indeed a smooth manifold with local dimension given by the mod $4n_\K$-reduction of $\mu_{\inst}(\afa^-,  \afa^+)$. This resolves (ii).

				\medskip			
			
			The space ${\Mtwo}(\left[{\afa}^-\right], \left[{\afa}^+\right])$ admits a free action of $\mathbb{R}$ by translation, and we set
			
			$$\widehat{\Mtwo}(\left[{\afa}^-\right], \left[{\afa}^+\right]) \defeq {\Mtwo}(\left[{\afa}^-\right], \left[{\afa}^+\right]) / \mathbb{R}.$$
			These are the {\bfseries (unparametrized) instanton trajectories}. We denote by $\widehat{\Mtwo}_0(\left[{\afa}^-\right], \left[{\afa}^+\right]) $ the zero-dimensional component of $\widehat{\Mtwo}(\left[{\afa}^-\right], \left[{\afa}^+\right]) $; this may be empty. Theorem \ref{kronsmagic} implies that $\widehat{\Mtwo}_0(\left[{\afa}^-\right], \left[{\afa}^+\right])$ is a finite set. We define $\#_Q({\afa}^-, {\afa}^+)$ to be the mod-2 count of its elements. Finally, we define $\partial_{\inst}: CF_{\inst}^\bullet(Q) \rightarrow  CF_{\inst}^\bullet(Q)$ to be the degree 1 operator given by
			
			$$\partial_{\inst} \langle \left[{\afa}^-\right] \rangle = \sum  \#_Q({\afa}^-, {\afa}^+) \langle \left[{\afa}^+\right] \rangle,$$
			where the sum is over all $\left[{\afa}^+\right] \in {\mathcal{I}}_H(Q)$ with $\mu_\inst({\afa}^-, {\afa}^+) = 1$. Now we can state Floer's main theorem.

			\begin{theorem}\label{instmod2}	
			Assume (H1-4). Then $\partial_{\inst}^2 = 0$, and so
			
			$${HF}_{\inst}^{\bullet}(Q)^\K \defeq \fracd{\ker \partial_{\inst}}{\mathrm{im}\: \partial_{\inst}}$$
			is well-defined. This abelian group inherits a relative $\bb{Z}_{4n_\K}$-grading from the grading on $CF^\bullet_\inst(Q)$. Furthermore, ${HF}_{\inst}^\bullet (Q)^\K$ is independent of the choice of regular pair $(g, {H})$, up to isomorphism of relatively $\bb{Z}_{4n_\K}$-graded abelian groups.
			\end{theorem}

			The proof of Theorem \ref{instmod2} follows essentially as in \cite{Fl1}, where Floer considers the trivial $\SU(2)$-bundle over a homology 3-sphere. The key technical device needed to carry Floer's proof to our setting are hypotheses (H1) and (H3), which essentially say that there are no reducible flat connections for our bundles (in Floer's setting, the only reducible flat connection is the trivial connection, which can be easily avoided). See also \cite{Flinst}, \cite{DS} and \cite[Chapter 5]{Donfloer}.

					\begin{proof}[Proof of Proposition \ref{gradinglemma}]
			This proposition is well-known in the case $\PU(2) = \SO(3)$ (see \cite{Flinst}, \cite{DS2}, \cite{BD}), and essentially the same proof carries over to the more general case of $\PU(r)$. We include the details for convenience. We note that the perturbation $H$ only serves to ensure that the $H$-flat connections are non-degenerate. To simplify notation we assume $H= 0$ (see also Remark \ref{rem1}). 
			
			Let $\D_{\afA}$ be the operator 
			
			$$\D_{\afA} = d_{\afA}^+ \oplus d_{\afA}^*: \Omega^1(\bb{R} \times Q, Q(\frak{g})) \longrightarrow \Omega^+(\bb{R} \times Q, Q(\frak{g})) \oplus \Omega^0(\bb{R} \times Q, Q(\frak{g})).$$		
			Here ${\afA}\in \A(\bb{R} \times Q)$ is any connection that converges to non-degenerate flat connections $a^\pm$ at $\pm \infty$, and $d_{\afA}^+$ is the composition 
			
			$$d_{\afA}^+: \Omega^1 \stackrel{d_{\afA}}{\longrightarrow} \Omega^2 \longrightarrow \Omega^+,$$ 
			where the second arrow is the $L^2$-orthogonal projection to the space of anti-self dual 2-forms. (It is important here that ${\afA}$ is a connection on a pull-back bundle $\bb{R} \times Q$ and not on an arbitrary bundle over $\bb{R} \times Y$.) The operator $\D_{\afA}$ is the linearization of the ASD operator ${\afA} \mapsto F^+_{\afA}$ coupled with the Coulomb gauge fixing operator ${\afA} \mapsto d_{\afA}^*({\afA}_0 - {\afA})$. The assumption on the non-degeneracy of the $\afa^\pm$ implies that, after passing to suitable Banach space completions, $\D_{\afA}$ becomes a Fredholm operator. In particular, its index $\mathrm{Ind}(\D_{\afA}, {\afa}^-, {\afa}^+)$ is well-defined. 

\medskip

\noindent \emph{Claim: The index $\mathrm{Ind}(\D_{\afA}, {\afa}^-, {\afa}^+)$ is independent of the choice of connection ${\afA}$ on $\bb{R} \times Q$, as well as the choice of metric on $Y$ used to define the splittings above.}

\medskip

We prove independence of the choice of connection; independence of the choice of metric is similar. As a first case, suppose that ${\afA}$ and ${\afB}$ are connections that agree off of a compact set $K \subset \bb{R} \times Y$. Then the difference $\D_{\afA} - \D_{\afB}$ is an operator given by multiplication by a section of a bundle over $\bb{R} \times Y$, and this section has support in $K$. It follows from the compact embedding $W^{1,2}(K) \hookrightarrow L^2(K)$ that this difference is a compact operator. The index of a Fredholm operator is unchanged under compact perturbations, so $\mathrm{Ind}(\D_{\afA}, {\afa}^-, {\afa}^+) = \mathrm{Ind}(\D_{\afB}, {\afa}^-, {\afa}^+)$ in this case. 
 
 More generally, if ${\afA}$ and $ {\afB}$ do not agree on any compact set, then we can reduce to the previous case as follows. Since these connections have the same limits at $\pm \infty$, given any $\eps > 0$ one can find a compactly supported 1-form ${\afV}$ so that $\D_{{\afA}} - \D_{{\afB} + {\afV}}$ has norm less than $\eps$ (multiply the difference ${\afA} - {\afB}$ by a bump function with suitably large support). The previous case shows that the operators $\D_{{\afB} + {\afV}}$ and $\D_{{\afB}}$ have the same index. To see that $\D_{{\afB} + {\afV}}$ has the same index as $\D_{{\afA}}$, use the fact that if $F: X \rightarrow Y$ is a Fredholm operator, then there is some $\eps > 0$ such that if $L: X \rightarrow Y$ is any bounded linear operator with norm less than $\eps$, then $F +L$ is Fredholm and has the same index as $F$. This proves the claim.
 
 \medskip

			Set
			
			$$\mu_\inst\left({\afa}^- , {\afa}^+  \right) \defeq  \mathrm{Ind}(\D_{\afA}, {\afa}^-, {\afa}^+).$$
		Then $\mu_\inst$ is additive in each component, and so is in fact a relative $\bb{Z}$-grading. It therefore suffices to show $\mathrm{Ind}(\D_{\afA}, {\afa}, {\afu}^*{\afa})  = 4 \deg({\afu})$, where ${\afa} \in \A_\fl(Q)$ is a flat connection, ${\afu} \in \K$ is a gauge transformation, and ${\afA}$ limits to ${\afa}$ and ${\afu}^*{\afa}$.

	\begin{remark}\label{rem1}
			(a) Recall we have only assumed that the perturbation $H$ is invariant under the group $\K$, and not necessarily under the full group $\G(Q)$. In particular, if ${\afu}$ is a gauge transformation not in $\K$, it could be the case that ${\afu}^*{\afa}$ is \emph{not} $H$-flat, even if ${\afa} \in \A_\fl(Q, H)$. Hence, the extended Hessian at ${\afu}^*{\afa}$ may be degenerate and this could result in the failure of $\D_{\afA}$ to be Fredholm. This being said, the proof we give below will show that $\mathrm{Ind}(\D_{\afA}, {\afa}, {\afu}^*{\afa})  = 4 \deg({\afu})$ for any ${\afu}$ in the larger group $\G(Q)$, provided one knows that the operator $\D_{{\afA}}$ is Fredholm when ${\afA}$ limits to ${\afa}$ and ${\afu}^*{\afa}$. 
			
			\medskip
			
			(b) In general, suppose $G$ is a simple Lie group and consider a $G$-bundle $R \rightarrow X$ over a closed 4-manifold. Then the index of the linearized ASD operator at a connection ${\afB}$ on $R$ can by computed using the Atiyah-Singer index formula to give
		
		\begin{equation}\label{asindex}
		\mathrm{Ind}(\D_{\afB}) = c(G)\kappa(R) - \dim(G)\left(1 - b_1(X) + b^+(X) \right).
		\end{equation}
		See, for example, \cite[Equation 7.1.3]{DK}. Here $\kappa(R)$ is a characteristic number for $R$, and $c(G)$ is a normalizing constant depending only on $G$. For example, when $G = \SU(r)$, one typically takes $\kappa(R) = c_2(R)\left[X \right]$ to be the 2nd Chern number. Then one can show $c(\SU(r)) = 4r$. 
		
		When $G = \PU(r)$, we will have $\kappa(R) = q_4(R)\left[X \right]$ is the generalized Pontryagin number from Section \ref{TopolicaAspectsOfPrincipalPU(r)-Bundles}. That this is the correct choice of $\kappa$ can be seen as follows: The connection $\afB$ can be equivalently viewed as a connection on the complexified adjoint bundle $R(\frak{g})_\bb{C}$. The structure group of this bundle reduces to $\SU(r)$ since $c_1 =0$ for complexified bundles, so by the previous paragraph we have $\kappa(R(\frak{g})_\bb{C}) = c_2(R(\frak{g})_\bb{C}) \left[ X \right] = q_4(R) \left[X \right]$, where the second equality is the definition of $q_4$. We will compute $c(\PU(r))$ for this choice of characteristic number in a moment (essentially the same argument can be used to show $c(\SU(r)) = 4r$). 
		
		We also note that when ${\afB}$ is an irreducible ASD connection, the value (\ref{asindex}) recovers the dimension near ${\afB}$ of the moduli space of irreducible ASD connections on $R$.
			\end{remark}

		Now we return to the situation where ${\afA}$ is a connection on $\bb{R} \times Q$ with flat limits ${\afa}$ and ${\afu}^*{\afa}$. We may view ${\afA}$ as being a connection on the bundle $Q_{\afu} \defeq I \times Q / (0, \afu(q)) \sim (1, q)$ over the closed manifold $S^1 \times Y$. Apply Remark \ref{rem1} (b) to the case
		
		$$G = \PU(r),\indent X = S^1 \times Y, \indent R = Q_{\afu}, \indent \mathrm{and} \indent {\afB} = {\afA}.$$
		Then by (\ref{asindex}) and the definition of the degree, we have
		
		$$\mathrm{Ind(\D_{\afA})} = c(\PU(r))\kappa(Q_{\afu}) = 2c(\PU(r)) \deg({\afu})$$
		since the term involving the Betti numbers vanishes for $X = S^1 \times Y$. We will be done if we can show $c(\PU(r)) = 2$. 
		
		To do this, we use (\ref{asindex}) again, but with $X = S^4$. Before specifying the $\PU(r)$-bundle $R$, we first note that the $\SU(r)$-bundles over $S^4$ are classified by their second Chern class. Let $R'$ denote the $\SU(r)$-bundle with $c_2(R')\left[S^4\right] = 2r$. Then we take $R \rightarrow S^4$ to be the adjoint bundle associated to $R'$. So $R$ is a $\PU(r)$-bundle with $q_4(R) \defeq c_2(\End(R')) =  2rc_2(R') = 4r^2.$ Atiyah-Hitchin-Singer show that $R'$ always admits an irreducible ASD connection ${\afB}$ \cite[Theorem 8.4]{AHS}. Moreover, in \cite[Table 8.1]{AHS}, the authors compute the dimension of the moduli space of irreducible ASD connections on the $\SU(r)$-bundle $R'$ to be $2(4r^2) - r^2 +1$. This moduli space has the same dimension as the moduli space of irreducible ASD connections on the associated $\PU(r)$-bundle $R$, and so it recovers the index. Combining this with (\ref{asindex}) gives 
		
		$$2(4r^2) - r^2 +1  =  \mathrm{Ind}(\D_{\afB}) = c(\PU(r))4r^2 - r^2 + 1$$
		and so it follows that $c(\PU(r)) = 2$.
		\end{proof}

			\subsection{Gradings} \label{Gradings}

				We continue to assume $Q \rightarrow Y$ is a bundle satisfying hypothesis (H1). In the previous section we defined a cohomology group $HF^\bullet_\inst(Q)^\K$ associated to a subgroup $\K \subseteq \G(Q)$ satisfying hypothesis (H3). In this section we work with the specific cases where $\K$ is $\G_0(Q), \ker \eta$ and $\G(\Sigma)$; here $\eta: \G(Q) \rightarrow H^1(Y, \bb{Z}_r)$ is the parity operator, and $\G(\Sigma)$ is the subgroup generated by $\G_0(Q)$ and the degree 1 gauge transformation from Proposition \ref{degprop} (b). In each case we will show that $\K$ satisfies hypothesis (H3), and that we obtain a relatively $\bb{Z}, \bb{Z}_{4r}$ and $\bb{Z}_4$-graded Floer theory in the respective cases (the latter case recovers the $\bb{Z}_4$-graded Floer group that is more common in the literature \cite{Flinst} \cite{BD} \cite{DS2}). Along the way, we also compare the associated Floer cohomology groups in the different cases, and show that they contain the same information. Throughout we assume $g$ and $H$ are chosen so hypotheses (H2) and (H4) hold.

				\subsubsection{A $\bb{Z}$-grading}\label{AZGrading} 
				
				Here we show that $\K = \G_0(Q)$ acts freely on the space $\A_\fl(Q, H)$. Once we have done this, it follows immediately that the induced grading on $HF^\bullet_\inst(Q)^{\G_0(Q)}$ is a relative $\bb{Z}$-grading, since each element of $\G_0(Q)$ has degree zero.

			If follows from (H2) that every $\afa \in \A_\fl(Q, H)$ is irreducible, and so (\ref{ppppd}) holds for all ${\afr} \in \Omega^0(Y, Q(\frak{g}))$. As a preliminary step, we prove the following.
			
			\begin{lemma}\label{12days}
			Assume (H2) and let $p > 3$. Then there are constants $\eps_0, C  > 0$ such that (\ref{ppppd}) holds for all ${\afr} \in \Omega^0(Y, Q(\frak{g}))$ and for all connections ${\afa}$ on $Q \rightarrow Y$ with $\Vert F_{\afa} - X({\afa}) \Vert_{L^p(Y)} \leq \eps_0.$
				\end{lemma}
				
				\begin{proof} The result obviously holds for $H$-flat connections ${\afa}$, with a constant $C = C_{\afa}$ possibly depending on ${\afa}$. We first show that the constant $C$ can actually be chosen to be independent of the $H$-flat connection ${\afa}$. If not, then one could find sequences $\left\{{\afr}_n\right\}_n$ of 0-forms and $\left\{{\afa}_n\right\}_n$ of $H$-flat connections with
	
	\begin{equation}\label{jjskklla}
	\Vert {\afr}_n \Vert_{L^2} = 1, \indent \Vert d_{{\afa}_n} {\afr}_n \Vert_{L^2} \rightarrow 0.
	\end{equation}
	The bound (\ref{ppppd}) is gauge invariant, so by the perturbed version of Uhlenbeck compactness, we may assume that the ${\afa}_n$ converge strongly in $L^\infty$ to an $H$-flat connection ${\afa}_\infty$, after possibly passing to a subsequence. Moreover, this convergence of the ${\afa}_n$ combines with the assumptions on the ${\afr}_n$ to imply that the ${\afr}_n$ are uniformly bounded in $W^{1,2}$ (with the derivatives defined using $d_{{\afa}_\infty}$). So by passing to a further subsequence, we may assume the ${\afr}_n$ converge weakly in $W^{1,2}$ and hence strongly in $L^2$ to some ${\afr}_\infty$. This gives $1 = \Vert {\afr}_\infty \Vert_{L^2},$ and $d_{{\afa}_\infty} {\afr}_\infty = 0,$ which contradicts (\ref{ppppd}) applied to the flat connection ${\afa} = {\afa}_\infty$, and so the constant $C$ is independent of the choice of $H$-flat connection.
	
	Next, we show that the estimate (\ref{ppppd}) continues to hold for any connection ${\afa}$ with $F_{\afa} - X_{\afa}$ sufficiently $L^p$-small. This follows easily by the same kind of contradiction argument: If not, then there are ${\afa}_n$ and ${\afr}_n$ with (\ref{jjskklla}), and $\Vert F_{{\afa}_n} - X({\afa})\Vert_{L^p} \rightarrow 0.$ Then $\left\{{\afa}_n\right\}_n$ has a subsequence that converges strongly in $L^\infty$, modulo gauge equivalence, to some limiting $H$-flat connection, and the result follows exactly as before. (We use $p > 3$ to obtain a compact embedding $W^{1,p} \hookrightarrow L^\infty$.) 
	\end{proof}

		Hypothesis (H3) for the group $\K = \G_0(Q)$ is a special case of the following.
			
			\begin{lemma}\label{14days}
			Assume (H2) and let $p > 3$. There is some $\eps_0 > 0$ such that if ${\afa}$ is any connection on $Q \rightarrow Y$ with 
			
			\begin{equation}\label{3445}
			\Vert F_{\afa} - X(\afa)\Vert_{L^p(Y)} < \eps_0,
			\end{equation}
			then the stabilizer of ${\afa}$ in $\G_0(Q)$ is trivial. Moreover, the set ${\mathcal{I}}_H(Q)$ is naturally equipped with the structure of a smooth manifold of dimension zero, and it is non-compact whenever it is not empty.
			\end{lemma}

			\begin{proof} First we show how the second assertion follows from the first. The curvature is continuous (in fact, smooth) when viewed as a map $\A^{1,p}(Q) \longrightarrow L^p(\Omega^2),$ where $L^p(\Omega^2)$ is the $L^p$-completion of the space $\Omega^2(Y, Q(\frak{g}))$. In particular, the set ${\mathcal{U}}$ of connections ${\afa}$ satisfying (\ref{3445}) is a $\G_0^{2, p}(Q)$-invariant open set. It follows from the first assertion of the lemma that the set ${\mathcal{B}} \defeq {\mathcal{U}} / \G_0^{2, p}$ is naturally a smooth Banach manifold. The tangent space at a $\left[\afa\right] \in {\mathcal{B}}$ is the $W^{1,p}$-completion of the space 
			
			\begin{equation}\label{thjkkk}
			\Omega^1(Y, Q(\frak{g})) / \mathrm{im}\: d_{\afa}.
			\end{equation}
			We are interested in the Banach vector bundle ${\mathcal{E}} \rightarrow {\mathcal{B}}$ whose fiber over $\left[\afa\right]$ is the $L^p$-completion of (\ref{thjkkk}). Then the map ${\afa} \longmapsto *\left(F_{\afa} - X(\afa)\right)$ descends to a smooth section of ${\mathcal{E}}$ that is Fredholm with index 0. Moreover, the critical points of this section are exactly the elements of ${\mathcal{I}}_H(Q)$. By hypothesis (H2) these critical points are all non-degenerate, so by the inverse function theorem ${\mathcal{I}}_H(Q)$ is a smooth 0-dimensional manifold. The non-compactness follows, for example, from the fact that there is a degree $d \neq 0$ gauge transformation $\afu$. Then, by (\ref{mod4grading1000}), the $k$-fold iterated action of $\afu$ on an $H$-flat connection $\afa$ yields distinct $H$-flat connections.

			It remains to prove the first assertion in Lemma \ref{14days}. By Lemma \ref{12days}, there is some neighborhood ${\mathcal{U}}(e)$ of the identity in $\G_0(Q)$ that acts freely on the space of all connections with sufficiently small curvature. We want to show that this neighborhood is all of $\G_0(Q)$. This is another contradiction argument. If it does not hold, then there is a sequence of connections ${\afa}_n$ on $Q$ and gauge transformations ${\afu}_n$ on $Q$ such that ${\afu}_n^* {\afa}_n = {\afa}_n$ and $\Vert F_{\afa_n} - X(\afa_n)\Vert_{L^p} \rightarrow 0$. Apply Uhlenbeck's weak compactness theorem to conclude that, after passing to a subsequence, there is a sequence of gauge transformations ${\afu}'_n$ such that $({\afu}'_n)^* {\afa}_n$ converges weakly in $W^{1,p}$ to some flat connection ${\afa}_\infty$. For simplicity, we relabel $({\afu}'_n)^* {\afa}_n$ by ${\afa}_n$. Similarly, replace ${\afu}_n'{\afu}_n({\afu}_n')^{-1}$ with ${\afu}_n$, so we still have ${\afu}_n^* {\afa}_n = {\afa}_n.$ This relation combines with (\ref{slop}) to give
			
			$$ {\afa}_n{\afu}_n - {\afu}_n{\afa}_n = d {\afu}_n.$$
			Since the left-hand side is bounded in $W^{1,p}$, it follows that the right is as well. Thus the ${\afu}_n$ converge weakly in $W^{2,p}$ to some limiting gauge transformation ${\afu}_\infty$. This satisfies ${\afu}_\infty^* {\afa}_\infty = {\afa}_\infty.$ It follows from Lemma \ref{gaugestab0} that ${\afu}_\infty = e$ is the identity. To reiterate, the ${\afu}_n$ are a sequence of gauge transformation converging to the identity $e$, and each fixes a connection with small curvature. This contradicts the existence of the neighborhood ${\mathcal{U}}(e)$ containing the identity.
		\end{proof}

		\subsubsection{A $\bb{Z}_{4r}$-grading} \label{AZ_4rGrading}
			
			Here we consider the subgroup $\K = \ker \eta \subseteq \G(Q)$ given by the kernel of the parity operator $\eta: \G(Q) \rightarrow H^1(Y, \bb{Z}_r)$. First we want to show that this group acts freely on the space of $H$-flat connections. To see this, note that any gauge transformation in $\ker \eta$ with non-zero degree automatically acts freely on $\A_\fl(Q, H)$ by (\ref{mod4grading1000}). So it suffices to work with the degree zero elements of $\ker \eta$. However, any gauge transformation $\afu$ with $\deg(\afu) = 0$ and $\eta(\afu) = 0$ is automatically in the identity component $\G_0(Q)$, so the result follows by the analysis of Section \ref{AZGrading}.

		We saw in Section \ref{TopolicaAspectsOfPrincipalPU(r)-Bundles} that each element of $\ker \eta$ has degree divisible by $r$. Moreover, there always exists gauge transformations in $\ker \eta$ of degree $r$ (take $\afu_1^r$, where $\afu_1$ is the degree 1 gauge transformation from Proposition \ref{degprop} (b)). In particular, the induced grading on $HF^\bullet_\inst(Q)^{\ker \eta}$ is a relative $\bb{Z}_{4r}$-grading.

		It is interesting to compare the groups $HF^\bullet_\inst(Q)^{\G_0(Q)}$ and $HF^\bullet_\inst(Q)^{\ker \eta}.$ Fix a degree $r$ gauge transformation $\afu_r$, and assume that the instanton perturbation $H$ is invariant under the action of $\afu_r$. Then pullback by $\afu_r$ determines a degree $4r$ chain map on the underlying chain complex of $HF^\bullet_\inst(Q)^{\G_0(Q)}$. This therefore descends to a degree $4r$ map at the level of homology
		
		$$\iota_{\afu_r}: HF^\bullet_\inst(Q)^{\G_0(Q)} \longrightarrow HF^{\bullet+4r}_\inst(Q)^{\G_0(Q)},$$
		Moreover, this map is an isomorphism. It follows that all of the data of $HF^\bullet_\inst(Q)^{\G_0(Q)} $ is contained in any $4r$ consecutive degrees

		$$HF^k_\inst(Q) \oplus HF^{k+1}_\inst(Q) \oplus \ldots \oplus HF^{k+4r-1}_\inst(Q).$$
				
		In fact, we can say more. Quotienting $HF^\bullet_\inst(Q)^{\G_0(Q)}$ by the action of the map $\iota_{\afu_r}$ recovers the space $HF^\bullet_\inst(Q)^{\ker \eta}$
			
			\begin{equation}\label{quotbygu}
				HF^\bullet_\inst(Q)^{\G_0(Q)} {\longrightarrow  } HF^\bullet_\inst(Q)^{\ker \eta}.
			\end{equation}
			Fix the $k$th summand $HF^k_\inst(Q)^{\ker \eta}$ of $HF^\bullet_\inst(Q)^{\ker \eta}$. The fiber of the quotient map (\ref{quotbygu}) over this summand consists of the countably many summands
			
			$$\ldots \oplus HF^{k-4r}_\inst(Q)^{\G_0(Q)} \oplus  HF^{k}_\inst(Q)^{\G_0(Q)} \oplus  HF^{k+4r}_\inst(Q)^{\G_0(Q)} \oplus \ldots$$
			in $HF^{\bullet}_\inst(Q)^{\G_0(Q)}$. Moreover, the quotient map restricts to an isomorphism 
			
			$$HF^{k}_\inst(Q)^{\G_0(Q)}  \longrightarrow HF^{k}_\inst(Q)^{\ker \eta}$$ 
			for each $k \in \bb{Z}$. That is, all of the information of $HF^\bullet_\inst(Q)^{\G_0(Q)}$ is contained in $HF^\bullet_\inst(Q)^{\ker \eta}$, and vice-versa.
			
			\medskip

			Lastly, we discuss a sense in which the gauge group $\ker \eta$ is a very natural group to consider. Suppose $t_2(Q)$ is the reduction of an integral class. Then there is a $\U(r)$-bundle $\overline{Q} \rightarrow Y$ with $Q = \overline{Q} \times_{\U(r)} \PU(r)$ and $c_1(\overline{Q}) \equiv t_2(Q)$ mod $r$. In \cite{Donfloer} and \cite{BD}, the authors define a $\U(2)$-Floer theory, and their discussion generalizes easily to the group $\U(r)$. Their $\U(r)$-Floer chain complex is generated by central curvature connections with fixed determinant, and one works modulo the group of gauge transformations on $\overline{Q}$ with fixed determinant. It turns out that the $\U(r)$-Floer cohomology they describe is exactly the group $HF^\bullet_\inst(Q)^{\ker \eta}$ defined here. Indeed, the distinction between their set-up and ours is really just a matter of language. For example, flat connections on $Q$ are identified with central curvature connections on $\overline{Q}$ with fixed determinant. Likewise, the subgroup $\ker \eta \subset \G(Q)$ is identified with the gauge transformations on $\overline{Q}$ that have trivial determinant.\footnote{However, in \cite{Donfloer} and \cite{BD} it is pointed out that the issue of orientations for the moduli spaces is more difficult in the $\PU(r)$ theory. By working with $\bb{Z}_2$-coefficients we have avoided issues with orientations altogether. So from our perspective, $\PU(r)$-Floer theory on $Q$ is identical to $\U(r)$-Floer theory on $\overline{Q}$, provided of course that $t_2(Q)$ is induced from an integral class.}

			\subsubsection{A $\bb{Z}_{4}$-grading}\label{AZ_4Grading}

			Let $\afu_1 \in \G(Q)$ be a degree 1 gauge transformation satisfying the conclusion of Proposition \ref{degprop} (b). Note that the homotopy class of $\afu_1$ is uniquely determined by the conditions of the proposition, and hence uniquely determined by the choice of surface $\Sigma \subset Y$ from that proposition. Let $\G(\Sigma) \subseteq \G(Q)$ denote the group generated by $\afu_1$ and the identity component $\G_0(Q)$.
			
				\begin{example}\label{ex100}
		Suppose $d \in \bb{Z}_r$ is a generator, and $Q$ is a bundle that is $d$-compatible with a broken circle fibration $Y \rightarrow S^1$, as in Section \ref{CompatibleBundlesForBrokenCircleFibrations}. Take $\Sigma = \Sigma_0$. In that section we saw that $\G(\Sigma_0) = \G_{\Sigma_\bullet},$ where the latter group consists of those gauge transformations that restrict, on $\Sigma_\bullet$, to identity-component gauge transformations. 
		\end{example}

			 It follows that $\G(\Sigma)$ acts freely on $\A_\fl(Q, H)$, since each element of $\G(\Sigma)$ either has non-zero degree, or is in $\G_0(Q)$. It is immediate that the instanton Floer cohomology group $HF^\bullet_\inst(Q)^{\G(\Sigma)}$ admits a relative $\bb{Z}_4$-grading. 
			
			We want to compare the group $HF^\bullet_\inst(Q)^{\G(\Sigma)}$ with those defined in the previous two section. We begin with $HF^\bullet_\inst(Q)^{\G_0(Q)}$. As in Section \ref{AZ_4rGrading}, pullback by the degree 1 gauge transformation $\afu_1$ determines a degree 4 automorphism $\iota_{\afu_1}$ of $HF^\bullet_\inst(Q)^{\G_0(Q)}$. Quotienting by the action of the group of automorphisms generated by $\iota_{\afu_1}$ recovers $HF^\bullet_\inst(Q)^{\G(\Sigma)}$. Just as in Section \ref{AZ_4rGrading}, this says that both groups contain the same information. 
		
		In exactly the same way, $\afu_1$ determines a degree $4$ map $\iota_{\afu_1}'$ on $HF^\bullet_\inst(Q)^{\ker \eta} $. Once again, taking the quotient recovers the $\bb{Z}_4$-graded group $HF^\bullet_\inst(Q)^{\G(\Sigma)}$.

		\begin{remark}
		In \cite{BD} and \cite[Section 5.6]{Donfloer}, the authors address the map $\iota_{\afu_1}'$ in the case when $t_2(Q)$ is the reduction of an integral class. They define $\iota_{\afu_1}'$ by counting instantons on a $\U(r)$-bundle over $\bb{R} \times Y$ determined by $\Sigma$. 
		\end{remark}

		\subsubsection{More General $\K$}
		
		The extension of these results to more general $\K$ is difficult. First, for the analysis of Section \ref{InstantonFloerCohomologyForG=PU(r)}, we must have that $\K$ consists entirely of connected components of $\G(Q)$. To generalize this would require some serious reconstructions of the theory above. We therefore continue to assume that $\K$ consists entirely of connected components of $\G(Q)$. If a gauge transformation has non-zero degree, then it automatically acts freely by (\ref{mod4grading1000}). So the interesting cases are when we consider gauge transformations $\afu$ with $\deg(\afu) = 0$. These are determined, up to homotopy, by their parity. We are therefore considering the action of the finite group $H^1(Y, \bb{Z}_r)$. In \cite{Donfloer}, Donaldson points out that it is not clear how to construct instanton perturbations $H$ for which $H^1(Y, \bb{Z}_r)$ acts freely on $\A_\fl(Q, H)$. For example, one tends to construct perturbations $H$ satisfying (H2) and (H4) using the holonomy, as in Section \ref{CompatiblePerturbations}. When this is the case, the gauge transformations coming from elements of $H^1(Y, \bb{Z}_r)$ typically will \emph{not} act freely on the space of $H$-flat connections.

			\subsection{Instanton Floer cohomology for broken circle fibrations}\label{InstantonFleorCohomologyForBrokenCircleFibrations}

Now we specialize the discussion to the case relevant to the quilted Atiyah-Floer conjecture. Suppose $f: Y \rightarrow S^1$ is a broken circle fibration as in Section \ref{OverviewOfTheQuiltedAtiyahFloerConjecture}, and fix an integer $r   \geq 2$. Let $d \in \bb{Z}_r$ be a generator and let $Q\rightarrow Y$ be a principal $\PU(r)$-bundle that is $d$-compatible with $f$. In particular, $t_2(Q) \left[ \Sigma_0 \right] = d$, where $\Sigma_0$ is a regular fiber of $f$. Since $d \in \bb{Z}_r$ is a generator, this clearly satisfies hypothesis (H1) with $\Sigma = \Sigma_0$.

We will work with the relatively $\bb{Z}_4$-graded group $HF^\bullet_\inst(Q)^{\G_{\Sigma_\bullet}}$ from Section \ref{AZ_4Grading}; see also Example \ref{ex100}. It follows that this is uniquely determined by the data of $Y, f, t_2(Q), r$ and $d$ (in fact, it follows from the definition that it only depends on $f$ through its homotopy class in $\left[ Y, S^1\right]$). So in analogy with the quilted theory, we set

$$HF^\bullet_\inst(Y, f, Q)_{r, d} \defeq HF^\bullet_\inst(Q)^{\G_{\Sigma_\bullet}}$$
and call this the {\bfseries instanton Floer cohomology} of $(Y, f, t_2(Q), r, d)$. The goal of this section is to describe the generators and boundary operator for $HF^\bullet_\inst(Y, f, Q)_{r, d}$ in a manner analogous to the quilted case in Section \ref{QuiltedFloerCohomologyForBrokenCircleFibrations}. 
		
			\medskip
			
			We begin by restricting attention to a subclass of instanton perturbations that will help tie in the discussion with quilted Floer theory. Suppose for each $j\in \left\{0, \ldots, N-1\right\}$ we are equipped with a time-dependent function 
			$$H_{j, t}:  \A^{k,p}(P_j) \rightarrow \bb{R}$$ 
			that vanishes for $t \notin \left(0, 1\right)$, and is invariant under $ \G_0(P_j)$ for each $t$. Assume this is smooth with respect to the $W^{k, p}$-topology, though we suppress this from the notation. Then we obtain a smooth function $H : \A(Q) \rightarrow \bb{R}$ by the formula
				
				\begin{equation}\label{deforhandshit1}
				H({\afa}) \defeq  \sumdd{j= 0}{N-1} \intdd{0}{1} \: H_{j,t}\left({\afa} \vert_{\left\{t \right\} \times \Sigma_j} \right),
				\end{equation}
				and this is invariant under the gauge subgroup $\G_{\Sigma_\bullet}$. Let $X: \A(Q) \rightarrow \Omega^2(Y, Q(\frak{g}))$ represent the linearization $dH$ as in the beginning of Section \ref{InstantonFloerCohomologyForG=PU(r)}. Similarly, there is a map $X_{j, t}: \A(P_j) \rightarrow \Omega^1(\Sigma_j, P_j(\frak{g}))$ representing the linearization of $H_{j, t}$ (this is the Hamiltonian vector field for $h_{j , t}$). One can check these are related by
			$$X (\afa)\vert_{I  \times \Sigma_j} = dt \wedge  X_{j,t}(\afalpha(t)),$$
				where we have written $\afa = \afalpha(t) +\psi(t) \: dt$ on $I \times \Sigma_j$. Moreover, $X(\afa)$ vanishes on $Y_\bullet$. If, in addition, $H$ is an instanton perturbation (i.e., $V \defeq X$ satisfies conditions (a) and (b) from Theorem \ref{kronsmagic}), then we will say that $H$ is {\bfseries compatible (with quilted Floer theory)}. 
				
				It is often convenient to identify $H$ with the split-type Hamiltonian $\underline{H} = \underline{H}_t= \smash{\left( H_{j, t}\right)_j}$ on $\A(P_0) \times \ldots \times \A(P_{N-1})$. Moreover, restricting $H_{j, t}$ to $\A_\fl(P_j)$ and then quotienting by $\G_0(P_j)$, we obtain a Hamiltonian on $\M(P_j)$ as in Section \ref{QuiltedFloerCohomologyForBrokenCircleFibrations}. As another point on notation, if $\afalpha$ is a connection on $\Sigma_\bullet$, then we write 
				$$X_{t}(\afalpha) \in \Omega^1(\Sigma_\bullet, P_\bullet(\frak{g}))$$ 
				for the $t$-dependent vector field given over $\Sigma_j$ by $X_{j, t}(\afalpha \vert_{\Sigma_j})$.

\medskip

			The discussion above carries over directly to the bundle $Q^\epsilon \rightarrow Y^\epsilon$, equipped with the $\eps$-smooth structure and metric $g_\eps$ (see Sections \ref{EpsDependentSmoothStructures} and \ref{CompatibleBundlesForBrokenCircleFibrations}). Since instanton Floer cohomology depends only on the diffeomorphism type of $Y$ and isomorphism class of $Q$, it is clear that $HF^\bullet(Y^\eps, f, Q^\eps)_{r, d} = HF^\bullet(Y, f, Q)_{r, d}$ is independent of $\eps$.
			
			On the other hand, there is $\eps$-dependence at the chain level. To see this dependence on $\eps$ explicitly, we pass to certain local coordinates as follows: Let $H$ be a compatible instanton perturbation that is smooth on $\A(Q)$. The compatibility condition tells us that $H(\afa)$ only depends on $\afa \vert_{I \times \Sigma_\bullet}$, and so $H$ is actually a smooth function on $\A(Q^\eps)$ for any $\eps > 0$. For example, the 2-form $X(\afa)$ vanishes on $Y_\bullet$, and so does not fall victim to the pathological behavior illustrated in Figure \ref{figure7}.
			
			We have seen that every connection ${\afA}$ on $\mathbb{R} \times Y$ can be written in the form ${\afA} = {\afa}(s) + {\afp}(s)\:ds$. Similarly, over $\mathbb{R} \times I \times \Sigma_\bullet$ we can write $\left.{\afA}\right|_{\left\{(s, t) \right\} \times \Sigma_\bullet} = \afalpha(s, t) + {\afphi}(s, t)\:ds + {\afpsi}(s, t) \: dt,$ for some $\afalpha: \mathbb{R} \times I \rightarrow \A(P_\bullet)$ and ${\afphi}, {\afpsi}: \mathbb{R} \times I \rightarrow \Omega^0(\Sigma_\bullet, P_\bullet(\frak{g}))$. It follows that $\left.{\afa}(s)\right|_{\left\{t\right\} \times \Sigma} = \afalpha(s, t) + {\afpsi}(s, t)~dt$ and $\left. {\afp}(s) \right|_{\left\{t \right\} \times \Sigma} = {\afphi}(s, t).$ The curvature $F_{\afA}$ on the four-manifold $\bb{R} \times I \times \Sigma_\bullet$ can be written in terms of its components as 
			
			$$\begin{array}{rcl}
			F_{\afA}& = & F_{\afalpha} + (\partial_s {\afpsi} - \partial_t {\afphi} - \left[ {\afpsi}, {\afphi}\right] )\:ds\wedge dt\\
			&&~~ - (\partial_s {\afalpha}  - d_{\afalpha}  {\afphi}) \wedge ds - (\partial_t {\afalpha} - d_{\afalpha} {\afpsi})\wedge dt .
			\end{array}$$

			With this notation in place, we can therefore view the instanton boundary operator $\partial_{\inst}$ as counting isolated $\G_{\Sigma_\bullet}$-equivalence classes of connections ${\afA} \in \A(\mathbb{R} \times Q^\epsilon )$ satisfying

			\begin{equation}
			\begin{array}{lrcl}
				\bullet~~\textrm{($(g_\epsilon, {H})$-ASD on $\mathbb{R}\times I \times \Sigma_\bullet$}) \indent & \partial_s {\afalpha} - d_{\afalpha} {\afphi} + *\left(\partial_t {\afalpha} - X_t({\afalpha}) - d_{\afalpha} {\afpsi}\right) &=& 0\\
&  \partial_s {\afpsi} - \partial_t {\afphi} - \left[ {\afpsi}, {\afphi}\right] +\epsilon^{-2} *F_{\afalpha}& = &0 \\
				&&&\\
				\bullet~~\textrm{($(g_\epsilon, {H})$-ASD on $\mathbb{R} \times Y_\bullet$)} & \partial_s {\afa} - d_{\afa} {\afp} +\epsilon^{-1}*F_{\afa}&=& 0\\
 &\\
			\end{array}
			\label{asdconditions}
			\end{equation}
			with appropriate limits at $\pm \infty$. Here the norms and Hodge stars are all with respect to the fixed metric $g$. Compare the small $\eps$-limit of these equations with (\ref{jholoconditions}).

\section{The quilted Atiyah-Floer conjecture}\label{TheQuiltedAtiyahFloerConjecture}

In this section we state and discuss the quilted Atiyah-Floer conjecture and its cousins. Throughout we refer to Sections \ref{QuiltedFloerCohomologyForBrokenCircleFibrations} and \ref{InstantonFleorCohomologyForBrokenCircleFibrations} for notation. We begin by discussing a simple invariance property of Floer cohomology. Fix an integer $r \geq 2$ and a generator $d \in \bb{Z}_r$. For $j = 0, 1$, let $f_j : Y_j \rightarrow S^1 $ be a broken circle fibration and $Q_j \rightarrow Y_j$ a bundle that is $d$-compatible with $f_j$. Suppose there is an orientation-preserving diffeomorphism $Y_0 \rightarrow Y_1$ pulling $f_1$ back to $f_0$ and $Q_1$ back to $Q_0$. Then this diffeomorphism induces canonical isomorphisms of $\bb{Z}_4$-graded abelian groups:

\begin{equation}\label{nat1}
\begin{array}{rcl}
HF^\bullet_\quilt(Y_0, f_0, Q_0)_{r, d} & \cong & HF^\bullet_\quilt(Y_1, f_1, Q_1)_{r, d}\\
&&\\
HF^\bullet_\inst(Y_0, f_0, Q_0)_{r, d} & \cong & HF^\bullet_\inst(Y_1, f_1, Q_1)_{r, d}. 
\end{array}
\end{equation}
Now we can formally state the conjecture.

\begin{qconjecture}
Fix an integer $r \geq 2$ and let $d \in \bb{Z}_r$ be a generator. Given any broken circle fibration $f: Y \rightarrow S^1$ and principal $\PU(r)$-bundle $Q \rightarrow Y$ that is $d$-compatible with $f$, there is an isomorphism $HF^\bullet_\quilt(Y, f, Q)_{r, d} \cong HF^\bullet_\inst(Y, f, Q)_{r, d} $ of relatively $\bb{Z}_4$-graded abelian groups. Moreover, this isomorphism is natural in the sense that it intertwines the isomorphisms (\ref{nat1}).
\end{qconjecture}

There are various stronger versions of this conjecture. For example, one could repeat the discussion of Sections \ref{QuiltedFloerCohomology} and \ref{InstantonFloerCohomology}, generating the chain complexes over $\bb{Z}$, rather than $\bb{Z}_2$. This would require coherently orienting the moduli spaces of instantons and $J$-holomorphic curves in such a way that the boundary operator in each theory still squares to zero.

A second direction in which one could state a stronger conjecture is via a chain level version. The remainder of this section is dedicated to describing such a version. We let $CF^\bullet_\quilt \left(Y, f,Q,  \underline{J}, \underline{H}\right)_{r, d}$ (resp. $CF^\bullet_\inst \left(Y, f, Q, g, H\right)_{r, d}$) denote the $\bb{Z}_4$-graded quilted (instanton) Floer chain complex. We are suppressing the boundary operator in the notation. As usual, whenever we write this, we assume the data have been chosen so that these chain complexes are defined. Just as at the homology level, these chain groups satisfy an independence property: Suppose there is an orientation-preserving diffeomorphism from $Y_0$ to $Y_1$ that pulls the data on $Y_1$ back to the data on $Y_0$. Then this diffeomorphism determines a canonical isomorphism of the associated $\bb{Z}_4$-graded chain complexes.

\begin{qcconjecture} 
Fix an integer $r \geq 2$ and let $d \in \bb{Z}_r$ be a generator. Then given any broken circle fibration $f: Y \rightarrow S^1$, and principal $\PU(r)$-bundle $Q \rightarrow Y$ that is $d$-compatible with $f$, there are choices of $\underline{J}, \underline{H}$ and $g, H$ such that the following conditions hold.

\begin{itemize}
\item[(a)] The chain complexes $CF^\bullet_\inst \left(Y, f, Q, g, H\right)_{r, d}$ and $CF^\bullet_\quilt \left(Y, f, Q, \underline{J}, \underline{H}\right)_{r, d}$ are well-defined.
\item[(b)] There is a quasi-isomorphism 
$$\Psi: CF^\bullet_\inst \left(Y, f, Q, g, H\right)_{r, d} \longrightarrow CF^\bullet_\quilt \left(Y, f, Q, \underline{J}, \underline{H}\right)_{r, d}$$ 
that preserves the gradings and is natural in the sense that it intertwines the canonical isomorphisms on each chain complex.
\end{itemize}
\end{qcconjecture}

\subsection{A program for proving the conjecture}\label{AProgramForProvingTheConjecture}

Here we introduce a program for proving the chain level Atiyah-Floer conjecture. At the end we discuss briefly why we believe this is a fruitful approach.

To begin, we need to prescribe how the auxiliary data $\underline{J}, \underline{H}$ and $g, H$ can be chosen. Let $g$ be any metric on $Y$, and fix $\eps > 0$. Define the metric $g_\eps$ as in Section \ref{EpsDependentSmoothStructures}. It is convenient (but not necessary) to assume that $g$ restricts to a metric on $I \times \Sigma_\bullet$ that is constant in the $I$-direction. The choice of $g_\eps$ induces a split-type compatible almost complex structure $\underline{J}_g$ as in Remark \ref{surfacesremark} (a). Moreover, $\underline{J}_g$ is independent of $\eps$ because of the conformal invariance of 1-forms on surfaces. 

In Section \ref{CompatiblePerturbations}, we describe how to construct instanton perturbations $H$ that are compatible with quilted Floer theory in the sense of Section \ref{InstantonFleorCohomologyForBrokenCircleFibrations}. Any such $H$ induces a split-type Hamiltonian $\underline{H}_t$ as in Section \ref{InstantonFleorCohomologyForBrokenCircleFibrations}. Moreover, Proposition \ref{comppertprop} shows that $H$ can be chosen to be invariant under $\G(Q)$, and so that $CF^\bullet_\inst \left(Y^\eps, f^\eps, Q^\eps, g_\eps, H\right)_{r, d}$ and $CF^\bullet_\quilt \left(Y^\eps, f^\eps, Q^\eps, \underline{J}_g, \underline{H}_t\right)_{r, d}$ are well-defined. This takes care of (a) in the statement of the chain level conjecture. 

\begin{remark}\label{epsdependenceremark}
In the set-up we are describing, we are working with the $\eps$-dependent bundle $Q^\eps \rightarrow Y^\eps$, for a fixed $\eps > 0$ small. This ensures that the metric $g_\eps$ is smooth. However, the associated instanton chain complex only depends on $\eps$ through the boundary operator $\partial_\inst$. That is, the underlying abelian group

$$CF^\bullet_\inst \left(Y^\eps, f^\eps, Q^\eps, g_\eps, H\right)_{r, d} = \bigoplus_{\left[ \afa \right] \in {\mathcal{I}}_{H}(Q^\eps)}\: \bb{Z}_2 \langle \left[ \afa \right] \rangle$$
is independent of $\eps$, up to the action of the canonical diffeomorphism (\ref{milnorsdiffeo}).

In the case of the quilted Chain complex, the independence of $\eps$ is even stronger. Since $\underline{J}_{g_\eps} = \underline{J}_g$ for all $\eps$, it follows that both the abelian group $CF^\bullet_\quilt\left(Y^\eps, f^\eps, Q^\eps, \underline{J}_{g_\eps}, \underline{H}_t\right)_{r, d}$ and the boundary operator $\partial_\quilt$ are independent of $\eps$. Consequently, we will drop $\eps$ from the notation, unless it is relevant to the discussion.
\end{remark}

As for (b) in the statement of the chain level conjecture, it suffices to define the map $\Psi$ by specifying its value on the generators ${\mathcal{I}}_{H}(Q)$ of the domain. Similarly, ${\mathcal{I}}_{\underline{H}_t}(\underline{L}(Q))$ denotes the generating set of the quilted complex. Given any $\left[\afa\right] \in {\mathcal{I}}_{H}(Q)$, one can obtain an element of ${\mathcal{I}}_{\underline{H}_t}(\underline{L}(Q))$ by restricting $\left[ \afa\right]$ to each cobordism in $Y_\bullet$ (to do this one needs to know that $H$ is compatible with quilted Floer theory, which we have assumed is the case):

\begin{equation}
\begin{array}{rcl}
\Psi: {\mathcal{I}}_{H}(Q)&\longrightarrow  &{\mathcal{I}}_{\underline{H}_t}(\underline{L}(Q))\\
\left[\afa\right] & \longmapsto & \underline{e}_{\afa}\defeq \left(\left[\left.\afa\right|_{Y_{12}}\right], \left[\left.\afa\right|_{Y_{23}}\right], \ldots, \left[\left.\afa\right|_{Y_{N1}}\right] \right).
\end{array}
\label{psidef}
\end{equation}
Here the bracket on the left denotes the $\G_{\Sigma_\bullet}$-equivalence class, and the brackets on the right denote the $\G_0$-equivalence class on the relevant 3-manifold $Y_{j(j+1)}$. Observe that if $u \in \G_{\Sigma_\bullet}$ and $a \in \A_\fl(Q, H)$, then the restrictions $a\vert_{Y_{i(i+1)}}$ and $u^*a\vert_{Y_{i(i+1)}}$ are $\G_0(Q_{i(i+1)})$-gauge equivalent (this is essentially the definition of $\G_{\Sigma_\bullet}$). Hence, the map (\ref{psidef}) is well-defined.

\begin{theorem}\label{generators}
The map $\Psi$ defined in (\ref{psidef}) is a set bijection. Moreover, $\Psi$ is natural and respects the relative $\bb{Z}_4$-gradings.
\end{theorem}

\begin{proof}
The naturality is immediate. That $\Psi$ respects the gradings is proved in \cite{DunIndex}. The bijectivity is really just a direct consequence of the definitions, but we supply a proof for convenience. The idea is that, given any perturbed Lagrangian intersection point $\underline{e}$, the matching conditions on the boundary imply that we can construct an $H$-flat connection $\afa$ that maps to $\underline{e}$ under $\Psi$. When constructing $\afa$ there is a choice involved, and the ambiguity is measured by the group $\G_{\Sigma_\bullet}$. However, by Example \ref{ex100}, the set ${\mathcal{I}}_{H}(Q)$ is exactly the quotient $\A_\fl(Q, H)$ by the action of $\G_{\Sigma_\bullet}$, and so the class of $\afa$ in ${\mathcal{I}}_{H}(Q)$ is uniquely determined by the intersection point $\underline{e}$. 

Now we work this out in detail. To simplify the notation we assume $H = 0$, and so $\underline{H}_t = 0$. We will denote by $\iota: \Sigma_i \hookrightarrow Y_{i(i+1)}$ and $\iota': \Sigma_{i+1} \hookrightarrow Y_{(i-1)i}$ the inclusion of the boundary components. We do not keep track of the index $i$ in the notation of $\iota$ and $\iota'$. We first prove that $\Psi$ is surjective. Fix some 

$$\left(\Big[a_{01}\Big], \: \Big[a_{12}\Big], \: \ldots,\: \Big[a_{(N-1)0}\Big] \right) \in {\mathcal{I}}(\underline{L}(Q)).$$ 
We need to show that there is some $a \in \A_{\fl}(Q)$ with $\left. a \right|_{Y_{i(i+1)}}  \in \left[a_{i(i+1)}\right].$ Choose a representative $a_{i(i+1)} \in \A_{\fl}(Q_{i(i+1)})$ for each $\left[a_{i(i+1)}\right]$. These give us an obvious definition for $a$ over the $Y_{\bullet}$, however we need to define $a$ over $I \times \Sigma_\bullet$ as well. To do this, let $\alpha_i \defeq \left. a_{(i-1)i}\right|_{\iota'(\Sigma_i)}$. Then $\alpha_i\in \A_{\fl}(P_i)$ is flat, and therefore so is $\mathrm{proj}^* \alpha_i$ where $\mathrm{proj}: I \times \Sigma_i \rightarrow \Sigma_i$ is the projection (so $\mathrm{proj}^* \alpha_i$ is a flat connection on the 3-manifold $I \times \Sigma_i$). The boundary condition (\ref{quiltboundarycond}) implies that there is some $\mu_i \in \G_0(P_i)$ with $\mu_i^*\alpha_i = \left. a_{i(i+1)}\right|_{\iota(\Sigma_i)}$. By definition, $\G_0(P_i)$ is path-connected so there is some path $u_i: I \rightarrow \G_0(P_i)$ connecting the identity to $\mu_i$. We can equivalently view the path $u_i$ as an element of $\G(I \times P_i)$, i.e. as a gauge transformation over the cylinder $I \times \Sigma_i$. Then $u_i^* \left(\mathrm{proj}^* \alpha_i\right)$ is a flat connection on $I \times P_i$ that agrees with $a_{(i-1)i}$ and $a_{i(i+1)}$ at each of the two boundary components. Then we define $a$ to be $u_i^* \left(\mathrm{proj}^* \alpha_i\right)$ over $I \times \Sigma_i$. It follows that $a$ is continuous, flat and restricts to the desired connections over the $Y_{i(i+1)}$. By choosing the paths $u_i: I \rightarrow \G_0(P_i)$ appropriately, we can also ensure that $a$ is smooth. (For example, choose the $u_i$ so they extend to be smooth maps $\bb{R} \rightarrow \G_0(P_i)$ that are constant on the complement of $I \subset \bb{R}$.)

For injectivity, suppose ${\afa}, {\afa}' \in \A_{\fl}(Q)$ are two flat connections that descend to the same Lagrangian intersection point under $\Psi$. We need to show that $\afa' = \afu^* \afa$ for some gauge transformation $\afu \in \G_{\Sigma_\bullet}$. The fact that $\afa$ and $\afa'$ restrict to the same Lagrangian intersection point implies that these connections restrict (modulo $\G_0(Q_{i(i+1)})$) to the same connection on each $Y_{i(i+1)}$. First we will show that $\afa$ and $\afa'$ also restrict (modulo $\G_0(I \times P_i)$) to the same connection on each $I \times \Sigma_i$. This is essentially a consequence of the injectivity of the restriction $\M(I \times P) \hookrightarrow \M\left(\left\{0\right\} \times P\right) \times \M\left(\left\{1\right\} \times P\right)$ from Theorem \ref{cobordisms} for bundles $I \times P \rightarrow I \times \Sigma$ over product cobordisms. Indeed, that $a$ and $a'$ restrict to the same connection (mod gauge) on $Y_\bullet$, in particular, means that they restrict to the same connection (mod gauge) on each boundary component $\Sigma_i$. So $\left[\left.a \right|_{I \times \Sigma_i} \right]$ and $\left[\left.a' \right|_{I \times \Sigma_i} \right]$ have the same image in $\M(\Sigma_i) \times \M(\Sigma_i)$ and hence these are equal $\left[\left.a \right|_{I \times \Sigma_i} \right] = \left[\left.a' \right|_{I \times \Sigma_i} \right]$. So there are gauge transformations $u_{i} \in \G_0(I \times P_i)$ and $u_{i(i+1)} \in \G_0(Q_{i(i+1)})$ such that 

$$\left.u_i^*a\right|_{I \times \Sigma_i} = \left.a'\right|_{I \times \Sigma_i} \indent \textrm{and} \indent \left.u_{i(i+1)}^*a\right|_{Y_{i(i+1)}} = \left.a'\right|_{Y_{i(i+1)}}.$$ 
The data $\left\{u_i,u_{i(i+1)} \right\}_{i = 0, \ldots, N-1}$ patch together to form a global (possibly discontinuous) gauge transformation $u$ on $Q \rightarrow Y$. We need to show that $u$ is smooth; it will then automatically be contained in $\G_{\Sigma_\bullet}$.

We first claim that $u$ is continuous. To see this, note that 

$$\left(\left.u_i\right|_{\left\{1 \right\} \times \Sigma} \right)^* \left(\left.a\right|_{\left\{1 \right\}\times \Sigma_i} \right) = \left.a'\right|_{\left\{1 \right\}\times \Sigma_i}=  \left.a'\right|_{\iota(\Sigma_i)} = \left(\left.u_{i(i+1)}\right|_{\iota(\Sigma_i)}\right)^*\left(\left.a\right|_{\iota(\Sigma_i)}\right).$$
The functoriality of the characteristic classes from Section \ref{TopolicaAspectsOfPrincipalPU(r)-Bundles} implies that the map $\iota^*: \G(Q_{i(i+1)}) \rightarrow \G(P_i)$ induced by the inclusion map $\iota$ restricts to a map of the form $\G_0(Q_{i(i+1)}) \rightarrow \G_0(P_i)$. This shows that 

$$ \left(\left.u_{i(i+1)}\right|_{\iota(\Sigma_i)}\right)^{-1}\left.u_i\right|_{\left\{1 \right\} \times \Sigma_i} \in \G_0(P_i)$$ 
and that this fixes $\left.a\right|_{\left\{1 \right\}\times \Sigma_i}  = \left.a\right|_{\iota(\Sigma_i)}$. By Lemma \ref{gaugestab0}, we must have that this is the identity gauge transformation, and so $\left.u_{i(i+1)}\right|_{\iota(\Sigma_i)}= \left.u_i\right|_{\left\{1 \right\} \times \Sigma_i}.$ A similar argument shows $\left.u_{i(i+1)}\right|_{\iota'(\Sigma_{i+1})}= \left.u_{i+1}\right|_{\left\{0 \right\} \times \Sigma_{i+1}}$. This proves the claim.

To see that $u$ is actually smooth, we use the following trick from \cite[Chapter 2.3.7]{DK} to bootstrap: As in (\ref{slop}), by choosing a faithful matrix representation we can write the action of $u$ on $a$ as $u^*a = u^{-1}au -u^{-1}du,$ where on the right we are viewing the gauge transformation as a map $u: Q \rightarrow G$, and the concatenation is matrix multiplication. Rearranging this and setting $a' \defeq u^*a$ gives $du = au + ua'.$ The right-hand side is ${\mathcal{C}}^0$, so $u$ is of differentiability class ${\mathcal{C}}^1$. Repeatedly bootstrapping in this way shows that $u$ is in ${\mathcal{C}}^\infty$, and hence $u \in \G(Q)$. 

\end{proof}

As we have seen, in order to make each of the Floer chain complexes well-defined, we need to choose perturbations in such a way that the generators are non-degenerate. It turns out that non-degeneracy for one chain complex implies non-degeneracy for the other; see \cite{DunIndex} for a proof.

\begin{proposition}\label{nondegprop}
Suppose $H$ is an instanton perturbation that is compatible with quilted Floer theory, and let $\underline{H}$ be the split-type Hamiltonian perturbation associated to $H$ via the procedure of Section \ref{InstantonFleorCohomologyForBrokenCircleFibrations}. Then the elements of ${\mathcal{I}}_{H}(Q)$ are non-degenerate if and only if the elements of ${\mathcal{I}}_{\underline{H}}(\underline{L}(Q))$ are non-degenerate.
\end{proposition}

To prove the chain level quilted Atiyah-Floer conjecture, it therefore suffices to show that $\Psi$ intertwines the boundary operators (it will automatically induce an isomorphism at the level of homology, since it is a bijection). Recall from Remark \ref{epsdependenceremark} the though the underlying chain groups of each Floer theory is independent of $\eps> 0$, the instanton boundary operator $\partial_\inst$ does depend on $\eps$ since the instanton moduli spaces do. If one could show that the instanton moduli spaces of dimension zero are in bijective correspondence with the holomorphic curve moduli spaces of dimension zero, then it would be automatic that $\Psi$ is a chain map. When $\eps > 0$ is small, this is likely the case since the defining equations for each moduli space are close for small $\eps$. In other words, proving the chain level quilted Atiyah-Floer conjecture reduces to showing that each isolated $\eps$-instanton is close to a unique isolated holomorphic curve representative, and vice versa. In \cite{DunComp} we show that the first holds. Showing that the `vice versa' holds is currently an active research project.

In the special case where $f : Y \rightarrow S^1$ has no critical points, Dostoglou and Salamon were able to prove the chain level conjecture by exactly the approach we have described here \cite{DS2} \cite{DS3} \cite{DSerr}. Our approach to the more general case with critical points was motivated by their work.

\subsection{Compatible perturbations}\label{CompatiblePerturbations}

Fix a generator $d \in \bb{Z}_r$, and a broken circle fibration $f: Y \rightarrow S^1$. Let $Q \rightarrow Y$ be a principal $\PU(r)$-bundle that is $d$-compatible with $f$, in the sense of Section \ref{CompatibleBundlesForBrokenCircleFibrations}. Then $Q$ satisfies hypothesis (H1) from Section \ref{InstantonFloerCohomology}. In this section we prove the following proposition.

\begin{proposition}\label{comppertprop}
Fix a countable set of metrics $\left\{g_n \right\}_{n \in \bb{N}}$ on $Y$. Then there is a smooth instanton perturbation $H: \A^{k, p}(Q) \longrightarrow \bb{R}$ satisfying the following:

\begin{itemize}
\item[(i)] $(g_n, H)$ is regular for instanton Floer theory, for each $n$;
\item[(ii)] $H$ is invariant under the full gauge group $\G^{k+1, p}(Q)$;
\item[(iii)] $H$ is compatible with quilted Floer theory;
\item[(iv)] $(\underline{J}_{0}, \underline{H}_t)$ is regular for quilted Floer theory.
\end{itemize}
Here, $\underline{H}_t$ is the split-type Hamiltonian associated to $H$ via the procedure of Section \ref{InstantonFleorCohomologyForBrokenCircleFibrations}, and $\underline{J}_{0}$ is the split-type almost complex structure associated, as in Remark \ref{surfacesremark} (a), to the metric $g_0$ on $Y$. 
\end{proposition}

The existence of instanton perturbations satisfying (i) and (ii) is quite standard; our proof will follow \cite[Chapter 5.5]{Donfloer} closely. The novel feature here is that the perturbation can be chosen to satisfy (iii) and (iv). The proof will effectively show that the set of $H$ satisfying (i-iv) is comeager in a suitable function space. It will also show that $H$ can be taken to be smooth in the $W^{k, p}$-topology on $\A(Q)$ for all $k \geq 0$, $ p\geq 1$. Moreover, $H$ can be chosen to be arbitrarily small in the following sense: For any $\eps > 0$, $p \geq 1$, and non-negative integers $k_0, \ell_0$, there is a perturbation $H$ satisfying the conclusions of Proposition \ref{comppertprop} as well as the estimates (\ref{lemma1}) and (\ref{lemma2}) with $C_H(k, \ell) = \eps$, for all $0 \leq k \leq k_0$ and $0 \leq \ell \leq \ell_0$.

\begin{remark}\label{appremark}
Our intended application to the quilted Atiyah-Floer conjecture is as follows (see the discussion of Section \ref{AProgramForProvingTheConjecture}): Fix a metric $g$ on $Y$, and a sequence $\eps_n \rightarrow 0$ of positive numbers with $\eps_1 = 1$. Then we set $g_n \defeq g_{\eps_n}$. The existence of $H$ from Proposition \ref{comppertprop} shows that both Floer theories are well-defined for arbitrarily small $\eps = \eps_n$. Note that, in this case, the $g_n$ are not all smooth with respect to the same smooth structure on $Y$. This, however, is not an issue since compatible instanton perturbations on $Q$ are $\eps$-smooth for all $\eps > 0$. The reason we want a single perturbation $H$ for arbitrarily small $\eps$ is that this allows us to fix $H$ ahead of time, and then carry out the analysis in \cite{DunIndex} and \cite{DunComp}. This analysis then tells us how small we need to take $\eps_n$.
\end{remark}

\begin{proof}[Proof of Proposition \ref{comppertprop}]

Our proof is a combination of the arguments given in \cite[Section 7]{DS2}, \cite[Section 5.5]{Donfloer}, and \cite[Section 3]{Kron}. As usual, we drop the Sobolev exponents unless they are relevant to the discussion.

We begin by discussing the class of perturbation $H$ we are considering; we will call these \emph{$\Sigma_\bullet$-holonomy perturbations}. These, in turn, will be constructed using \emph{$I$-dependent holonomy perturbations} on the surfaces $\Sigma_i$. To describe the latter, fix a component $\Sigma_i \subset \Sigma_\bullet$, and a finite collection of embeddings

$$\gamma_{j}: \bb{R} / \bb{Z} \times \bb{D}^1 \longrightarrow P_i, \indent j \in \left\{ 1, \ldots, J\right\}$$ 
 such that each $\pi \circ \gamma_j$ is orientation-preserving; here, $\pi: P_i \rightarrow \Sigma_i$ is the projection. Assume $\smash{\gamma_{j}(0, \tau)}$ is close to a common basepoint in $P_i$ for all $j$ and all $\tau \in \bb{D}^1$. Then, for $\tau \in \bb{D}^1$, we can consider the map $\rho_\tau: \A(P_i) \rightarrow \SU(r)^{J}$ given by the based $\SU(r)$-valued holonomy around each of the $J$ loops $\theta \mapsto \gamma_{ j}(\theta, \tau)$. Fix a word $W$ in $J$ variables, and view this word as a monomial $\SU(r)^{J} \rightarrow \SU(r)$. Then this data determines a gauge invariant $\bb{R}$-valued map $\A(P_i) \times \bb{D}^1  \rightarrow  \bb{R}$ by the formula $(\afalpha, \tau)   \mapsto   \mathrm{Tr}\left(W\left( \rho_\tau(\afalpha) \right)  \right)$. Next, integrate this over $\bb{D}^1$ to obtain a smooth $\G(P_i)$-invariant map $h_{\Sigma_i}: \A(P_i)  \rightarrow  \bb{R}$ given by 

\begin{equation}\label{holpertbounds}
h_{\Sigma_i}(\afalpha   )  = \intd{\bb{D}^1} \:  \mathrm{Tr}\left(W\left( \rho_\tau(\afalpha) \right)  \right) \: \beta_0 \: d\tau,
\end{equation}
where $\beta_0: \mathrm{int}\:\bb{D}^1 \rightarrow \bb{R}^{\geq 0}$ is a smooth compactly supported bump function with total integral 1. This is a holonomy perturbation for $\Sigma_i$. To incorporate the $I$-dependence (necessary from the definition of \emph{compatible}, and for the proof below), fix another bump function $\beta_1 : I \rightarrow \bb{R}^{\geq 0}$ that vanishes at the boundary, and set
$$h_{i, t} \defeq \beta_1 (t) h_{\Sigma_i}.$$
We will refer to this as an {\bfseries $I$-dependent holonomy perturbation for $\Sigma_i$}, and we think of $h_{i, t} = h_{i , t}(\left\{ \gamma_{j} \right\}, W )$ as depending only on the $\gamma_{j}$ and the word $W$ (e.g., we think of the bump functions $\beta_0, \beta_1$ as fixed). 

Repeating the construction of the previous paragraph for each $\Sigma_i$, one can then define $H: \A(Q) \rightarrow \bb{R}$ by (\ref{deforhandshit1}). Suppose $H^1, \ldots, H^K : \A(Q) \rightarrow \bb{R}$ are constructed in this fashion, and let $\eta: \bb{R}^K \rightarrow \bb{R}$ be any smooth map. Then we can form the function
\begin{equation}\label{defofholo}
\A(Q) \longrightarrow \bb{R}, \indent \afa \longmapsto \eta \left( H^1(\afa), \ldots, H^K(\afa) \right).
\end{equation}
We will refer to any function constructed in this way as a {\bfseries $\Sigma_\bullet$-holonomy perturbation}.

The arguments of \cite[Section 3]{Kron} (applied to the linearizations of the $h_{\Sigma_i}$) show that any such $\Sigma_\bullet$-holonomy perturbation satisfies conditions (a) and (b) in Theorem \ref{kronsmagic}. In particular, any $\Sigma_\bullet$-holonomy perturbation is a compatible instanton perturbation, and hence of the form required by the proposition we are working to prove. (In fact, the perturbation relevant for the theorem will lie in a suitable Banach space of infinite linear combinations of $\Sigma_\bullet$-holonomy perturbations.)

Fix a $\Sigma_\bullet$-holonomy perturbation $H$. Since the holonomy takes values in the compact set $\SU(r)$, this function is uniformly bounded (each $h_{\Sigma_i}$ is bounded by the $r$ in $\PU(r)$). Likewise, the linearization $X(\afa) \in \Omega^2(Y, Q(\frak{g}))$ of $H$ is ${\mathcal{C}^0}(Y)$-bounded with a bound independent of $\afa$. For a proof, see the discussion leading up to Lemma 5.12 in \cite{Donfloer}; our $X(\afa)$ corresponds to Donaldson's $T_\afA$. The higher derivatives of $H$ are bounded in $W^{k, p}$-operator norm with bounds depending on $r$, the curvature of the connection $\afa$, and the number of derivatives taken; see \cite[Proposition 7]{Kron}.

Rather than working on the full space $\A^{k, p}(Q)$, we find it convenient below to work with the quotient space 
$$\smash{{\mathcal{B}} \defeq \A^{k, p}(Q) / \G^{k+1, p}_{\Sigma_\bullet}}.$$ 
In fact, by our assumptions on the bundle $Q$, most of our work will take place on the (open) subset ${\mathcal{B}}^* \subset {\mathcal{B}}$ of irreducible connections. Then ${\mathcal{B}}^*$ is a smooth Banach manifold. (Note that since we define ${\mathcal{B}}$ by quotienting only by the subgroup $\G_{\Sigma_\bullet}$, there is still a residual action of the finite group $\G_{\Sigma_\bullet}/ \G_0(Q)$ on this space. If we had quotiented by the full group $\G(Q)$, then ${\mathcal{B}}^*$ would have orbifold singularities.) Given any representative $\afa \in \left[ \afa \right]$, the tangent space $T_{\left[\afa\right]} {\mathcal{B}}^*$ can be canonically identified with
$$\Omega^1(Y, Q(\frak{g})) / \textrm{im} \: d_\afa.$$ 
The gauge invariance of any $\Sigma_\bullet$-holonomy perturbation $H$ implies that it descends to a well-defined smooth map $ {\mathcal{B}} \rightarrow \bb{R}$ on the quotient, and we denote the induced map by the same symbol $H$.

\medskip

Now that we have described the class of functions we are interested in, we are ready to prove the proposition. This will be carried out in several stages, each having the following form: Suppose $C \subset {\mathcal{B}}^*$ is compact, and assume each point has a neighborhood contained in a finite-dimensional submanifold. Fix a compact set $\underline{C} \subset T{\mathcal{B}}^* \vert_{C}$ consisting of vectors over $C$. Donaldson shows \cite[Section 5.5.1]{Donfloer} that there are a finite number of perturbations\footnote{Strictly speaking, the perturbations Donaldson uses differ from ours because they are defined by considering holonomy in all three directions in $Y$ and not just in the $\Sigma_\bullet$-directions. In Steps 1-3 below, we will discuss how to extend Donaldson's argument to our $\Sigma_\bullet$-holonomy perturbations.} $H^1, \ldots, H^K$ so that the map
\begin{equation}\label{donsembedding}
\left( H^1, \ldots, H^K\right): {\mathcal{B}}^*  \rightarrow  \bb{R}^K
\end{equation}
restricts to a $\underline{C}$-\emph{immersion} of $C$ in the sense that the linearization restricts to an injection from each fiber of $\underline{C}$ into $\bb{R}^K$. Donaldson's proof of this goes roughly as follows: Each element of ${\mathcal{B}}$ is determined up to $\G_{\Sigma_\bullet}/ \G_0(Q)$ by its holonomy. Likewise, each tangent vector at $\left[ \afa \right] \in {\mathcal{B}}$ is determined by its linearized holonomy (there is no $\G_{\Sigma_\bullet}/ \G_0(Q)$-ambiguity  at the tangent vector level since this is a discrete group). Since $\underline{C}$ is compact, only finitely many holonomies are needed to distinguish any two elements in a fiber of $\underline{C}$, which proves that (\ref{donsembedding}) is a $\underline{C}$-immersion for suitably chosen $H^1, \ldots, H^K$. We will use a variant of this construction repeatedly, first taking $C$ to be the space of flat connections, then coming from the space of $g_n$-instantons, and then from the $\underline{J}_0$-holomorphic curve representatives. It will then follow from a standard Sard-Smale argument that the relevant perturbed operators are surjective, which will verify (i) and (iv) in the statement of Proposition \ref{comppertprop}. To verify (ii) and (iii), we need to argue that, in each case, it suffices to take $H^1, \ldots, H^K$ in (\ref{donsembedding}) to be $\Sigma_\bullet$-holonomy perturbations, as opposed to the perturbations Donaldson considers. This latter part will follow from irreducibility, together with a unique continuation argument.

\medskip

\noindent \underline{Step 1}: \emph{There is some $\eps > 0$, and a positive-dimensional and finite-dimensional vector space $V$ consisting of $\Sigma_\bullet$-holonomy perturbations $H$, so that the following holds. Let $B_\eps(0) \subset V$ be the $\eps$-ball. Then for a dense, open subset of $H \in B_{\eps} (0)$, all $H$-flat connections are non-degenerate. }

\medskip

To prove this, take $C$ to be the image of the space flat connections in ${\mathcal{B}}^*$. Uhlenbeck's compactness theorem implies this is a compact set. We want to eliminate the degenerate flat connections. Note that a flat connection $\afa$ is non-degenerate if and only if $H^1_{\afa} \in T_{\left[ \afa\right]} {\mathcal{B}}^*$ is zero. We therefore consider the set $\underline{C}  \subset T{\mathcal{B}}^*$ over $C$ whose fiber over $\left[ \afa \right] \in C$ is the zero space if $\afa$ is non-degenerate, and the unit sphere in $H^1_{\afa}$ otherwise. Then $\underline{C}$ is compact as well. 

First we argue that $\Sigma_\bullet$-holonomy perturbations are sufficient for our purposes. Let $\left[ \afv \right] \in \underline{C}$ be non-zero. This can be represented by a vector $\afv \in \ker d_\afa \cap \ker d_\afa^*$ with unit norm, where $\afa$ is a degenerate flat connections. Since $\afv$ lies in the kernel of the elliptic operator $d_\afa \oplus d_\afa^*$, it satisfies a unique continuation property. In particular, it is uniquely determined by its value on $I \times \Sigma_\bullet \subset Y$. Write $\afv \vert_{I \times \Sigma_\bullet} = \afnu + \omega \: dt$ and $\afa \vert_{I \times \Sigma_\bullet}  = \afalpha + \psi \: dt$. Then since $\afv$ is in the kernel of $d_\afa$, we have
$$\nabla_t \afnu = d_\afalpha \omega.$$
Moreover, since $\afa$ is flat, it follows that $\afalpha$ is as well; in particular, $\afalpha$ is irreducible. It follows that $\omega$, and hence all of $\afv$ is uniquely determined by the value of $\afnu$ on $I \times \Sigma_\bullet$, which we view as a path from $I$ into the space of 1-forms on $\Sigma_\bullet$ with values in $\ker d_\afalpha^*$. Just as connections mod gauge are determined uniquely by their holonomy, 1-forms in the kernel of $d_\afa^*$ are distinguished by their linearized holonomy. Since $\afv$ is non-zero, the above observations imply that there is some $t \in (0, 1)$ for which $\afnu(t) \neq 0$. This means that we can find an $I$-dependent holonomy perturbation $h_{i, t}: \A(P_i) \rightarrow \bb{R}$ on some $\Sigma_i$, so that the linearization is such that $d_{\afalpha(t)} h_{i, t} (\afnu(t))$ is non-negative for all $t$, and $d_{\afalpha(t)}h_{i, t}(\afnu(t)) > 0$ for some $t$. Use this to construct a $\Sigma_\bullet$-holonomy perturbation $H: \A(Q) \rightarrow \bb{R}$ by the formula (\ref{deforhandshit1}). It follows that
$$d_{\afa}H(\afv) = \intdd{0}{1} d_{\afalpha(t)} H(\afnu(t)) \: dt > 0.$$
Repeating for each non-zero element of $\underline{C}$, and using compactness of $\underline{C}$, it follows that one can find a finite number of $\Sigma_\bullet$-holonomy perturbations $H^1, \ldots, H^K$ for which (\ref{donsembedding}) is a $\underline{C}$-embedding, as desired.  

To wrap up the proof of Step 1, take $V$ to be the span over $\bb{R}$ of the $H_k$. It now follows from a Sard-Smale argument \cite[Proposition 5.15]{Donfloer} that in any sufficiently small neighborhood of the origin in $V$, there is a comeager set of $\left( \eps_1, \ldots, \eps_K\right) \in \bb{R}^K$ such that all $H$-flat connections are non-degenerate, where we have set $H \defeq \sum_{k = 1}^{K} \: \eps_k H_k.$ That the set of such $H$ is open is a consequence of the compactness of $C$. This completes the proof of Step 1. Note that by Proposition \ref{nondegprop}, it follows that all $\underline{H}_t$-perturbed Lagrangian intersection points are non-degenerate as well.

\medskip

\noindent \underline{Step 2}: \emph{Suppose $H'$ is a $\Sigma_\bullet$-holonomy perturbation so that all $H'$-flat connections are non-degenerate. Fix $n \in \bb{N}$ and $E^\inst > 0$. Then there is some $\eps > 0$, and a positive-dimensional and finite-dimensional vector space $V$ consisting of $\Sigma_\bullet$-holonomy perturbations so that for a dense, open subset of $H'' \in B_\eps(0)$, the perturbation $H \defeq H'+ H''$ satisfies the following:}

\begin{itemize}
\item[\emph{(a)}] \emph{all $H$-flat connections are non-degenerate,}
\item[\emph{(b)}] \emph{all $(g_n, H)$-instantons with energy bounded by $E^\inst$ are regular.}
\end{itemize}

\medskip

Fix $H', n$ and $E^\inst$ as stated in Step 2. Since non-degeneracy is an open condition, by choosing $\eps$ to be sufficiently small, it follows that all $H$-flat connections are non-degenerate whenever $H = H' + H''$ and $H'' \in B_\eps(0) \subset V$. This takes care of (a). For simplicity, we may therefore assume $H' = 0$. 

To prove the rest of Step 2, follow the argument Donaldson gives, starting on p. 142 in \cite{Donfloer}. The only new point here is that we are working with the set of $\Sigma_\bullet$-holonomy perturbations, which are different than the set of perturbations considered by Donaldson. To see why this is sufficient, we extend the line of reasoning from Step 1 as follows: First, by unique continuation, it suffices to restrict attention to the end $\left[S, \infty \right) \times I \times \Sigma_\bullet$ for any $S$. Writing instantons here as $\afA = \afalpha + \afphi \: ds + \afpsi \: dt$, we may suppose that $S$ is large enough so that $\afalpha(s, t)$ is irreducible for all $(s, t) \in \left[S, \infty \right) \times I$. The game for the instanton case is to distinguish non-zero elements of $H^+_\afA$, where $\afA$ is an instanton. The elements of $H^+_\afA$ can be identified with the 2-forms $W$ on $\bb{R} \times Y$ with $*W = W$ and $d_\afA W = 0$. These 2-forms satisfy a unique continuation property, so it suffices to distinguish them on $\left[S, \infty \right) \times I \times \Sigma_\bullet$. The self-duality condition gives
$$W \vert_{\left[S, \infty \right) \times I \times \Sigma_\bullet} = \omega + ds \wedge \mu + dt \wedge *_\Sigma \mu + ds \wedge dt \: *_\Sigma \omega,$$
where $*_\Sigma$ is the Hodge star on $\Sigma_\bullet$, and $\omega$ (resp. $\mu$) is a map from $\left[ S, \infty \right) \times I$ into the space of 2-forms (resp. 1-forms) on $\Sigma_\bullet$. The condition $d_\afA W = 0$ implies, in particular, that $\omega, \mu$ are related by
$$ \nabla_s \mu + *_\Sigma \nabla_t \mu   + d_\afalpha^* \omega  = 0.$$
Since $\afalpha$ is irreducible, this shows that the values of $\mu$ on $\left[S, \infty \right) \times I \times \Sigma_\bullet$ determines $\omega$ uniquely, and hence these values determine $W$ uniquely. Just as in Step 1, we can find an $I$-dependent holonomy perturbation $h_{i, t}$ on some $\Sigma_i$ to distinguish $\mu$ from 0, provided $W \neq 0$; here we just work, as Donaldson does, at any $s_0 \geq S$. Now use this to construct a $\Sigma_\bullet$-holonomy perturbation $H^1$ that distinguishes the 1-form $\phi \defeq \iota_{\partial_s} W \vert_{s  = s_0}$ from zero. Similarly, find further $\Sigma_\bullet$-holonomy perturbations $H^2, \ldots, H^K$ such that the linearization of
$$(H^1, H^2, \ldots, H^K): {\mathcal{B}}^* \longrightarrow \bb{R}^K$$
sends $\phi$ and $\partial_s \afalpha $ to linearly independent vectors in $\bb{R}^K$. At this point, Donaldson's proof carries through without modification, starting at the bottom of p. 143, and this completes Step 2.

\medskip

\noindent \underline{Step 3}: \emph{Suppose $H'$ is a $\Sigma_\bullet$-holonomy perturbation so that all $H'$-flat connections are non-degenerate. Fix $E^\quilt > 0$. Then there is some $\eps > 0$, and a positive-dimensional and finite-dimensional vector space $V$ consisting of $\Sigma_\bullet$-holonomy perturbations so that for a dense, open subset of $H'' \in B_\eps(0)$, the perturbation $H \defeq H'+ H''$ satisfies the following:}

\begin{itemize}
\item[\emph{(a)}] \emph{all $\underline{H}_t$-Lagrangian intersection points are non-degenerate,}
\item[\emph{(b)}] \emph{all $(\underline{J}_0, \underline{H}_t)$-holomorphic strips with energy bounded by $E^\quilt$ are regular.}
\end{itemize} 
\emph{Here $\underline{H}_t$ is the split-type Hamiltonian associated to $H$. }

\medskip

Just as in the previous step, we will automatically have (a) provided $H''$ is small. The idea for obtaining the rest of the claims in Step 3 is to view each holomorphic strip in terms of representatives, i.e., connections on $\bb{R} \times Y$. This, in turn, can then be viewed as maps $\bb{R} \rightarrow {\mathcal{B}}^*$, and by unique continuation one could replace $\bb{R}$ with an interval of the form $\left[S, \infty\right)$. Here one wants to perturb so the elements of the cokernel of each linearized Cauchy-Riemann operator vanishes. Just as with the elements of $H^+_\afA$ in the previous step, any element in a cokernel can be viewed as a path $\phi$ from $\bb{R}$ into $T {\mathcal{B}}^*$; once again, unique continuation allows us to replace the domain of $\phi$ with $\left[ S, \infty \right)$, if we wish. Write $ \phi \vert_{I \times \Sigma_\bullet} = \mu + dt \wedge \eta$, where $\mu$ (resp. $\eta$) maps into the space of 1-forms (resp. 0-forms) on $\Sigma_\bullet$. Then irreducibility, together with the defining equations for the elements of the cokernel imply that the values of $\mu$ determine $\phi$ uniquely. This is analogous to what we saw with elements of $H^+_\afA$. Now the argument follows as in Step 2.

\medskip

To put this all together, let $H_0$ be as in Step 1. Now recursively define a sequence $\left\{H_k\right\}$ of $\Sigma_\bullet$-holonomy perturbations as follows: Fix sequences 
$$\left\{E_M^\inst\right\}_{M \in \bb{N}}, \indent \left\{E_N^\quilt\right\}_{N \in \bb{N}}$$ 
of real numbers increasing to $\infty$. Fix an enumeration
$$\varphi: \bb{N}_{\geq 1} \longrightarrow \left\{(g_n, E_M^\inst), E_N^\quilt \right\}_{n, M, N \in \bb{N}},$$ 
where $g_n$ are the metrics from the statement of the proposition. For each $\ell \geq 1$, fix a $\Sigma_\bullet$-holonomy perturbation $H_\ell$ as follows:

\begin{itemize}
\item If $\varphi(\ell) =  (g_n, E_N^\inst)$ for some $n, N$, then take define $H_\ell \defeq H''$, where $H''$ satisfies the conclusions of Step 2 with $H' \defeq H_{\ell-1}$ and $E^\inst \defeq E_N^\inst$. 

\item If $\varphi(\ell) =  E_M^\quilt$ for some $M$, then take define $H_\ell \defeq H''$, where $H''$ satisfies the conclusions of Step 3 with $H' \defeq H_{\ell-1}$ and $E^\quilt \defeq E_M^\quilt$.
\end{itemize}
At each stage, refine the choice of $H_\ell$ as follows: Given a $\Sigma_\bullet$-holonomy perturbation $H$ and non-negative integers $k_0 , \ell_0$, define
$$\Vert H \Vert_{k_0, \ell_0} \defeq \sumd{k' \leq k_0, \ell' \leq \ell_0} C_H(k', \ell')$$ 
where $C_H(k,\ell)$ is the minimum constant satisfying the conclusion of Lemma \ref{smallnesslemma}, below. By scaling $H$ by a small constant, the value $ \Vert H \Vert_{k_0, \ell_0}$ can be made as small as we wish. Then we refine the choice of $H_\ell$ by demanding that 
$$\Vert H_\ell \Vert_{\ell, \ell} \leq \mathrm{min}\: \left\{2^{-\ell}, \eps_\ell \right\},$$ 
where $\eps_\ell$ is one third of the distance from $H_{\ell-1}$ to the nearest instanton perturbation not satisfying the conditions (a) or (b) from the appropriate Step. This refinement is possible at each stage because of two reasons: (i) in each of Steps 2 and 3, each $\Vert \cdot \Vert_{k_0, \ell_0}$ restricts to a well-defined norm on the vector space $V$, and (ii) in each Step, the set of $H''$ satisfying (a) and (b) is open. 

Finally, consider the series $H = \sum_{\ell\in \bb{N}} \: H_\ell$. Then for each $k, p$, this series converges in $\C^\infty$ on bounded subsets of the space of maps $\A^{k,p}(Q) \rightarrow \bb{R}$. Moreover, $H$ satisfies the conclusions of the proposition.
\end{proof}

\begin{lemma}\label{smallnesslemma}
Let $H$ be any $\Sigma_\bullet$-holonomy perturbation as in (\ref{defofholo}). Fix a non-negative integer $\ell$. Then there is a constant $C_H(0, \ell)$ so that
\begin{equation}\label{lemma1}
\begin{array}{lcl}
\Vert d^\ell X_\afa (v_1, v_2, \ldots, v_\ell) \Vert_{L^{p}(Y)} \\
\indent \indent \leq C_H(0, \ell) \Vert v_1 \Vert_{L^{p}(Y)} \Vert v_2 \Vert_{L^{p}(Y)}\ldots \Vert v_\ell \Vert_{L^{p}(Y)},
\end{array}
\end{equation}
for all connections $\afa \in \A(Q)$, 1-forms $v_1, \ldots, v_\ell \in \Omega^1(Y, Q(\frak{g}))$, and $1 \leq p \leq \infty$. Here $d^\ell X_\afa: \Omega^1 \times \Omega^1 \times \ldots \times \Omega^1 \rightarrow \Omega^2$ is the $\ell$th derivative of $\afa \mapsto X(\afa)$ with respect to, say, the $W^{1,2}$-topology. When $\ell = 0$, this says that $X(\afa)$ is uniformly bounded in $L^p$.

More generally, if $k$ is a positive integer, then there is a constant $C_H(k, \ell)$ so that
\begin{equation}\label{lemma2}
\begin{array}{lcl}
\Vert d^\ell X_\afa (v_1, v_2, \ldots, v_\ell) \Vert_{W^{k,p}(Y)} \\
\indent \indent \leq C_H(k, \ell) \left(1 + \Vert F_\afa \Vert_{W^{k-1,p}(Y)}^{k-1}  \right) \Vert v_1 \Vert_{W^{k,p}(Y)} \Vert v_2 \Vert_{W^{k,p}(Y)} \ldots \Vert v_\ell \Vert_{W^{k,p}(Y)}.
\end{array}
\end{equation}
\end{lemma}

\begin{proof}
The $L^p$-estimates were proved in \cite[Lemma 8]{Kron}, and reduce to the fact that the holonomy maps into the compact set $\SU(r)$. In fact, Kronheimer shows that for each positive integer $\ell$, the map $H$ is bounded in $\C^{\ell}$ as a map $\A(Q) \rightarrow \bb{R}$ with the $L^p$-topology on $A(Q)$. It therefore suffices to verify the $W^{k,p}$-estimates for $k \geq 1$. We will focus on the $W^{1,p}$-estimate for $ X(\afa)$; the ones for the derivatives $d^\ell X_\afa$ are similar. 

Recall that $H$ is constructed from finitely many functions $h = h_{\Sigma_i}: \A(P_i) \rightarrow \bb{R}$ of the form (\ref{holpertbounds}). To establish the desired bound for $X_\afa$, it suffices to show there is some constant $C$ so that 
$$\Vert x(\afalpha) \Vert_{W^{1,p}(\Sigma_i)}  \leq   C(1 + \Vert F_\afalpha \Vert_{L^p(\Sigma_i)} )$$
for all connections $\afalpha \in \A(P_i)$; here $x: \A(P_i) \rightarrow \Omega^1(\Sigma_i, P_i(\frak{g}))$ is the Hamiltonian vector field for $h$. Throughout we will liberally refer to the notation $\gamma_j$, etc., leading up to (\ref{holpertbounds}). 

For simplicity, we assume $J = 1$ and so the word $W: \SU(r) \rightarrow \SU(r)$ is just the identity map. This gives
$$h(\afalpha) = \intd{\bb{D}^1} \: \mathrm{Tr}(\rho_\tau (\afalpha)) \: \beta_0 \: d \tau .$$
This implies that $x(\afalpha)$ a 1-form supported in the image of the thickened loop $\pi \circ \gamma_1: \bb{R} / \bb{Z} \times \bb{D}^1 \rightarrow \Sigma_i$. By trivializing the bundle over this image and pulling back, we find it convenient to think of $x(\afalpha)$ as a $\frak{g}$-valued 1-form on $\bb{R} / \bb{Z} \times \bb{D}^1$, in which case we have 
\begin{equation}\label{xproj}
x(\afalpha) = \proj \: \rho_\tau (\afalpha) \: \beta_0 \: d\tau,
\end{equation}
where $\proj$ is the orthogonal projection from the Lie group $\SU(r)$ to its Lie algebra, as in \cite[Section 3(i)]{Kron} and \cite[p. 135]{Donfloer}. This is covariantly constant in the $\bb{R} / \bb{Z}$-direction since this is the direction where the holonomy occurs. It follows that to estimate the $W^{1,p}$-norm of $x(\afalpha)$ we need only estimate the $L^p$-norm of
$$\partial_\tau x(\afalpha),$$
where $\tau$ is the variable in $\bb{D}^1$, as above. 

By (\ref{xproj}), it suffices to bound the $L^p$ norm of $\partial_\tau \rho_\tau (\afalpha)$. Let $\afalpha(\tau)$ be the restriction of $\afalpha$ to $\bb{R} / \bb{Z} \times \left\{ \tau \right\}$, which we can view as a $\tau$-dependent connection on $\bb{R} / \bb{Z} \times \left\{0 \right\}$. Then, up to an overall constant depending only on the derivative $\partial_\tau \gamma_1$, the norm of $\partial_\tau \rho_\tau (\afalpha) \in \frak{g}$ is controlled by the norm of $\partial_\tau \mathrm{hol}(\afalpha(\tau)) \in \frak{g}$, where $\mathrm{hol}$ is the holonomy around $\bb{R} / \bb{Z} \times \left\{0 \right\}$. Letting $d \mathrm{hol}_\afalpha$ denote the linearization at $\afalpha$ of the holonomy map, we have
$$\begin{array}{rcl}
\partial_\tau \mathrm{hol}(\afalpha(\tau)) & = & d \mathrm{hol}_{\afalpha(\tau)} (\partial_\tau \afalpha(\tau)).
\end{array}$$
Now $\partial_\tau \afalpha$ is, up to gauge, the $d\tau$-component of $F_\afalpha$. Combining this with the uniform $L^p$ bounds from \cite[Lemma 8]{Kron}, we have
$$\vert \partial_\tau \mathrm{hol}(\afalpha(\tau))  \vert \leq  C \Vert F_\afalpha \Vert_{L^p(\bb{R} / \bb{Z} \times \left\{\tau \right\})}.$$
Raising both sides to the power of $p$, and integrating $\tau$ over $\bb{D}^1$ gives
$$\Vert \partial_\tau x(\afalpha) \Vert_{L^p(\Sigma_i)}  = \Vert \partial_\tau x(\afalpha) \Vert_{L^p(\bb{R} / \bb{Z} \times \bb{D}^1)} \leq C \Vert F_\afalpha \Vert_{L^p(\bb{R} / \bb{Z} \times \bb{D}^1)},$$
as desired.

\end{proof}

\end{document}